\documentclass[reqno,twoside,11pt]{article}
\usepackage[latin1]{inputenc}
\usepackage[T1]{fontenc}
\usepackage[english]{babel}
\usepackage{color}
\usepackage{tocloft}
\usepackage{verbatim}

\usepackage[
  hmarginratio={1:1},     % equal left and right margins
  vmarginratio={1:1},     % equal top and bottom margins
  textwidth=17cm,        % new text width
  textheight=24cm,
  heightrounded          % always useful
]{geometry}

\usepackage[noindentafter]{titlesec}
\titlespacing{\paragraph}{%
  0pt}{%              left margin
0.25\baselineskip}{% space before (vertical)
1em}%               space after (horizontal)

\titlespacing{\section}{%
0pt}{%              left margin
0.2cm}{% space before (vertical)
0em}%               space after (horizontal)

\titlespacing{\subsection}{%
0pt}{%              left margin
0.2cm}{% space before (vertical)
0em}%               space after (horizontal)

\titlespacing{\subsubsection}{%
0pt}{%              left margin
0cm}{% space before (vertical)
0em}%               space after (horizontal)

\usepackage{amsmath,amsthm,amssymb,accents}
\usepackage{mathrsfs, mathtools,xpatch}
\usepackage{dsfont}

\DeclareMathAlphabet{\mathbbe}{U}{bbold}{m}{n}

\mathtoolsset{showonlyrefs}

\allowdisplaybreaks

\usepackage[all]{xy} 
\input xy 
\xyoption{all}
\xyoption{2cell} 
\xyoption{v2}
\SelectTips{cm}{11} % Makes all xy arrows match the usuadifferl LaTeX arrows 

\usepackage{tikz}
\usetikzlibrary{snakes}
\usetikzlibrary{arrows}
\usepackage{tikz-cd}

\tikzset{string/.style={decorate, decoration={snake, segment length=3pt, amplitude=1pt}}}

\usepackage[hidelinks]{hyperref}

\usepackage[numbers,sort&compress]{natbib}
\makeatletter
\renewcommand{\@biblabel}[1]{[#1]\hfill}
\makeatother
\usepackage[mathscr]{euscript}
\DeclareMathAlphabet{\pazocal}{OMS}{zplm}{m}{n}

\usepackage[shortlabels]{enumitem}
\newtheoremstyle{thm}                                                           % Name
{0.15cm}                                         % Space above
{0.15cm}                                         % Space below
{\itshape}      % Body font
{}                                      % Indent amount
{\bfseries}                             % Theorem head font
{}                                      % Punctuation after theorem head
{0.2cm}                                         % Space after theorem head
{\thmname{#1}~\thmnumber{#2}\thmnote{ (#3)}}%

\makeatletter
\xpatchcmd{\proof}{\topsep6\p@\@plus6\p@\relax}{}{}{}
\makeatother

\newtheoremstyle{rmk}                                                           % Name
{0.15cm}                                         % Space above
{0.15cm}                                         % Space below
{}      % Body font
{}                                      % Indent amount
{\bfseries}                             % Theorem head font
{}                                      % Punctuation after theorem head
{0.2cm}                                         % Space after theorem head
{}                                      % Theorem head spec (can be left empty, meaning 'normal')

\theoremstyle{thm}
\newtheorem{theorem}[equation]{Theorem}

\newtheorem{corollary}[equation]{Corollary}

\newtheorem{lemma}[equation]{Lemma}
\newtheorem{proposition}[equation]{Proposition}
\newtheorem{definition}[equation]{Definition}

\theoremstyle{rmk}
\newtheorem{example}[equation]{Example}
\newtheorem{remark}[equation]{Remark}

\numberwithin{equation}{section}

\makeatletter
\newlength{\@thlabel@width}%
\newcommand{\thmenumhspace}{\settowidth{\@thlabel@width}{(1)}\sbox{\@labels}{\unhbox\@labels\hspace{\dimexpr-\leftmargin+\labelsep+\@thlabel@width-\itemindent}}}
\makeatother
\usepackage[todo]{optional}

\newcommand{\inv}{{\mathrm{inv}}}
\newcommand{\pr}{\mathrm{pr}}

\newcommand{\Sing}{\operatorname{Sing}}

\newcommand{\Spin}{{\mathrm{Spin}}}
\newcommand{\String}{{\mathrm{String}}}

\newcommand{\Dec}{\operatorname{Dec}}
\newcommand{\obj}{{\operatorname{obj}}}

\newcommand{\rmH}{\mathrm{H}}
\newcommand{\rmU}{{\mathrm{U}}}

\newcommand{\rmB}{{\mathrm{B}}}

\newcommand{\Sym}{{\mathrm{Sym}}}

\newcommand{\ev}{\operatorname{ev}}

\newcommand{\rmS}{\mathrm{S}}

\newcommand{\opp}{\mathrm{op}}

\newcommand{\Ran}{\mathrm{Ran}}

\newcommand{\colim}{{\mathrm{colim}}}

\newcommand{\rmL}{{\mathrm{L}}}
\newcommand{\rmR}{{\mathrm{R}}}
\newcommand{\rmF}{{\mathrm{F}}}
\newcommand{\Loc}{\operatorname{Loc}}

\renewcommand{\lim}{{\mathrm{lim}}}

\newcommand{\scI}{\mathscr{I}}

\newcommand{\scH}{\mathscr{H}}

\newcommand{\scP}{\mathscr{P}}
\newcommand{\scC}{\mathscr{C}}
\newcommand{\scD}{\mathscr{D}}

\newcommand{\scU}{{\mathscr{U}}}
\newcommand{\scV}{{\mathscr{V}}}

\newcommand{\scM}{{\mathscr{M}}}

\newcommand{\CG}{\mathcal{G}}

\newcommand{\CA}{\mathcal{A}}

\newcommand{\bH}{{\mathbf{H}}}

\newcommand{\bS}{{\mathbf{S}}}

\newcommand{\sfL}{{\mathsf{L}}}

\newcommand{\sfU}{\mathsf{U}}

\newcommand{\sfO}{\mathsf{O}}

\newcommand{\sfc}{\mathsf{c}}
\newcommand{\sfS}{\mathsf{S}}

\newcommand{\NN}{\mathbb{N}}
\newcommand{\RN}{\mathbb{R}}
\newcommand{\ZN}{\mathbb{Z}}

\newcommand{\bbE}{\mathbb{E}}

\newcommand{\Des}{\mathfrak{Des}}

\newcommand{\ul}[1]{\underline{#1}}

\newcommand{\cC}{{\check{C}}}

\newcommand{\HLB}{{\mathrm{HLB}}}
\newcommand{\Ab}{{\mathscr{A}\mathrm{b}}}
\newcommand{\Grb}{{\mathscr{G}\mathrm{rb}}}

\newcommand{\Set}{{\mathscr{S}\mathrm{et}}}
\newcommand{\sSet}{{\mathscr{S}\mathrm{et}_{\hspace{-.03cm}\Delta}}}
\newcommand{\Fun}{{\mathscr{F}\hspace{-.03cm}\mathrm{un}}}

\newcommand{\Cat}{{\mathscr{C}\mathrm{at}}}

\newcommand{\Mfd}{{\mathscr{M}\mathrm{fd}}}

\newcommand{\Cart}{{\mathscr{C}\mathrm{art}}}

\newcommand{\Gpd}{{\mathscr{G}\mathrm{pd}}}
\newcommand{\Grp}{{\mathscr{G}\hspace{-0.02cm}\mathrm{rp}}}
\newcommand{\EEpi}{{\mathrm{EEpi}}}

\newcommand{\Op}{{\mathscr{O}\mathrm{p}}}
\newcommand{\Bun}{{\mathscr{B}\hspace{-0.02cm}\mathrm{un}}}

\newcommand{\eq}{\overset{\simeq}{\longrightarrow}}
\newcommand{\wc}{{\widetilde{\mathsf{c}}}}

\newenvironment{myenumerate}{\begin{enumerate}[topsep=-\parskip+0.1cm, itemsep=0.1cm, parsep=0cm, leftmargin=*, label=(\arabic*)]}{\end{enumerate}}

\newcommand{\qen}{\hfill$\triangleleft$}

\newcommand{\dslash}{{/\hspace{-0.1cm}/}}

\newcommand{\bbDelta}{{\mathbbe{\Delta}}}

\newcommand{\textint}{\textstyle{\int}}

\newlength{\Displayskip}
\begin{document}

\setlength{\abovedisplayskip}{\Displayskip}
\setlength{\belowdisplayskip}{\Displayskip}

\begin{flushright}
\small
\textsf{Hamburger Beiträge zur Mathematik Nr.\,858}\\
{\sf ZMP--HH/20-14} 
\end{flushright}

%\vspace{0.25cm}

\begin{center}
\LARGE{\textbf{Principal $\infty$-Bundles and Smooth String Group Models}}
\end{center}
\begin{center}
\large Severin Bunk
\end{center}

\begin{abstract}
\noindent
We provide a general, homotopy-theoretic definition of string group models within an $\infty$-category of smooth spaces and present new smooth models for the string group.
Here, a smooth space is a presheaf of $\infty$-groupoids on the category of cartesian spaces.
The key to our definition and construction of smooth string group models is a version of the singular complex functor, which assigns to a smooth space an underlying ordinary space.
We provide new characterisations of principal $\infty$-bundles and group extensions in $\infty$-topoi, building on work of Nikolaus, Schreiber and Stevenson.
These insights allow us to transfer the definition of string group extensions from the $\infty$-category of spaces to the $\infty$-category of smooth spaces.
Finally, we consider smooth higher-categorical group extensions that arise as obstructions to the existence of equivariant structures on gerbes.
These extensions give rise to new smooth models for the string group, as recently conjectured in joint work with Müller and Szabo.
\end{abstract}

\tableofcontents

\section{Introduction and overview}
\label{sec:Introduction}

The most direct way to define the string group is via the Whitehead tower of $\sfO(n)$,
\begin{equation}
\label{eq:Whitehead tower of O(n)}
	\cdots \longrightarrow \String(n) \longrightarrow \Spin(n) \longrightarrow \sfS\sfO(n) \longrightarrow \sfO(n)\,.
\end{equation}
By this approach, $\String(n)$ is defined as a 3-connected topological space with a continuous map $\String(n) \to \Spin(n)$ which induces an isomorphism on all homotopy groups except for in degree three.
So far, this defines $\String(n)$ only as a space, but in~\cite{Stolz:Pos_Ric} Stolz constructed $\String(n)$ as a topological group and the map $\String(n) \to \Spin(n)$ as a morphism of topological groups.
In fact, he presented a construction that produces, for any compact, simple, and simply connected Lie group $H$, a morphism $\String(H) \to H$ of topological groups whose underlying continuous map is a three-connected covering.
A covering of this type is also called a \emph{string group extension of $H$}.
In these conventions, we write $\String(n) \coloneqq \String(\Spin(n))$.

The string group is important in geometry and topology in several ways.
Originally, Killingback~\cite{Killingback:WS_Anomalies} and Witten~\cite{Witten:Dirac_on_Loop_Space} investigated the two-dimensional supersymmetric $\sigma$-model on background manifolds $M$ and found that this is well-defined only if the free loop space $LM$ admits a spin structure.
Witten, moreover, computed the index of a hypothetical Dirac operator on $LM$ based on physical arguments, leading to the definition of the Witten genus.
By now, it has been understood that the Witten genus is related to the cohomology theory of topological modular forms (TMF).
The string group enters in this story, for example by defining orientations in TMF~\cite{AHR:TMF_Orientations,DHH:Obstr_to_String_Or}, analogously to how the spin group underlies orientations in real K-theory.

Since the free loop space $LM$ is less tractable than the manifold $M$ itself, it is an important question whether the condition that $LM$ admit a spin structure can be recast as a condition on the manifold $M$ itself.
This is indeed the case:
\emph{spin} structures on $LM$ correspond to \emph{string} structures on $M$~\cite{ST:Spinors_on_Loops,ST:What_is_an_elliptic_object,Waldorf:String_v_Spin}.
Topologically, a string structure on $M$ is a lift of the classifying map $M \to \rmB \sfO(n)$ of the tangent bundle $TM \to M$ to a map $M \to \rmB \String(n)$.
That is, a string structure is a reduction of the structure group of $TM$ to $\String(n)$.
From a geometric perspective, the interest ultimately is in identifying consequences and constructions that are facilitated by a string structure on a manifold.
Concrete examples include the Höhn-Stolz conjecture~\cite{Hoehn:Dpl_Thesis,Stolz:Pos_Ric} that the Witten genus is trivial for any Riemannian $4k$-manifold with positive Ricci curvature which admits a string structure, or the long-standing goal to define a Dirac operator on the loop space $LM$.

In order to study the differential geometric, rather than topological, implications of string structures, it is paramount to have models for $\String(n)$ not just as a topological group, but as a group object in some geometric category.
For instance, given a Riemannian manifold $M$, the construction of the Dirac operator associated with a spin structure on $M$ depends on the ability to glue the tangent bundle $TM$ from \emph{smooth} $\Spin(n)$-valued functions.
Technically, one also needs to find local frames for $TM$ in which the Levi-Civita connection of $M$ is represented by 1-forms valued in the Lie algebra $\mathfrak{spin}(n)$ rather than $\mathfrak{so}(n)$; however, since the fibre of the map $\Spin(n) \to \sfS\sfO(n)$ is discrete, these Lie algebras happen to be canonically isomorphic (for more background on spin geometry and Dirac operators, see, for instance,~\cite{LM:Spin_Geometry}).
Analogously to how spin structures on $LM$ stem from string structures on $M$, a hypothetical Dirac operator on $LM$ may well stem from a geometric operator on $M$ itself (e.g.~via some transgression procedure), obtained from a further lift of the Levi-Civita connection to the Lie algebra $\mathfrak{string}(n)$.
However, for this to make sense, one must work with a \emph{smooth}, rather than topological, model for $\String(n)$.

Classical results on cohomology readily imply that it is impossible to construct $\String(H)$ as a finite-dimensional Lie group (for any compact, simple, simply connected Lie group $H$).
Thus, to find geometric models for $\String(H)$, one needs to look beyond the category of smooth, finite-dimensional manifolds.
Indeed, a number of models for $\String(H)$ have been found in (higher) categories of smooth spaces that generalise the notion of a manifold in various ways~\cite{BCSS:String_via_Loops,Henriques:Integrating_L_oo,SP:String,Waldorf:String_via_TR,NSW:Smooth_String_Model,FRS:U1-Gerbe_connections}.

In each of these constructions, an extension
\begin{equation}
	A \longrightarrow \String(H) \longrightarrow H
\end{equation}
of a compact, simple, simply connected Lie group $H$ is constructed within the chosen ambient category of smooth spaces.
It is then argued that on the underlying ordinary spaces (meaning topological spaces or simplicial sets) one obtains a string group extension in the sense of~\eqref{eq:Whitehead tower of O(n)}.
However, so far there is no general definition of $\String(H)$ in a smooth context that formalises this procedure.
Consequently, in geometric models for $\String(H)$ the extending group $A$ currently has to be chosen ad hoc as an explicit delooping of the Lie group $\sfU(1)$ in a rather strict sense.
This obscures the homotopy-theoretic nature of $\String(H)$, since from a homotopical point of view, not $A$ is fixed, but only its homotopy type.

In~\cite{BMS:Sm2Grp}, studying symmetries of gerbes, we came across extensions of Lie groups $H$ not by a delooping of the Lie group $\sfU(1)$, but by the delooping of the diffeological group $\sfU(1)^H$ of smooth maps from $H$ to $\sfU(1)$.
However, if $H$ is simply connected, then the smooth group $\sfU(1)^H$ is homotopy equivalent to $\sfU(1)$.
Therefore, extensions of $H$ by the delooping $\rmB (\sfU(1)^H)$ potentially have the correct homotopy type to produce smooth string group extensions of $H$.
Nevertheless, we could not make this rigorous due to the lack of a homotopy-theoretic notion of smooth string group extensions that does not fix the extending group, but only its homotopy type.

Here, we provide such a general definition of smooth string group extensions, and we prove that the string group models proposed in~\cite{BMS:Sm2Grp} fit within this definition.
Let $\Mfd$ denote the category of manifolds and smooth maps, and let $\Cart \subset \Mfd$ be the full subcategory on those manifolds that are diffeomorphic to $\RN^n$ for any $n \in \NN_0$.
We denote the $\infty$-category of spaces by $\bS$.
As our ambient \textit{$\infty$-category of smooth spaces}, we choose the $\infty$-category $\bH_\infty \coloneqq \Fun(\Cart^\opp, \bS)$ of presheaves of spaces on $\Cart$.
This provides a very general notion of smooth space: for instance, $\bH_\infty$ contains the categories of manifolds, diffeological spaces, and Lie groupoids.
We write $\ul{M}$ for the image of a manifold $M$ under the fully faithful inclusion $\Mfd \hookrightarrow \bH_\infty$.

The $\infty$-category $\bH_\infty$ is even an $\infty$-topos.
There exists an established theory of group objects in $\infty$-topoi~\cite{Lurie:HTT}.
Moreover, there exists a notion of principal $\infty$-bundles and extensions of group objects in $\infty$-topoi, due to~\cite{NSS:oo-bundles}.
A large part of this paper is devoted to developing this theory further.
In particular, we show that group actions in $\infty$-topoi automatically form groupoid objects (Theorem~\ref{st:grp actions are gpd obs}) and that principal $\infty$-bundles essentially consists of an effective epimorphism and a principal group action (Theorem~\ref{st:pfbun characterisation}); this is analogous to the definition of principal bundles of topological spaces as a locally trivial map and a principal group action.
A group object in an $\infty$-topos $\bH$ is a simplicial object $A \in \Fun(\bbDelta^\opp, \bH)$ satisfying certain properties (see Definitions~\ref{def:Gpd objs} and~\ref{def:Grp(H)}).
We provide the following characterisation of extensions of group objects:

\begin{theorem}
\label{st:intro oo-grp ext via pfbuns}
Let $\bH$ be an $\infty$-topos.
Let \smash{$A \xrightarrow{\iota} G \xrightarrow{p} H$} be a sequence of morphisms of group objects in $\bH$.
The following are equivalent:
\begin{myenumerate}
\item \smash{$A \xrightarrow{\iota} G \xrightarrow{p} H$} is an extension of group objects in $\bH$ in the sense of~\cite{NSS:oo-bundles}, i.e.~the sequence $\rmB A \to \rmB G \to \rmB H$ is a fibre sequence in $\bH$.

\item \smash{$A \xrightarrow{\iota} G \xrightarrow{p} H$} is a fibre sequence of group objects in $\bH$ and $p_1 \colon G_1 \to H_1$ is an effective epimorphism (the subscript $1$ denotes evaluation at $[1] \in \bbDelta$).

\item $A_1 \xrightarrow{\iota_1} G_1 \xrightarrow{p_1} H_1$ is a fibre sequence in $\bH$ and $p_1 \colon G_1 \to H_1$ is an effective epimorphism.

\item The morphism $p_1 \colon G_1 \to H_1$ together with the action of $A$ on $G_1$ induced by $\iota$ define a principal $A$-bundle over $H_1$.
\end{myenumerate}
\end{theorem}

In order to give a general homotopy-theoretic definition of string group extensions within $\bH_\infty$, we need to associate an underlying space to an object in $\bH_\infty$.
In~\cite{Bun:Sm_Spaces} we investigated (a model categorical presentation of) a functor $\rmS_e \colon \bH_\infty \to \bS$ from $\bH_\infty$ to the $\infty$-category $\bS$ of spaces.
It evaluates a smooth space $B \in \bH_\infty$ on the extended affine simplices $\Delta_e^k \in \Cart$ and then takes the geometric realisation of the resulting simplicial object in $\bS$.
One can think of $\rmS_e$ is a version of the singular complex functor for smooth spaces.
Here, we give further interpretation and context to this functor.
Consider the adjunction $\wc \dashv \Gamma$, where $\Gamma \colon \bH_\infty \to \bS$ is the global-section functor and $\wc$ is the constant-presheaf functor.
This fits into a triple adjunction $\Pi \dashv \wc \dashv \Gamma \dashv codisc$, where $codisc$ is fully faithful and where $\Pi$ preserves finite products.
That is, the $\infty$-topos $\bH_\infty$ is cohesive.

\begin{theorem}
\label{st:intro S_e and cohesion}
The functor $\rmS_e \colon \bH_\infty \to \bS$ is part of the cohesion of $\bH_\infty$:
there is a canonical equivalence
\begin{equation}
	\Pi \simeq \rmS_e\,.
\end{equation}
\end{theorem}

This has already been argued in~\cite{BEBdBP:Class_sp_of_oo-sheaves} and proven on the level of model categories of simplicial presheaves in~\cite{Bun:Sm_Spaces}; here we provide an $\infty$-categorical proof based on findings from~\cite{Bun:Sm_Spaces}.

Let $\rmL \colon \bH \to \bH'$ be a functor between $\infty$-topoi which preserves finite products and geometric realisations of simplicial objects.
We show that $\rmL$ maps principal $\infty$-bundles in $\bH$ to principal $\infty$-bundles in $\bH'$ and group extensions in $\bH$ to group extensions in $\bH'$ (this relies on Theorem~\ref{st:grp actions are gpd obs}).
In particular, the functor $\rmS_e \colon \bH_\infty \to \bS$ has these properties.
In $\bS$, a string group extension of a compact, simple, simply connected Lie group $H$ can be defined as usual:
it is an extension $A \to \String(H) \to H$ of group objects in $\bS$ such that $\String(H)$ is 3-connected and such that the morphism $\String(H) \to H$ induces an isomorphism on all homotopy groups of the underlying spaces except for in degree three.
Using that $\rmS_e \ul{M} \simeq M$ for any manifold $M$ (see~\cite{Bun:Sm_Spaces} for a proof of this classical fact using the present technology) and that $\rmS_e$ preserves principal $\infty$-bundles and group extensions, we can now transfer this definition to $\bH_\infty$:

\begin{definition}
\label{def:intro String(H) in H_oo}
Let $H$ be a compact, simple, and simply connected Lie group, and let $\ul{H}$ denote the induced group object in $\bH_\infty$.
An extension of the group object $\ul{H}$ in $\bH_\infty$ is called a \emph{smooth string group extension of $H$} if its image under $\rmS_e$ is a string group extension in $\bS$.
\end{definition}

We show that the string group models conjectured in~\cite{BMS:Sm2Grp} fit within Definition~\ref{def:intro String(H) in H_oo}.
Let $M$ be a manifold endowed with a bundle gerbe $\CG$ (a categorified hermitean line bundle).
In~\cite{BMS:Sm2Grp}, we addressed the question of when an action of a Lie group $H$ on $M$ lifts to an equivariant structure on $\CG$.
We found that the obstruction to such a lift is captured by an extension
\begin{equation}
\label{eq:intro Sym(G) ext general}
	\HLB^M \overset{i}{\longrightarrow} \Sym(\CG) \overset{p}{\longrightarrow} H
\end{equation}
of $H$ by the smooth 2-group $\HLB^M$ of hermitean line bundles on $M$.
Each of the above objects can be interpreted as a group object in $\bH_\infty$ via the nerve functor $N$, and so the sequence~\eqref{eq:intro Sym(G) ext general} enhances to an extension
\begin{equation}
\label{eq:intro NSym(G) ext general}
	N \big( \HLB^M \big)  \overset{Ni}{\longrightarrow} N \big( \Sym(\CG) \big) \overset{Np}{\longrightarrow} \ul{H}
\end{equation}
of $\ul{H}$ as a group object in $\bH_\infty$.
The case relevant for string group extensions is $M=H$, where $H$ is a compact, simple and simply connected Lie group, acting on itself via left multiplication.
Since $H$ is 2-connected, there is an objectwise equivalence $\HLB^H \simeq \rmB(\sfU(1)^H)$, and since $H$ is 1-connected, there is a smooth homotopy equivalence $\sfU(1)^H \simeq \sfU(1)$.
Therefore, the extending group in~\eqref{eq:intro Sym(G) ext general} has the correct homotopy type for a string group extension.
We prove:

\begin{theorem}
\label{st:intro String model from Sym(G)}
Let $H$ be a compact, simple, simply connected Lie group, and let $N$ be the nerve functor.
Consider the left-action of $H$ on itself via left multiplication.
Let $\CG \in \Grb(H)$ be a gerbe on $H$ whose class in $\rmH^3(H;\ZN) \cong \ZN$ is a generator.
The sequence
\begin{equation}
\label{eq:intro Sym_G String sequence}
\begin{tikzcd}
	N \big( \HLB^H \big) \ar[r, "Ni"]
	& N \big( \Sym(\CG) \big) \ar[r, "Np"]
	& \ul{H}
\end{tikzcd}
\end{equation}
is a smooth string group extension of $H$.
\end{theorem}

This string group model is somewhat similar to the model in~\cite{FRS:U1-Gerbe_connections}, which is obtained by studying symmetries of gerbes with connection.
However, here the presence of connections forces the extending group to be the delooping $\rmB \sfU(1)$.
It is interesting that the connection does not change the homotopy type of the extension.
In~\cite[Def.~5.33]{BMS:Sm2Grp}, we also constructed a second extension of $H$ with a connection on the gerbe $\CG$ acting as crucial auxiliary data.
We showed that this extension is equivalent to the one in~\eqref{eq:intro Sym(G) ext general}~\cite[Thm.~5.36]{BMS:Sm2Grp}, it gives rise to a second smooth string group extension of $H$.

Finally, we expect that most (or possibly all) of the aforementioned smooth string group models fit within Definition~\ref{def:intro String(H) in H_oo}.
Checking this in full detail in each case would go beyond the scope of this article, but we outline the relevant arguments here:
for the models in~\cite{FRS:U1-Gerbe_connections} and~\cite{Waldorf:String_via_TR} the methods we use here should adapt in a straightforward manner.
The models~\cite{Henriques:Integrating_L_oo, SP:String} should fit into the present framework via the presentation of sheaves of $\infty$-groupoids on $\Cart$ as $\infty$-Lie groupoids (see, for instance,~\cite{Nuiten:Stacks_and_fractions, Nuiten:Hypercovers, Schreiber:DCCT}).
For the infinite-dimensional models in~\cite{Stolz:Pos_Ric} and~\cite{BCSS:String_via_Loops}, one needs to be able to compute homotopy types of infinite-dimensional manifolds, such as gauge groups and loop spaces, via the functor $\rmS_e$.
In the case of loop spaces, this is facilitated by the Smooth Oka Principle from~\cite{SS:Equivar_oo-Bundles}.
For gauge groups one needs that the homotopy type induced from the infinite-dimensional manifold structure agrees with the homotopy type extracted by using $\rmS_e$.
This should follow from general results on the relation between diffeological spaces and infinite-dimensional manifolds in~\cite{Kihara:smooth_oo-dim_Mfds}.

%\vspace{-0.5cm}
\paragraph*{Outline.}

In Section~\ref{sec:Sm spaces and oo-topoi} we investigate the functor $\rmS_e \colon \bH_\infty \to \bS$.
Further, we recall some basic notions and facts about $\infty$-topoi and prove Theorem~\ref{st:intro S_e and cohesion}.

Section~\ref{sec:PBuns and Grp Exts in oo-topoi} is devoted to the theory of group objects, group extensions, and principal $\infty$-bundles in $\infty$-topoi.
We recall the definitions of these notions from~\cite{NSS:oo-bundles} and provide new characterisations of principal $\infty$-bundles and group extensions.
In particular, we prove Theorem~\ref{st:intro oo-grp ext via pfbuns}.

In Section~\ref{sec:HoThy Smooth String}, we use the results obtained thus far to transfer the definition of string group extensions in $\bS$ to the $\infty$-topos $\bH_\infty$.
After recalling from~\cite{BMS:Sm2Grp} the smooth 2-group extensions which control equivariant structures on gerbes, we show that these extensions give rise to new smooth models for string group, thus proving Theorem~\ref{st:intro String model from Sym(G)}.

Finally, in Appendix~\ref{app:Actions and CatObs} we prove Theorem~\ref{st:grp actions are gpd obs}:
we show that group actions in an $\infty$-topos give rise to groupoid objects.

%\vspace{-0.5cm}
\paragraph*{Notation.}

We usually make no notational distinction between ordinary categories and $\infty$-categories; the nerve functor will be used implicitly where necessary.

We write $\bbDelta$ for the simplex category, and $\sSet$ for the category of simplicial sets.
In a simplicial category $\scC$, we denote the simplicially enriched hom-functor by $\ul{\scC}(-,-) \colon \scC^\opp \times \scC \to \sSet$.

We write $|{-}| = \colim^\scC_{\bbDelta^\opp}$ for the colimit of simplicial objects in an $\infty$-category $\scC$.
Moreover, we also refer to $|X|$ (if it exists) as the geometric realisation of a simplicial object $X$ in $\scC$.

Usually, we denote $\infty$-categories by letters $\scC, \scD, \dots$, but for $\infty$-topoi we use bold-face letters $\bH$.
In particular, the $\infty$-topos of spaces is denoted by $\bS$.
We write $\ul{\scD}(-,-) \colon \scD^\opp \times \scD \to \bS$ for the mapping spaces in an $\infty$-category $\scD$.

We model $\infty$-categories by quasi-categories.
Given an $\infty$-category $\scC$ and a simplicial set $K \in \sSet$, we write $\Fun(K,\scC) = \ul{\sSet}(K,\scC) = \scC^K$ for the $\infty$-category of functors from $K$ to $\scC$.

We let $\bbDelta_+$ denote the augmented simplex category, i.e.~the category $\bbDelta$ with an initial object adjoined.
We usually do not distinguish notationally between augmented simplicial objects $X \in \Fun(\bbDelta_+^\opp, \scC)$ in an $\infty$-category $\scC$ and their underlying simplicial objects.
If we wish to make this distinction explicit for clarity, we will denote the latter by the restriction $X_{|\bbDelta^\opp}$.

If $\scM$ is a simplicial model category, then $\scM^\circ$ is the full simplicial subcategory on the cofibrant-fibrant objects of $\scM$.
Recall from~\cite{Lurie:HTT} that the coherent nerve $N_\Delta (\scM^\circ)$ is an $\infty$-category.

If $\scC$ is a (small) $\infty$-category, we write $\scP(\scC) = \Fun(\scC^\opp, \bS)$ for the $\infty$-category of presheaves of spaces on $\scC$.

%\vspace{-0.5cm}
\paragraph*{Acknowledgements.}

The author would like to thank Lukas Müller, Birgit Richter, Christoph Schweigert, Walker Stern, and Konrad Waldorf for helpful discussions.
The author acknowledges partial support by the Deutsche Forschungsgemeinschaft (DFG, German Research Foundation) under Germany's Excellence Strategy---EXC 2121 ``Quantum Universe''---390833306.

\section{Smooth spaces and $\infty$-topoi}
\label{sec:Sm spaces and oo-topoi}

In this section we recall and develop some background on the $\infty$-categories most relevant in this paper.
Most importantly, we consider a presheaf $\infty$-category $\bH_\infty$, whose objects can be interpreted as a general notion of smooth spaces.
We study an $\infty$-categorical version $\rmS_e \colon \bH_\infty \to \bS$ of a Quillen functor considered in~\cite{Bun:Sm_Spaces}, which provides a singular complex functor for smooth spaces.
Subsequently, we briefly recall the definition of an $\infty$-topos and of cohesion of $\infty$-topoi, and we show that $\rmS_e$ is part of the cohesion of $\bH_\infty$.

\subsection{Presheaves on cartesian spaces and the smooth singular complex}
\label{sec:Psh(Cart)}

We let $\Cart$ denote the (small) category whose objects are submanifolds of $\RN^\infty$ that are diffeomorphic to $\RN^n$ for any $n \in \NN_0$, and whose morphisms are the smooth maps between these manifolds.
We let
\begin{equation}
	\bH_\infty \coloneqq \scP(\Cart) = \Fun(\Cart^\opp, \bS)
\end{equation}
denote the $\infty$-category of presheaves of spaces on $\Cart$.
The $\infty$-category $\bH_\infty$ is presented by several model categories of simplicial presheaves on $\Cart$---for example, there is a canonical equivalence~\cite{Lurie:HTT}
\begin{equation}
	\bH_\infty \simeq N_\Delta \big( (\scH_\infty^i)^\circ \big)\,,
\end{equation}
where $\scH_\infty^i$ is the category of simplicial presheaves on $\Cart$, endowed with the injective model structure.

Let $I \coloneqq \{ c \times \RN \to c\, | \, c \in \Cart\}$ denote the set of morphisms in $\Cart$ of the form $1_c \times c_\RN$, where $c_\RN \colon \RN \to *$ is the map that collapses the real line to the point.
We can localise both $\scH_\infty^i$ and $\bH_\infty$ at this set of morphisms (or rather at its image under the Yoneda embedding), and there is still a canonical equivalence between the localisations~\cite{Lurie:HTT},
\begin{equation}
	N_\Delta \big( (L_I \scH_\infty^i)^\circ \big) \simeq L_I \bH_\infty\,.
\end{equation}
The simplicial model categories $\scH_\infty^i$ and $L_I \scH_\infty^i$ were the subject of~\cite{Bun:Sm_Spaces}.
On the level of their underlying $\infty$-categories, one of the main results of that paper can be phrased as follows.
For $k \in \NN_0$, we let $\Delta_e^k \coloneqq \{ t \in \RN^{k+1}\, | \, \sum_{i = 0}^k t^i = 1 \}$ denote the extended (affine) $k$-simplex.
This is a $k$-dimensional affine subspace of $\RN^{k+1}$, and hence forms a cartesian space.
The face and degeneracy maps of the standard topological simplices $|\Delta^k|$ extend to the extended affine simplices, turning them into a functor
\begin{equation}
	\Delta_e \colon \bbDelta \to \Cart\,,
	\quad
	[k] \mapsto \Delta_e^k\,.
\end{equation}
We let $\rmS_e \colon \bH_\infty \to \bS$ denote the composition of functors
\begin{equation}
\label{eq:def S_e}
\begin{tikzcd}[column sep=1.25cm]
	\rmS_e \colon \bH_\infty \ar[r, "\Delta_e^*"]
	& \Fun(\bbDelta^\opp, \bS) \ar[r, "\colim"]
	& \bS\,.
\end{tikzcd}
\end{equation}
We refer to this functor as the \emph{smooth singular complex functor}; viewing the $\infty$-category $\bH_\infty$ as an $\infty$-category of smooth spaces, $\rmS_e$ thus assigns an underlying ordinary space to a smooth space.

\begin{theorem}
\label{st:S_e Thm}
\emph{\cite{Bun:Sm_Spaces}}
There exist adjunctions of $\infty$-categories
\begin{equation}
\begin{tikzcd}[column sep=3cm, row sep=1.75cm]
	\bH_\infty \ar[r, shift left=0.14cm, "\Loc", "\perp"' yshift=0.05cm]
	\ar[d, shift left=-0.14cm, "\rmS_e"', "\dashv" xshift=-0.05cm]
	& L_I \bH_\infty \ar[l, hookrightarrow, shift left=0.14cm, "\iota"]
	\ar[d, "\rmS_e^I" description]
	\\
	\bS \ar[u, shift left=-0.14cm, "\rmR_e"']
	\ar[r, equal]
	& \bS \ar[u, bend left=55, "\rmL_e^I" description, "\dashv"' {xshift=0.2cm,pos=0.45}]
	\ar[u, bend left=-55, "\rmR_e^I" description, "\dashv" {xshift=-0.2cm,pos=0.45}]
\end{tikzcd}
\end{equation}
where $\rmS_e^I$ is the restriction of $\rmS_e$ to $L_I \bH_\infty \subset \bH_\infty$.
Furthermore, the following statements hold true:
\begin{myenumerate}
\item The functor $\rmS_e \colon \bH_\infty \to \bS$ preserves and reflects $I$-local equivalences.

\item The morphism $\iota$ is fully faithful, i.e.~$\Loc$ is a reflective localisation.

\item The three right-hand vertical functors are equivalences of $\infty$-categories.

\item The diagram obtained by omitting the morphism $\rmL^I_e$ is (weakly) commutative.
\end{myenumerate}
\end{theorem}

\begin{proof}
The first claim follows readily from Proposition~3.6, Corollary~3.12, and Corollary~3.37 of~\cite{Bun:Sm_Spaces}.
(Note that model categorical presentations of $\bH_\infty$, $L_I \bH_\infty$, and $\bS$ are used in~\cite{Bun:Sm_Spaces}, and the functors in the statement are presented by Quillen functors.)

Further, claim~(1) follows readily from~\cite[Cor.~3.15]{Bun:Sm_Spaces}.
Claim~(2) follows from general properties of $\infty$-categories underlying simplicial model categories and their Bousfield localisations~\cite{Lurie:HTT}.
Claim~(3) is the version on the underlying $\infty$-categories of Theorems~3.14 and~3.40 of~\cite{Bun:Sm_Spaces}.
Claim~(4) holds true because the diagram of the right-adjoints clearly commutes ($\iota$ is an inclusion, and $\rmR_e$ simply factors through $L_I \bH_\infty \subset \bH_\infty$~\cite{Bun:Sm_Spaces}).
\end{proof}

\begin{remark}
\label{eg:Mfd hkra H_oo}
There is a fully faithful embedding $\Mfd \hookrightarrow \bH_\infty$ from the category of manifolds into $\bH_\infty$:
it sends a manifold $M$ to the presheaf $\ul{M}$ of discrete spaces that maps a cartesian space $c$ to the set $\Mfd(c, M)$ of smooth maps from $c$ to $M$.
By~\cite[Thm.~5.1]{Bun:Sm_Spaces} there is a canonical equivalence of spaces $M \simeq \rmS_e \ul{M}$ for any $M \in \Mfd$, which is natural in $M$.
\qen
\end{remark}

\begin{proposition}
\label{st:Loc pres fin prods}
The localisation functor $\Loc \colon \bH_\infty \to L_I \bH_\infty$ preserves finite products.
The class $W_I$ of $I$-local equivalences in $\bH_\infty$ is closed under finite products.
\end{proposition}

\begin{proof}
By~\cite[Prop.~2.13]{Bun:Sm_Spaces}, the localisation $L_I \bH_\infty$ agrees with the localisation $L_W \bH_\infty$ of $\bH_\infty$ at the set $W$ of all collapse morphisms $c \to *$, for $c \in \Cart$.
The set $W$ is stable under finite products in $\bH_\infty$, since $\Cart$ has finite products.
Therefore, the first claim follows from~\cite[Cor.~7.1.16]{Cisinski:HCats_HoAlg}.
The second claim then follows since a morphism in $\bH_\infty$ is in $W_I$ precisely if its image under $\Loc$ is an equivalence~\cite[Prop.~5.5.4.15]{Lurie:HTT}.
\end{proof}

\begin{proposition}
For $B, C \in \bH_\infty$, let $C^B \in \bH_\infty$ denote their internal hom object in $\bH_\infty$.
The localisation functor $\Loc \colon \bH_\infty \to L_I \bH_\infty$ is given (up to equivalence) by
\begin{equation}
	\Loc \simeq \underset{\bbDelta^\opp}{\colim}^{\bH_\infty} \big( (-)^{\Delta_e} \big)\,.
\end{equation}
\end{proposition}

\begin{proof}
By Theorem~\ref{st:S_e Thm}(4), there is a canonical equivalence $S_e^I \circ \Loc \simeq \rmS_e$.
Combining this with Theorem~\ref{st:S_e Thm}(3), we obtain canonical equivalences
\begin{equation}
	\Loc
	\simeq \rmL_e^I \circ \rmS_e^I \circ \Loc
	\simeq \rmL_e^I \circ \rmS_e\,.
\end{equation}
Consider the adjunction $\tilde{\sfc} : \bS \rightleftarrows \bH_\infty : \ev_*$, where $\tilde{\sfc}$ assigns to a space $K$ the constant presheaf with value $K$, and where $\ev_*$ evaluates a presheaf on the final object $* \in \Cart$.
These functors induce an equivalence $\tilde{\sfc} : \bS \rightleftarrows L_I \bH_\infty : \ev_*$~\cite[Thm.~2.17]{Bun:Sm_Spaces}, and there is a canonical equivalence $\ev_* \simeq \rmS_e^I$ of functors $L_I \bH_\infty \to \bS$ by~\cite[Prop.~2.7, Cor.~3.15]{Bun:Sm_Spaces}.
By adjointness, we also obtain a canonical equivalence $\tilde{\sfc} \simeq \rmL_e^I$.
Consequently, there is a canonical equivalence
\begin{equation}
	\Loc
	\simeq \tilde{\sfc} \circ \rmS_e\,.
\end{equation}
We observe that there exists a canonical equivalence
\begin{equation}
	\rmS_e = \underset{\bbDelta^\opp}{\colim}^\bS \big( \Delta_e^*(-) \big)
	\simeq \ev_* \circ \,\underset{\bbDelta^\opp}{\colim}^{\bH_\infty} \big( (-)^{\Delta_e} \big)\,.
\end{equation}
By~\cite[Prop.~6.2]{Bun:Sm_Spaces}, we have that $\colim_{\bbDelta^\opp}^{\bH_\infty} ((-)^{\Delta_e})$ is a functor $\bH_\infty \to L_I \bH_\infty$; that is, it takes values in $I$-local objects.
It follows that there are canonical equivalences
\begin{equation}
	\Loc \simeq \tilde{\sfc} \circ \rmS_e
	\simeq \tilde{\sfc} \circ \ev_* \circ\, \underset{\bbDelta^\opp}{\colim}^{\bH_\infty} \big( (-)^{\Delta_e} \big)
	\simeq \underset{\bbDelta^\opp}{\colim}^{\bH_\infty} \big( (-)^{\Delta_e} \big)\,.
\end{equation}
This completes the proof.
\end{proof}

\subsection{Background on $\infty$-topoi}
\label{sec:oo-topoi}

In this section, we briefly recall some background on $\infty$-topoi.
Most of the material in this section can be found in~\cite{Lurie:HTT,Schreiber:DCCT,NSS:oo-bundles}.
For $n \in \NN_0$ and a subset $S \subset [n]$, let $\Delta^S \subset \Delta^n$ be the full $\infty$-subcategory on the vertices that lie in $S$.
There is a canonical isomorphism $\Delta^S \cong \Delta^{|S|}$ as simplicial sets, where $|S|$ is the cardinality of $S$.
The simplicial set $\Delta^S$ can equivalently be seen as the image of  an inclusion $\Delta^{|S|} \hookrightarrow \Delta^n$ that sends the $i$-th vertex of $\Delta^{|S|}$ to the vertex of $\Delta^n$ which corresponds to the $i$-th element of $S$ (with the order induced from the inclusion $S \subset [n]$).
Given an $\infty$-category $\scC$ and a simplicial object $X \in \Fun (\bbDelta^\opp, \scC)$, we set $X(S) \coloneqq X(\Delta^{|S|})$.
This comes with a canonical morphism $X_n \to X(S)$, induced by the inclusion $S \subset [n]$.

\begin{definition}
\label{def:Gpd objs}
Let $\scC$ be an $\infty$-category.
A \emph{groupoid object} in $\scC$ is a simplicial object $X \in \Fun(\bbDelta^\opp, \scC)$ such that, for every $n \in \NN_0$ and every partition $[n] = S \cup S'$ (as finite sets) with $S \cap S' \cong \{*\}$ consisting of a single element, the diagram
\begin{equation}
\begin{tikzcd}
	X_n \ar[r] \ar[d]
	& X(S') \ar[d]
	\\
	X(S) \ar[r]
	& X_0
\end{tikzcd}
\end{equation}
is a pullback diagram in $\scC$.
\end{definition}

In particular, any groupoid object is a category object (see also Definition~\ref{def:category object}):
for every $n \geq 1$, the spine decomposition $[n] = [1] \sqcup_{[0]} \cdots \sqcup_{[0]} [1]$ induces a canonical equivalence
\begin{equation}
	X_n \simeq \underbrace{X_1 \times_{X_0} \cdots \times_{X_0} X_1}_{n+1 \text{ factors}}\,.
\end{equation}
We denote the full subcategory of $\Fun(\bbDelta^\opp, \scC)$ on the groupoid objects by
\begin{equation}
	\Gpd(\scC) \subset \Fun(\bbDelta^\opp, \scC)\,.
\end{equation}
Let $\bbDelta_+$ denote the simplex category with an initial object $[-1]$ adjoined.
For $n \in \NN_0$, let $\bbDelta_{+, \leq n} \subset \bbDelta_+$ be the full subcategory on the objects $[-1], \ldots, [n]$.
In particular, $\bbDelta_{+, \leq 0}^\opp$ is the category with two objects and one non-trivial morphism $[0] \to [-1]$.
Therefore, any morphism $p \colon P \to B$ in an $\infty$-category $\scC$ defines an object $\{p\} \in \Fun(\bbDelta_{+, \leq 0}^\opp, \scC)$.

\begin{definition}
\label{def:Cech nerve}
Given a morphism $p \colon P \to B$ in an $\infty$-category $\scC$, its \emph{\v{C}ech nerve $\cC p$} (if it exists) is the augmented simplicial object obtained as the right Kan extension
\begin{equation}
\begin{tikzcd}[column sep=1.25cm, row sep=1cm]
	\bbDelta_{+, \leq 0}^\opp \ar[r, "\{p\}"] \ar[d, hookrightarrow, "\imath"'] & \scC
	\\
	\bbDelta_+^\opp \ar[ur, dashed, "\cC p"']
\end{tikzcd}
\end{equation}
That is, $\cC p = \Ran_\imath \{p\}$, where $\imath$ is the inclusion $\bbDelta_{+, \leq 0}^\opp \hookrightarrow \bbDelta_+^\opp$.
\end{definition}

For later use, we record:

\begin{proposition}
\label{st:gpd obs via X_1 X_1 X_0}
\emph{\cite[Prop.~6.1.2.11]{Lurie:HTT}}
Let $\scC$ be an $\infty$-category, and let $X \colon \bbDelta_+^\opp \to \scC$ be an augmented simplicial object.
The following are equivalent:
\begin{myenumerate}
\item $X$ is a right Kan extension of $X_{|\bbDelta_{+, \leq 0}^\opp}$.

\item The underlying simplicial object $X_{|\bbDelta^\opp}$ is a groupoid object in $\scC$ and the diagram
\begin{equation}
\begin{tikzcd}
	X_1 \ar[r, "d_0"] \ar[d, "d_1"']
	& X_0 \ar[d]
	\\
	X_0 \ar[r]
	& X_{-1}
\end{tikzcd}
\end{equation}
is a pullback square in $\scC$.
\end{myenumerate}
\end{proposition}

\begin{definition}
Let $\scC$ be an $\infty$-category, and let $p \colon P \to B$ be a morphism in $\scC$.
Then, $p$ is an \emph{effective epimorphism} if the augmented \v{C}ech nerve $\cC p \in \Fun(\bbDelta_+^\opp, \scC) \cong \Fun((\bbDelta^\opp)^\triangleright, \scC)$ is a colimiting cocone in $\scC$.
In other words, the morphism $p \colon P \to B$ is an effective epimorphism precisely if the colimit $|\cC p|$ exists in $\scC$ and the induced morphism $|\cC p| \to B$ is an equivalence.
\end{definition}

Let $X \colon \bbDelta_+^\opp \to \scC$ be an augmented simplicial object in an $\infty$-category $\scC$.
We denote the morphism \smash{$X_0 \to X_{-1}$} by $p$.
Suppose that its \v{C}ech nerve $\cC p$ exists.
Observe that \smash{$\{p\} = \imath^*X$} as objects in $\Fun(\Delta^1, \scC) \cong \Fun(\bbDelta_{+, \leq 0}^\opp, \scC)$.
By the adjointness property of the right Kan extension, there is a canonical equivalence of mapping spaces
\begin{equation}
	\ul{\Fun}(\bbDelta_{+, \leq 0}^\opp,\scC) (\imath^* X, \{p\})
	\simeq \ul{\Fun}(\bbDelta_+^\opp,\scC) (X, \cC p)\,.
\end{equation}
The identity $\imath^*X = \{p\}$ thus induces a canonical morphism
\begin{equation}
\label{eq:comp mp gpd v Cech nerve}
	\eta \colon X \longrightarrow \cC p\,.
\end{equation}

We define $\infty$-topoi in terms of the Giraud-Lurie-Rezk axioms~\cite[Def.~6.1.0.4, Thm.~6.1.0.6]{Lurie:HTT}:

\begin{definition}
\label{def:oo-topos}
An \emph{$\infty$-topos} is an $\infty$-category $\bH$ satisfying the following axioms:
\begin{myenumerate}
\item $\bH$ is presentable.
In particular, $\bH$ has all limits and colimits~\cite[Def.~5.5.0.1, Cor.~5.5.2.4]{Lurie:HTT}.
We denote its initial object by $\emptyset \in \bH$ and its final object by $* \in \bH$.

\item Colimits in $\bH$ are universal:
for any diagram $D \colon K \to \bH$, any cocone $\overline{D} \colon K^\triangleright \to \bH$ under $D$ with apex $C \in \bH$, and for any morphism $f \colon B \to C$ in $\bH$, the induced morphism
\begin{equation}
	\underset{K}{\colim}^\bH(D \underset{\sfc C}{\times} \sfc B)
	\longrightarrow \big( \underset{K}{\colim}^\bH D \big) \underset{C}{\times} B
\end{equation}
is an equivalence (on the left-hand side, $\sfc B, \sfc C \colon K \to \bH$ are the constant diagrams with values $B$ and $C$, respectively, and the pullback is formed in $\Fun(K, \bH)$).

\item Coproducts in $\bH$ are disjoint:
for every pair of objects $B, C \in \bH$, the pushout diagram
\begin{equation}
\begin{tikzcd}
	\emptyset \ar[r] \ar[d] & B \ar[d]
	\\
	C \ar[r] & B \sqcup C
\end{tikzcd}
\end{equation}
is also a pullback diagram.

\item Groupoids in $\bH$ are effective:
given any groupoid object $X \in \Gpd(\bH)$, let $p \colon X_0 \to |X|$ denote the canonical morphism which is part of the colimiting cocone.
Then, the comparison morphism \smash{$\eta \colon X \to \cC p$} constructed in~\eqref{eq:comp mp gpd v Cech nerve} is an equivalence of simplicial objects in $\bH$.
In particular, $p$ is an effective epimorphism.
\end{myenumerate}
\end{definition}

\begin{example}
We list some examples of $\infty$-topoi; we will mostly be using the first two of these.
\begin{myenumerate}
\item The $\infty$-category of spaces $\bS$ is an $\infty$-topos.

\item Any $\infty$-category $\scP(\scC)$ of presheaves of spaces on a (small) $\infty$-category $\scC$ is an $\infty$-topos.

\item Any accessible, left-exact, reflective localisation of an $\infty$-category $\scP(\scC)$ of presheaves on a small $\infty$-category $\scC$ is an $\infty$-topos; in fact, every $\infty$-topos is equivalent to an $\infty$-topos of this form~\cite[Thm.~6.1.0.6, Prop.~6.1.5.3]{Lurie:HTT}.
\qen
\end{myenumerate}
\end{example}

We will later need the following properties of effective epimorphisms in an $\infty$-topos:

\begin{lemma}
\label{st:EEpis stable under pb and po}
In an $\infty$-topos $\bH$, effective epimorphisms are stable under pullbacks and colimits.
\end{lemma}

\begin{proof}
The fact that effective epimorphisms are stable under pullback is~\cite[Prop.~6.2.3.15]{Lurie:HTT}.
The effective epimorphisms in $\bH$ are precisely the $(-1)$-connected%
\footnote{Note that there is a shift in convention between~\cite{Lurie:HTT} and the nLab:
A morphism $f$ in $\bH$ is $n$-connect\emph{ive} in~\cite{Lurie:HTT} if and only if it is $(n{-}1)$-connect\emph{ed} in the conventions used on the nLab.
We follow the nLab here.}
morphisms~\cite[Def.~6.5.1.10]{Lurie:HTT}.
The class of $n$-connected and $n$-truncated morphisms in an $\infty$-topos form a factorisation system~\cite[Rmk.~5.2.8.16]{Lurie:HTT}, and the left class of morphisms in a factorisation system is stable under colimits~\cite[Prop.~5.2.8.6(7)]{Lurie:HTT}.
\end{proof}

An important notion of morphism between $\infty$-topoi is that of a geometric morphism, which is more adapted to the additional structure on $\infty$-topoi than a mere functor of $\infty$-categories:

\begin{definition}
Let $\bH$ and $\bH'$ be $\infty$-topoi.
A \emph{geometric morphism} of $\infty$-topoi from $\bH$ to $\bH'$ is a functor $\rmF_* \colon \bH \to \bH'$ admitting a left exact left adjoint $\rmF^* \colon \bH' \to \bH$.
\end{definition}

One can show that the $\infty$-category $\bS$ of spaces is final in the $\infty$-category of $\infty$-topoi and geometric morphisms~\cite[Prop.~6.3.4.1]{Lurie:HTT}.
That is, for every $\infty$-topos $\bH$ there exists an essentially unique geometric morphism $\bH \to \bS$.
We will denote the corresponding adjunction by $\wc : \bS \rightleftarrows \bH : \Gamma$ and refer to $\Gamma$ as the \emph{global-section functor} of $\bH$.

\begin{example}
Consider a Grothendieck $\infty$-site, i.e.~a small $\infty$-category $\scC$ with a Grothendieck coverage.
Suppose $\scC$ additionally has a final object.
If $\bH$ is the $\infty$-category of sheaves of spaces on $\scC$, then the global section functor $\Gamma$ of $\bH$ agrees with the evaluation of sheaves at the final object of $\scC$.
In particular, this applies to $\bH_\infty$, the $\infty$-topos of presheaves of spaces on $\Cart$ from Section~\ref{sec:Psh(Cart)}.
\qen
\end{example}

\begin{definition}
\label{def:cohesive oo-topos}
An $\infty$-topos $\bH$ is called \emph{cohesive} if the adjunction $\wc : \bH \rightleftarrows \bS : \Gamma$ can be extended to a triple adjunction $\Pi \dashv \wc \dashv \Gamma \dashv codisc$, in which the left adjoint $\Pi$ preserves finite products and the right adjoint $codisc$ is fully faithful.
\end{definition}

Cohesive $\infty$-topoi have been studied extensively in~\cite{Schreiber:DCCT} and related works.

\begin{theorem}
\label{st:S_e and cohesion}
The $\infty$-topos $\bH_\infty$ is cohesive, i.e.~there exists a triple adjunction $\Pi \dashv \wc \dashv \Gamma \dashv codisc$ as in Definition~\ref{def:cohesive oo-topos}, and there is a canonical equivalence
\begin{equation}
	\Pi \simeq \rmS_e\,.
\end{equation}
\end{theorem}

\begin{remark}
The fact that $\bH_\infty$ is cohesive is not new, see~\cite{Schreiber:DCCT}.
The second statement has been proven in a model categorical presentation in~\cite{Bun:Sm_Spaces}, and a different argument has been given in~\cite{BEBdBP:Class_sp_of_oo-sheaves}.
Here, we give an $\infty$-categorical proof of this fact for completeness.
\qen
\end{remark}

\begin{proof}
The $\infty$-topos $\bH_\infty = \scP(\Cart)$ admits a right-adjoint to its global-section functor $\Gamma$ by abstract arguments: evaluation of a presheaf at any object preserves colimits, and since both $\bH_\infty$ and $\bS$ are presentable, $\Gamma$ must admit a further right adjoint~\cite[Prop.~7.11.8]{Cisinski:HCats_HoAlg}.
It is well-known that this can in fact be extended into a triple adjunction which establishes that $\bH_\infty$ is cohesive~\cite{Schreiber:DCCT}.

For the second part of the statement, we show that $\rmS_e$ is left-adjoint to the functor $\wc$.
Recall from Section~\ref{sec:Psh(Cart)} that here $\wc$ simply sends a space $K \in \bS$ to the constant presheaf on $\Cart$ with value $K$.
Further, recall from the proof of Proposition~\ref{st:Loc pres fin prods} (and~\cite[Prop.~2.13]{Bun:Sm_Spaces}) that the $I$-local objects in $\bH_\infty$ are precisely the essentially constant presheaves, i.e.~those $F \in \bH_\infty$ for which the canonical morphism $F(*) \to F(c)$ is an equivalence for every $c \in \Cart$.
Equivalently, $F$ is $I$-local if and only if the canonical morphism $\wc \circ \Gamma (F) \to F$ is an equivalence in $\bH_\infty$.
Further, by Theorem~\ref{st:S_e Thm} the right adjoint $\rmR_e$ to $\rmS_e$ factors through the localisation $L_I \bH_\infty \subset \bH_\infty$; this is precisely the full $\infty$-subcategory of $\bH_\infty$ on the $I$-local objects.

Consider the two adjunctions $\rmS_e : \bH_\infty \rightleftarrows \bS : \rmR_e$ and $\wc : \bS \rightleftarrows \bH_\infty : \Gamma$.
They induce an adjunction
\begin{equation}
\begin{tikzcd}[column sep=1.25cm]
	\rmS_e \circ \wc : \bS \ar[r, shift left=0.14cm, "\perp"' yshift=0.05cm]
	& \bS : \Gamma \circ \rmR_e\,. \ar[l, shift left=0.14cm]
\end{tikzcd}
\end{equation}
By the definition~\eqref{eq:def S_e} of $\rmS_e$, for any space $K \in \bS$ we have a canonical natural equivalence
\begin{align}
	\rmS_e \circ \wc (K)
	&= \underset{\bbDelta^\opp}{\colim}^\bS \big( \wc(K)(\Delta_e) \big)
	\simeq K\,,
\end{align}
because left-hand side is the colimit of a constant diagram over an indexing category whose nerve is contractible in the Kan-Quillen model structure on $\sSet$ (see Lemma~\ref{st:(co)lims of const diags}, Example~\ref{eg:contractible Cats for const diags}).
In other words, there is a canonical natural equivalence $\rmS_e \circ \wc \simeq 1_\bS$.
Consequently, there is also a canonical equivalence on the right adjoints, $\Gamma \circ \rmR_e \simeq 1_\bS$.
We obtain natural equivalences
\begin{equation}
	\wc
	\simeq \wc \circ \Gamma \circ \rmR_e
	\simeq \rmR_e\,.
\end{equation}
In the second equivalence we have used that $\rmR_e$ takes values in $L_I \bH_\infty \subset \bH_\infty$ and that on objects in $L_I \bH_\infty$ the morphism $\wc \circ \Gamma \to 1_{\bH_\infty}$ is an equivalence.
From the equivalence $\rmR_e \simeq \wc$ and the adjunction $\rmS_e \dashv \rmR_e$ we infer that $\rmS_e$ is a further left adjoint to $\wc$.
Hence, it is equivalent to the functor $\Pi$.
\end{proof}

Theorem~\ref{st:S_e and cohesion} shows that the smooth singular complex functor $\rmS_e \colon \bH_\infty \to \bS$ has a deep homotopical meaning for assigning homotopy types to objects in $\bH_\infty$ and for studying these homotopy types.
It also provides an additional, refined, perspective on the good homotopical properties of the functor $\rmS_e$ that were found and studied in~\cite{Bun:Sm_Spaces}.
Finally, note that we also obtain from this a natural equivalence
\begin{equation}
	\rmS_e(F) = \underset{\bbDelta^\opp}{\colim}^\bS \big( F(\Delta_e) \big)
	\simeq \underset{\Cart^\opp}{\colim}^\bS(F)\,.
\end{equation}
That is, we see that $\rmS_e$ computes the $\infty$-categorical colimit of $\Cart^\opp$-shaped diagrams of spaces, and thus, by~\cite[Thm.~6.4.5]{Cisinski:HCats_HoAlg}, that $N \Delta_e \colon N \bbDelta \to N \Cart$ is a cofinal morphism in $\sSet$.

\section{Principal $\infty$-bundles and group extensions in $\infty$-topoi}
\label{sec:PBuns and Grp Exts in oo-topoi}

In this section, starting from the theory introduced in~\cite{NSS:oo-bundles}, we develop characterisations of principal $\infty$-bundles and extensions of group objects in $\infty$-topoi.
These characterisations are interesting already in their own right.
In Section~\ref{sec:HoThy Smooth String} they will also allow us to transfer the definition of string group extensions from $\bS$ to $\bH_\infty$ and to construct explicit smooth models for the string group.

\subsection{Groups and group extensions}
\label{sec:Grps and Grp-Exts}

Here we recall the definitions of group objects and their extensions in $\infty$-topoi~\cite{NSS:oo-bundles}.
We investigate how to compute limits of group and groupoid objects in $\infty$-topoi, and how group objects and their classifying objects behave under functors between $\infty$-topoi that preserve finite products and geometric realisations.

Let $\bH$ be an $\infty$-topos, and let $\Gpd(\bH)$ be the $\infty$-category of groupoid objects in $\bH$.
Further, let $\EEpi(\bH) \subset \Fun(\Delta^1, \bH)$ denote the full $\infty$-subcategory on the effective epimorphisms in $\bH$.
Recall that by Definition~\ref{def:oo-topos}(4) and Proposition~\ref{st:gpd obs via X_1 X_1 X_0} there is a canonical equivalence
\begin{equation}
\label{eq:Gpd(H) = EEpi(H)}
	\Gpd(\bH) \simeq \EEpi(\bH)\,,
\end{equation}
given by forming colimits and \v{C}ech nerves, respectively.

\begin{definition}
\label{def:Grp(H)}
Let $\scC$ be an $\infty$-category.
Let $\Grp(\scC) \subset \Gpd(\scC)$ denote the full $\infty$-subcategory on those groupoid objects $X$ in $\scC$ for which $X_0$ is a final object of $\scC$.
We call $\Grp(\scC)$ the $\infty$-category of \emph{group objects} in $\scC$.
\end{definition}

\begin{proposition}
\label{st:sH, Gpd(H), Grp(H) localisations}
For any $\infty$-topos $\bH$, there are reflective localisations
\begin{equation}
\begin{tikzcd}
	\Fun(\bbDelta^\opp, \bH) \ar[r, shift left=0.14cm, "\perp"' yshift=0.05cm]
	& \Gpd(\bH) \ar[l, hookrightarrow, shift left=0.14cm] \ar[r, shift left=0.14cm, "\perp"' yshift=0.05cm]
	& \Grp(\bH)\,. \ar[l, hookrightarrow, shift left=0.14cm]
\end{tikzcd}
\end{equation}
\end{proposition}

\begin{proof}
First, the right adjoints in the above sequence of adjunctions are fully faithful by definition.
The first morphism has a left adjoint by~\cite[Prop.~6.1.2.9]{Lurie:HTT}.
For the second left adjoint, we use the equivalence~\eqref{eq:Gpd(H) = EEpi(H)}%
\footnote{This proof goes back to a mathoverflow answer by Jacob Lurie, see~\url{https://mathoverflow.net/questions/140639/is-the-category-of-group-objects-in-an-infty-1-topos-reflective-as-a-subcat/140742\#140742}.}:
this equivalence induces a commutative square
\begin{equation}
\begin{tikzcd}
	\Gpd(\bH) \ar[d, "\simeq"']
	& \Grp(\bH) \ar[l, hookrightarrow] \ar[d, "\simeq"]
	\\
	\EEpi(\bH)
	& \EEpi_*(\bH) \ar[l, hookrightarrow]
\end{tikzcd}
\end{equation}
where $\EEpi_*(\bH) \subset \EEpi(\bH)$ is the full $\infty$-subcategory on those effective epimorphisms $f \colon X_0 \to X_{-1}$ where $X_0$ is a final object.
A left adjoint to the bottom morphism is given by the functor that sends an effective epimorphism $f \colon X_0 \to X_{-1}$ to the morphism $g \colon * \to X_{-1} \sqcup_{X_0} *$ induced by the pushout.
Since $f$ is an effective epimorphism, Lemma~\ref{st:EEpis stable under pb and po} implies that so is $g$.
\end{proof}

For a group object $G \in \Grp(\bH)$ in an $\infty$-topos $\bH$, we set
\begin{equation}
	\rmB G \coloneqq \underset{\bbDelta^\opp}{\colim}^\bH\, G = |G| \in \bH\,.
\end{equation}
Note that in an $\infty$-topos $\bH$, for any groupoid object $X \in \Gpd(\bH)$ the map $X_0 \to \colim^\bH_{\bbDelta^\opp} X$ is an effective epimorphism.
Hence, given a group object $G$ in $\bH$, the morphism $* \simeq G_0 \to \rmB G$ is an effective epimorphism.
Moreover, the functor $\rmB$ is part of an equivalence~\cite[Lemma~7.2.2.11]{Lurie:HTT} (see also~\cite[Thm.~2.19]{NSS:oo-bundles})
\begin{equation}
\label{eq:Omega -| B adjoint equiv}
\begin{tikzcd}[column sep=1.25cm]
	\bH^{*/}_{\geq 1} \ar[r, shift left=0.14cm, "\perp"' yshift=0.05cm, "\Omega"]
	& \Grp(\bH)\,, \ar[l, shift left=0.14cm, "\rmB"]
\end{tikzcd}
\end{equation}
where \smash{$\bH^{*/}_{\geq 1}$} is the $\infty$-category of pointed, connected objects in $\bH$.
Note that, for each $T \in \bH_{\geq 1}^{*/}$, we view $\Omega T$ as a group object in $\bH$, i.e.~$\Omega T \in \Grp(\bH)$.
Its underlying object in $\bH$ is $(\Omega T)_1 \eqqcolon \Omega_1 T$.

Unravelling the definition, we obtain that a group object in an $\infty$-category $\scC$ with a final object $* \in \scC$ is equivalently a simplicial object $G$ in $\scC$ such that $G_0 \simeq *$ and, for any $[n] \in \bbDelta$ and any partition $[n] = S \cup S'$ as finite sets with $S \cap S' \cong \{*\}$ consisting of a single element, the diagram
\begin{equation}
\begin{tikzcd}
	G_n \ar[r] \ar[d] & G(S) \ar[d]
	\\
	G(S') \ar[r] & G_0 \simeq *
\end{tikzcd}
\end{equation}
is a pullback diagram in $\scC$.
That is, there is a canonical equivalence $G_n \eq G(S) \times G(S')$.
In particular, iterating this for the spine partition $[n] = [1] \sqcup_{[0]} \cdots \sqcup_{[0]} [1]$, we obtain a canonical equivalence
\begin{equation}
	G_n \eq G_1^n\,.
\end{equation}

\begin{proposition}
\label{st:preservation of Grp and B}
Let $\rmL \colon \bH \to \bH'$ be a functor between $\infty$-topoi.
\begin{myenumerate}
\item If $\rmL$ preserves finite products, then it preserves group objects.

\item If $\rmL$ additionally preserves geometric realisations, then, for any group object $G$ in $\bH$, there is a canonical equivalence
\begin{equation}
	\rmB (\rmL G) \simeq \rmL(\rmB G)\,.
\end{equation}
\end{myenumerate}
\end{proposition}

\begin{proof}
The first part of the Proposition is known~\cite{Lurie:HTT}; we include its proof only for completeness.
Any functor $F \colon \scC \to \scD$ between $\infty$-categories preserves simplicial objects, i.e.~it induces a functor $\Fun(\bbDelta^\opp, \scC) \longrightarrow \Fun(\bbDelta^\opp, \scD)$.
Suppose that $G \in \Fun(\bbDelta^\opp, \bH)$ is a group object in $\bH$.
Since $\rmL$ preserves finite products, it preserves final objects, so that $(\rmL G)_0 \simeq *$ is final in $\bH'$.
For $n \neq 0$ and any partition $[n] = S \cup S'$ with $S \cap S' \cong \{*\}$, we obtain a commutative diagram
\begin{equation}
\begin{tikzcd}[column sep=1.25cm, row sep=1cm]
	(\rmL G)_n = \rmL (G_n) \ar[r, "\simeq"] \ar[rd]
	& \rmL \big( G(S) \times G(S') \big) \ar[d, "\simeq"]
	\\
	& \rmL G(S) \times \rmL G(S')
\end{tikzcd}
\end{equation}
The top morphism is an equivalence since $G$ is a group object in $\bH$ and the vertical morphism is an equivalence since $\rmL$ preserves products.
This proves claim~(1).
Using that $\rmB G = \colim^\bH_{\bbDelta^\opp} G = |G|$, the second part is now immediate.
\end{proof}

\begin{remark}
We will prove a number of statements about functors as in Proposition~\ref{st:preservation of Grp and B}(2), i.e.~functors between $\infty$-topoi which preserve geometric realisations and finite products.
An important class of such functors is given by the additional left-adjoints of cohesive $\infty$-topoi---see Definition~\ref{def:cohesive oo-topos}.
In particular, the functor $\rmS_e \colon \bH_\infty \to \bS$ from Section~\ref{sec:Psh(Cart)} is of this type by Theorem~\ref{st:S_e and cohesion}.
\qen
\end{remark}

\begin{lemma}
\label{st:equivs of Gpd and Grp objects}
Let $\bH$ be an $\infty$-topos.
\begin{myenumerate}
\item A morphism $X \to Y$ in $\Gpd(\bH)$ is an equivalence if and only if $X_i \to Y_i$ is an equivalence in $\bH$ for $i = 0,1$.

\item A morphism $G \to H$ in $\Grp(\bH)$ is an equivalence if and only if $G_1 \to H_1$ is an equivalence in $\bH$.
\end{myenumerate}
\end{lemma}

\begin{proof}
Proposition~\ref{st:sH, Gpd(H), Grp(H) localisations} implies that an equivalence of groupoid objects $X \eq Y$ in $\bH$ is the same as an objectwise equivalence of the underlying simplicial objects in $\bH$:
$X$ and $Y$ are local objects in $\Fun(\bbDelta^\opp, \bH)$ with respect to the localisation $\Gpd(\bH) \subset \Fun(\bbDelta^\opp, \bH)$, so that the local equivalences between them are precisely the original, i.e.~the levelwise, equivalences.
In particular, this implies the `only if' part of claim~(1).

Conversely, if we are given a morphism $X \eq Y$ of groupoid objects in $\bH$ such that $X_i \to Y_i$ is an equivalence for $i = 0,1$, then it follows that $X \eq Y$ is a levelwise equivalence of simplicial objects; this is because for each $n \in \NN_0$ there is a canonical equivalence \smash{$X_n \simeq X_1 \times _{X_0} \cdots \times_{X_0} X_1$}, natural in $X \in \Gpd(\bH)$.
It then follows that the morphism $X \to Y$ is also an equivalence in $\Gpd(\bH)$.

The same line of argument shows the second claim.
\end{proof}

\begin{lemma}
\label{st:lims in Gpd and Grp objects}
Let $\bH$ be an $\infty$-topos, and let $K \in \sSet$ be a simplicial set.
\begin{myenumerate}
\item A diagram $X \colon K^\triangleleft \to \Gpd(\bH)$ of groupoid objects in $\bH$ is a limit diagram if and only if the composition $\iota X \colon K^\triangleleft \to \Gpd(\bH) \hookrightarrow \Fun(\bbDelta^\opp,\bH)$ is a limit diagram.

\item A diagram $X \colon K^\triangleleft \to \Gpd(\bH)$ of groupoid objects in $\bH$ is a limit diagram if and only if the induced diagrams $X_i \colon K^\triangleleft \to \bH$ are limit diagrams for $i = 0,1$.

\item A diagram $G \colon K^\triangleleft \to \Grp(\bH)$ of group objects in $\bH$ is a limit diagram if and only if the the composition $\jmath G \colon K^\triangleleft \to \Grp(\bH) \hookrightarrow \Gpd(\bH)$ is a limit diagram.

\item A diagram $G \colon K^\triangleleft \to \Grp(\bH)$ of group objects in $\bH$ is a limit diagram if and only if the induced diagram $G_1 \colon K^\triangleleft \to \bH$ is a limit diagram.
\end{myenumerate}
\end{lemma}

\begin{proof}
One can see the `only if' direction of claims~(1) and~(2) as follows:
we first note that since the inclusion $\Gpd(\bH) \subset \Fun(\bbDelta^\opp, \bH)$ is a right adjoint, we have that if $X \colon K^\triangleleft \to \Gpd(\bH)$ is a limit diagram, then so is $X \colon K^\triangleleft \to \Fun(\bbDelta^\opp, \bH)$.
Further, since limits in diagram $\infty$-categories are computed pointwise, this is equivalent to the functor $X_i \colon K^\triangleleft \to \bH$ being a limit diagram in $\bH$ for every $[i] \in \bbDelta$.

For the converse direction in claim~(1), we first show that limits of diagrams in $\Gpd(\bH)$ can be computed in $\Fun(\bbDelta^\opp, \bH)$.
More precisely, a functor $X \colon K^\triangleleft \to \Gpd(\bH)$ is a limit diagram whenever its composition with the inclusion $\iota \colon \Gpd(\bH) \hookrightarrow \Fun(\bbDelta^\opp, \bH)$ is so, i.e.~the inclusion reflects limits.
Equivalently, the $\infty$-subcategory $\Gpd(\bH) \hookrightarrow \Fun(\bbDelta^\opp, \bH)$ is closed under limits in $\Fun(\bbDelta^\opp, \bH)$.
This is seen as follows:
consider a functor $X \colon K^\triangleleft \to \Gpd(\bH)$ and a decomposition $[n] = S \cup S'$ with $S \cap S' = \{*\}$.
This induces an equivalence
\begin{equation}
	X_n \eq X(S) \underset{X_0}{\times} X(S')
\end{equation}
in $\Fun(K^\triangleleft, \bH)$.
Setting $Y \coloneqq \lim_K^{\Fun(\bbDelta^\opp,\bH)}(\iota X)$ and using that limits commute with limits~\cite[Lemma~5.5.2.3]{Lurie:HTT}, we have
\begin{align}
\label{eq:Gpd closed under lims}
	Y_n
	&\simeq \big( \underset{K}{\lim}^{\Fun(\bbDelta^\opp, \bH)}(\iota X) \big)_n
	\\*
	&\simeq \underset{K}{\lim}^{\bH}(\iota X_n)
	\\
	&\simeq \underset{K}{\lim}^{\bH} \big( \iota X(S) \underset{\iota X_0}{\times} \iota X(S') \big)
	\\
	&\simeq \big( \underset{K}{\lim}^{\bH} \iota X(S) \big) \underset{(\underset{K}{\lim}^{\bH} \iota X_0)}{\times} \big( \underset{K}{\lim}^{\bH} \iota X(S') \big)
	\\*
	&\simeq Y(S) \times_{Y_0} Y(S')\,,
\end{align}
which shows that $Y \in \Fun(\bbDelta^\opp, \bH)$ is local with respect to the localisation $\Gpd(\bH) \hookrightarrow \Fun(\bbDelta^\opp,\bH)$, i.e.~that $Y \in \Gpd(\bH)$.
Since the inclusion $\iota \colon \Gpd(\bH) \hookrightarrow \Fun(\bbDelta^\opp, \bH)$ is fully faithful, $Y$ is also a limit of the diagram $X \colon K \to \Gpd(\bH)$.
Consequently, if the composition $\iota X \colon K^\triangleleft \to \Gpd(\bH) \hookrightarrow \Fun(\bbDelta^\opp, \bH)$ is a limit diagram, then so is $X \colon K^\triangleleft \to \Gpd(\bH)$.

For the converse direction in claim~(2), suppose that $X \colon K^\triangleleft \to \Gpd(\bH)$ is a diagram such that the functors $X_i \colon K^\triangleleft \to \bH$ are limit diagrams for $i = 0,1$.
By part~(1) it suffices to show that the composition $\iota X \colon K^\triangleleft \to \Fun(\bbDelta^\opp,\bH)$ is a limit diagram; that is, it suffices to show that $X_i \colon K^\triangleleft \to \bH$ is a limit diagram for every $[i] \in \bbDelta$.

Since $X \colon K^\triangleleft \to \Fun(\bbDelta^\opp,\bH)$ is valued in groupoid objects, and since limits in $\bH$ commute with limits, it follows from~\eqref{eq:Gpd closed under lims} that for every $[n] \in \bbDelta$ the diagram $X_n \colon K^\triangleleft \to \bH$ is equivalent to a limit diagram $X_1 \times _{X_0} \cdots \times_{X_0} X_1 \colon K^\triangleleft \to \bH$, and is hence a limit diagram itself.

The proof of claim~(3) proceeds along the exact same line as the proof of part~(2):
the key insight is the fact that if $G \colon K^\triangleleft \to \Grp(\bH)$ is a diagram such that the composition $\jmath G \colon K^\triangleleft \to \Gpd(\bH)$ is a limit diagram, then \smash{$\lim_K^{\Gpd(\bH)}(\jmath G)$} is still local with respect to the localisation $\Grp(\bH) \subset \Gpd(\bH)$.

Claim~(4) is then the combination of claims~(2) and~(3).
\end{proof}

Having established several properties of the $\infty$-category of group objects in $\bH$, we now define extensions of group objects:

\begin{definition}
\label{def:oo-grp extension}
\emph{\cite[Def.~4.26]{NSS:oo-bundles}}
Let $A$ and $H$ be group objects in an $\infty$-topos $\bH$.
An \emph{extension of group objects of $H$ by $A$} is a sequence $A \to G \to H$ in the $\infty$-category $\Grp(\bH)$ such that the sequence $\rmB A \to \rmB G \to \rmB H$ is a fibre sequence in $\bH$.
\end{definition}

\begin{remark}
This definition of a group extension has advantages from a theoretical perspective.
Nevertheless, it appears that there should be a simpler definition that more directly generalises extensions of groups in $\Set$ to the $\infty$-categorical setting.
For group objects in $\Set$, a group extension is a sequence $A \to G \to H$ of group homomorphisms such that $G \to H$ is surjective and $A$ is the fibre of the morphism $G \to H$ at the identity element of $H$.
We will prove in Theorem~\ref{st:oo-grp ext via pfbuns} that one can indeed generalise this view on group extensions to group objects in $\infty$-topoi.
\qen
\end{remark}

\subsection{Group actions in $\infty$-categories}
\label{sec:Grp actions}

We now investigate actions of group objects in $\infty$-topoi.
Let $\sSet_J$ denote the category of simplicial sets endowed with the Joyal model structure.
For $K \in \sSet$, we let $\obj(K)$ be the set $K_0$ of vertices of $K$, seen as a discrete simplicial set.
Let $J \coloneqq \Delta^1[f^{-1}]$ be the localisation of $\Delta^1$ at its non-trivial edge (see e.g.~\cite[Sec.~3.3]{Cisinski:HCats_HoAlg}).

\begin{lemma}
\label{st:equiv on objs pulls back functoriality}
Let $\scC$ be an $\infty$-category, and let $K$ be a simplicial set.
\begin{myenumerate}
\item The inclusion $\iota \colon \obj(K) \hookrightarrow K$ induces a morphism
\begin{equation}
	\iota^* \colon \scC^K = \Fun(K, \scC) \longrightarrow \Fun \big( \obj(K), \scC \big) = \scC^{\obj(K)}
\end{equation}
of simplicial sets, which is a fibration between fibrant objects in the Joyal model structure.

\item Consider either of the inclusions $\Delta^{\{i\}} \hookrightarrow J$, where $i = 0,1$.
The induced morphism
\begin{equation}
	\Fun(J,\scC^K) \longrightarrow \scC^K \underset{\scC^{\obj(K)}}{\times} \Fun \big(J,\scC^{\obj(K)} \big)
\end{equation}
is a trivial Kan fibration.

\item Let $g \colon K \to \scC$ and $g' \colon \obj(K) \to \scC$ be functors.
For any equivalence $\eta \colon \iota^* g \eq g'$, consider the space of pairs $(\hat{g}', \hat{\eta})$, where $\hat{g}'$ is a lift of $g'$ to a functor $\hat{g}' \colon K \to \scC$, and where $\hat{\eta}$ is an equivalence $g \eq \hat{g}'$ such that $\iota^*\hat{\eta} = \eta$.
This space is a contractible Kan complex.
\end{myenumerate}
\end{lemma}

\begin{proof}
Part~(1) follows since $\obj(K) \hookrightarrow K$ is a cofibration in the Joyal model category $\sSet_J$, $\scC$ is a fibrant object in $\sSet_J$, and $\sSet_J$ is a (closed) symmetric monoidal model category.

For part~(2), we apply~\cite[Cor.~3.6.4]{Cisinski:HCats_HoAlg} to the categorical anodyne extension $\Delta^{\{i\}} \hookrightarrow J = \Delta^1[f^{-1}]$ and the Joyal fibration (i.e.~isofibration) from part~(1).

Part~(3) is obtained by taking the fibre of the morphism from part~(2), which is a contractible Kan complex since it is the fibre of a trivial Kan fibration.
This fibre is equivalently described as the space of lifts in the commutative diagram
\begin{equation}
\begin{tikzcd}
	\Delta^{\{i\}} \ar[d, hookrightarrow] \ar[r, "g"]
	& \scC^K \ar[d, "\iota^*"]
	\\
	J \ar[r, "\eta"'] & \scC^{\obj(K)}
\end{tikzcd}
\end{equation}
which is precisely the space of pairs $(\hat{g}',\hat{\eta})$ of lifts $\hat{g}' \colon K \to \scC$ of $g'$ and equivalences $ \hat{\eta} \colon g \eq \hat{g}'$ such that $\iota^*\hat{\eta} = \eta$.
\end{proof}

\begin{example}
\label{eg:G//G and Dec^0}
Let $G$ be a group object in an $\infty$-category $\scC$ with a final object.
This is, in particular, a simplicial object $G \colon \bbDelta^\opp \to \scC$ (we suppress the canonical inclusion functors $\Grp(\scC) \hookrightarrow \Gpd(\scC) \hookrightarrow \Fun(\bbDelta^\opp, \scC)$).
Consider the functor
\begin{equation}
	[0] \star (-) \colon \bbDelta \longrightarrow \bbDelta\,,
	\qquad
	[n] \longmapsto [0] \star [n] \cong [n+1]\,,
\end{equation}
where $\star$ denotes the join of categories (and where we view partially ordered sets as categories).
The induced pullback functor
\begin{equation}
	\Dec^0 \coloneqq \big( [0] \star (-) \big)^* \colon \Fun(\bbDelta^\opp, \scC) \longrightarrow \Fun(\bbDelta^\opp, \scC)
\end{equation}
is also called the \emph{decalage} functor; see~\cite{Stevenson:Decalage_and_spl_loop_grps} for more background.
For each $n \geq 1$, the partition $[n] = \{0,1\} \sqcup_{\{1\}} \{1, \ldots, n\}$ induces an equivalence
\begin{equation}
\label{eq:gamma_n for G//G}
	\gamma_n \colon (\Dec^0 G)_n = G_{n+1}
	\simeq G_1 \times G_n\,.
\end{equation}
We can phrase this as an equivalence of functors $\Dec^0 G \simeq G_1 \times \iota^*G \colon \obj(\bbDelta^\opp) \to \scC$.
From Lemma~\ref{st:equiv on objs pulls back functoriality} we obtain that there exists an essentially unique way to lift these data to a functor $\bbDelta^\opp \to \scC$, which we denote by $G_1 \dslash G$, and an equivalence $\gamma \colon \Dec^0 G \eq G_1 \dslash G$ in $\Fun(\bbDelta^\opp, \scC)$, whose components are exactly the equivalences $\gamma_n$ from~\eqref{eq:gamma_n for G//G}.
One can now check that $G_1 \dslash G$ is the simplicial object in $\scC$ that describes the right action of $G$ on itself via the group multiplication in $G$.
\qen
\end{example}

\begin{definition}
\label{def:oo-group action}
Let $\scC$ be an $\infty$-category with pullbacks and a final object.
Let $G$ be a group object in $\scC$, and let $P \in \scC$.
An \emph{action of $G$ on $P$} is a simplicial object $P \dslash G \in \Fun(\bbDelta^\opp, \scC)$ such that
\begin{myenumerate}
\item for each $n \in \NN_0$, we have $(P \dslash G)_n = P \times G_1^n$,

\item the morphism $d_1 \colon P \times G_1 \to P$ is the canonical projection onto $P$, the morphism $s_0 \colon P \to P \times G_1$ agrees with the morphism \smash{$1_P \times (* \overset{s_0}{\to} G_1)$}, and

\item the collapse morphism $P \to *$ induces a morphism $P \dslash G \to G$ in $\Fun(\bbDelta^\opp, \scC)$.
\end{myenumerate}
\end{definition}

Given a group action $P \dslash G$, we set $a \coloneqq d_0 \colon P \times G_1 \to P$.
It follows by the pasting law for pullbacks that there are canonical equivalences of morphisms between $d_0 \colon P \times G_1^n \to P \times G_1^{n-1}$ and \smash{$a \times 1_{G_1^{n-1}} \colon P \times G_1^n \to P \times G_1^{n-1}$}, and similarly between $d_n \colon P \times G_1^n \to P \times G_1^{n-1}$ and the projection onto the first $n$ factors.

\begin{remark}
Definition~\ref{def:oo-group action} is taken from~\cite[Def.~3.1]{NSS:oo-bundles} almost verbatim, but it differs from that source in that we do not \emph{require} group actions to be groupoid objects.
Instead, we show in Theorem~\ref{st:grp actions are gpd obs} that this is a \emph{consequence} of the axioms in Definition~\ref{def:oo-group action}.
A second (minor) difference is that we also fix the level-zero degeneracy map $s_0 \colon P \to P \times G_1$.
\qen
\end{remark}

\begin{example}
For any group object $G \in \Grp(\scC)$ there is a canonical trivial action $* \dslash G$ on the final object $* \in \scC$, coming from the canonical equivalence $* \times G \simeq G$ of simplicial objects.
That is, there is a canonical equivalence $G \simeq * \dslash G$ in $\Fun(\bbDelta^\opp, \scC)$.
\qen
\end{example}

\begin{example}
We can now give a precise meaning to the last sentence of Example~\ref{eg:G//G and Dec^0}:
the object $G_1 \dslash G \in \Fun(\bbDelta^\opp, \scC)$ is an action of $G$ on its underlying object $G_1 \in \bH$ via right multiplication;
see also~\cite[Def.~4.2.2.2, Example~4.2.2.4]{Lurie:HA} for more background.
\qen
\end{example}

Given an action of a group object $G$ on an object $P$ in $\scC$, we would like to think of the simplicial object $P \dslash G$ as the \emph{action groupoid} associated with this action.
This is indeed justified:

\begin{theorem}
\label{st:grp actions are gpd obs}
Let $\scC$ be an $\infty$-category with finite limits, let $G \in \Grp(\scC)$ be a group object in $\scC$, and let $P \dslash G \in \Fun(\bbDelta^\opp,\scC)$ be an action of $G$ on an object $P \in \scC$.
Then, $P \dslash G$ is a groupoid object in $\scC$.
\end{theorem}

\begin{remark}
Let $\rmL \colon \bH \to \bH'$ be a functor between $\infty$-topoi which preserves finite products and geometric realisations.
The relevance of Theorem~\ref{st:grp actions are gpd obs} is that it will allow us to show that functors of this type map group actions in $\bH$ to group actions in $\bH'$ (see Theorem~\ref{st:L pres (princ) actions}).
In~\cite{NSS:oo-bundles}, group actions are \emph{defined} to be groupoid objects, but functors $\rmL \colon \bH \to \bH'$ as above do not preserve groupoid objects in general.
However, Theorem~\ref{st:grp actions are gpd obs} shows that---as in the classical case of (set-theoretic) group actions---actions of group objects in $\infty$-topoi \emph{automatically} form groupoid objects.
Consequently, we do not need to require $\rmL$ to preserve groupoid objects.
\qen
\end{remark}

We prove Theorem~\ref{st:grp actions are gpd obs} in Appendix~\ref{app:Actions and CatObs}.
For the remainder of this section, let $\bH$ be an $\infty$-topos.

\begin{definition}
\label{def:G-action over X}
Let $G \in \Grp(\bH)$ be a group object.
A \emph{$G$-action over an object $B \in \bH$} is an augmented simplicial object $X \in \Fun(\bbDelta_+^\opp, \bH)$ whose underlying simplicial object is a $G$-action $P \dslash G$ on some object $P \in \bH$, and whose augmenting object is $B$, i.e.~$X_{-1} = B$.
Writing $p \colon P \to B$ for the morphism \smash{$X_{|\bbDelta_{+, \leq 0}^\opp}$}, we also denote a $G$-action over $B$ by
\begin{equation}
	P \dslash G \overset{p}{\longrightarrow} B
	\quad
	\in \Fun( \bbDelta_+^\opp, \bH)\,.
\end{equation}
A \emph{morphism of $G$-actions over $B \in \bH$},
\begin{equation}
	( P \dslash G \to B) \overset{f}{\longrightarrow} (Q \dslash G \to B)\,,
\end{equation}
is a morphism $f$ in $\Fun(\bbDelta_+^\opp, \bH)$ as above such that
\begin{myenumerate}
\item $f_{-1} = 1_B$ is the identity on $B$, and

\item the collapse morphisms $P \to *$ and $Q \to *$ induce a (weakly) commutative diagram
\begin{equation}
\begin{tikzcd}[column sep=0.5cm]
	P \dslash G \ar[rr, "f_{|\bbDelta^\opp}"] \ar[dr] & & Q \dslash G \ar[dl]
	\\
	& * \dslash G
\end{tikzcd}
\end{equation}
of simplicial objects in $\bH$.
\end{myenumerate}
The \emph{$\infty$-category of $G$-actions over $B \in \bH$} is the full $\infty$-subcategory of \smash{$\Fun(\bbDelta^\opp, \bH)_{/(B \times (* \dslash G))}$} on those objects whose underlying simplicial object is a $G$-action.
\end{definition}

An ordinary $G$-action is equivalent to a $G$-action over the final object $* \in \bH$.

\begin{example}
For a group object $G \in \Grp(\bH)$ and an action $P \dslash G$ of $G$ on an object $P \in \bH$, let $q \colon (\bbDelta^\opp)^\triangleright \to \bH$ be a colimiting cocone of the simplicial diagram $P \dslash G$ in $\bH$.
Observing that $(\bbDelta^\opp)^\triangleright \cong (\bbDelta^\triangleleft)^\opp \cong (\bbDelta_+)^\opp$, this defines an augmented simplicial object in $\bH$, which we denote as
\begin{equation}
	q \colon P \dslash G \longrightarrow \underset{\bbDelta^\opp}{\colim}^\bH (P \dslash G) = |P \dslash G|\,.
\end{equation}
Therefore, the data $P \dslash G \to |P \dslash G|$ form a $G$-action over $|P \dslash G|$.
In particular, the canonical morphism $* \dslash G \to \rmB G$ is of this form.
\qen
\end{example}

Another example of a morphism of this type is the collapse morphism $G_1 \dslash G \to *$, as we show now:

\begin{proposition}
\label{st:|G//G| = *}
If $G$ is a group object in $\bH$, then the canonical morphism
\begin{equation}
	|G_1 \dslash G| \eq *
\end{equation}
is an equivalence.
\end{proposition}

\begin{proof}
Since $\bH$ is presentable, there exists a combinatorial simplicial model category $\scM$ and an equivalence of $\infty$-categories $\bH \simeq N_\Delta (\scM^\circ)$~\cite[Prop.~A.3.7.6]{Lurie:HTT}.
Under this equivalence, colimits in $\bH$ over diagrams indexed by ordinary categories correspond to homotopy colimits in $\scM$~\cite[Cor.~4.2.4.8]{Lurie:HTT}.
It now suffices to observe that any simplicial object in $\scM$ obtained as the decalage of another simplicial object has an augmentation and extra degeneracies~\cite{Riehl:Cat_HoThy,Stevenson:Decalage_and_spl_loop_grps}.
\end{proof}

Any morphism $A \to G$ of group objects induces an action of $A$ on $G_1$ by the following construction:

\begin{proposition}
\label{st:G//A is action}
Let $f \colon A \to G$ be a morphism in $\Grp(\bH)$.
Define a simplicial object $G_1 \dslash A$ as the pullback
\begin{equation}
\label{eq:def G//A}
\begin{tikzcd}
	G_1 \dslash A \ar[r] \ar[d] & G_1 \dslash G \ar[d]
	\\
	* \dslash A \ar[r] & * \dslash G
\end{tikzcd}
\end{equation}
in $\Fun(\bbDelta^\opp, \bH)$.
Then, $G_1 \dslash A$ is an action of $A$ on $G_1$.
\end{proposition}

\begin{proof}
We check the axioms in Definition~\ref{def:oo-group action}:
axiom~(1) follows from the pasting law for pullbacks and the diagram
\begin{equation}
\begin{tikzcd}
	(G_1 \dslash A)_n \ar[r] \ar[d]
	& (G_1 \dslash G)_n \ar[d] \ar[r]
	& G_1 \ar[d]
	\\
	(* \dslash A)_n \ar[r]
	& (* \dslash G)_n \ar[r]
	& *
\end{tikzcd}
\end{equation}
in which the right-hand square is a pullback for any $n \in \NN_0$ by construction of $G_1 \dslash G$.

Axiom~(2) is readily seen from applying the maps $d_1$ and $s_0$ to the diagram~\eqref{eq:def G//A}, for $n = 0,1$.
Axiom~(3) follows since the morphism $G_1 \dslash A \longrightarrow * \dslash A$ induced by the above diagram agrees with the morphism obtained by collapsing the first factor $G_1$.
\end{proof}

\subsection{Principal $\infty$-bundles}
\label{sec:PFBuns}

In this subsection, we characterise principal $\infty$-bundles and group extensions in $\infty$-topoi.
Throughout this section, let $\bH$ be an $\infty$-topos and let $G \in \Grp(\bH)$ be a group object in $\bH$.

\begin{definition}
\label{def:principal oo-bundle}
\emph{\cite[Def.~3.4]{NSS:oo-bundles}}
A \emph{$G$-principal $\infty$-bundle on an object $B \in \bH$} is a $G$-action $P \dslash G \to B$ over $B$ such that the augmented simplicial object $P \dslash G \to B$ is a colimiting cocone for the simplicial diagram $P \dslash G \in \Fun(\bbDelta^\opp, \bH)$.
In other words, the augmenting map $p \colon P \to B$ induces an equivalence $\colim^\bH_{\bbDelta^\opp}(P \dslash G) \eq B$ in $\bH$.

A \emph{morphism of $G$-principal $\infty$-bundles on $B$}, denoted $(P \dslash G \to B) \longrightarrow (Q \dslash G \to B)$, is a morphism of the underlying $G$-actions over $B$.
The $\infty$-category $\Bun_G(B)$ of $G$-principal $\infty$-bundles over $B$ is the full $\infty$-subcategory of \smash{$\Fun(\bbDelta^\opp, \bH)_{/(B \times (* \dslash G))}$} (cf.~Definition~\ref{def:G-action over X}) on the $G$-principal $\infty$-bundles on $B$.
\end{definition}

\begin{example}
Let $G \in \Grp(\bH)$.
For any $G$-action $P \dslash G$ in $\bH$, the morphism $P \dslash G \to |P \dslash G|$ is a principal $G$-bundle in $\bH$ over $|P \dslash G|$.
As concrete examples of this type, we have already seen that $G_1 \dslash G$ exhibits $G_1$ as a principal $G$-bundle over $* \in \bH$ (Proposition~\ref{st:|G//G| = *}), and that $* \dslash G$ exhibits $*$ as a principal $G$-bundle over $\rmB G$ (by the definition of $\rmB G$).
\qen
\end{example}

We now provide an alternative characterisation of principal $\infty$-bundles in $\infty$-topoi.
Let $G \in \Grp(\bH)$ be a group object in $\bH$, and let $p \colon P \dslash G \to B$ be a $G$-action over an object $B \in \bH$.
Let \smash{$\imath \colon \bbDelta_{+, \leq 0}^\opp \hookrightarrow \bbDelta_+^\opp$} be the inclusion.
The identity provides a canonical equivalence
\begin{equation}
	\eta \colon \{p\} = \imath^*(P \dslash G \to B) \eq \imath^*(\cC p) = \{p\}
\end{equation}
in $\Fun(\bbDelta_{+, \leq 0}^\opp, \bH) \simeq \Fun(\Delta^1, \bH)$.
Since right Kan extension is a right adjoint, there is an equivalence
\begin{equation}
	\ul{\Fun}(\bbDelta_{+, \leq 0}^\opp, \bH) \big( \imath^*(P \dslash G \to B), \{p\} \big)
	\simeq \ul{\Fun}(\bbDelta_+^\opp, \bH) \big( (P \dslash G \to B), \cC p \big)
\end{equation}
of mapping spaces (compare also~\eqref{eq:comp mp gpd v Cech nerve}).
We denote the image of $\eta$ under this equivalence by
\begin{equation}
	\alpha \colon (P \dslash G \to B) \longrightarrow \cC p\,.
\end{equation}
Observe that, by construction, the restriction of $\alpha$ along $\imath$ is $\eta$.
We will not distinguish notationally between $\alpha$ as defined here and its restriction along the inclusion $\bbDelta^\opp \subset \bbDelta_+^\opp$ (since $\alpha_{-1} = 1_B$).

\begin{definition}
\label{def: principal action}
A $G$-action $P \dslash G \longrightarrow B$ over $B \in \bH$ is called \emph{principal} if the canonical morphism $\alpha \colon P \dslash G \longrightarrow \cC p$ is an equivalence in $\Fun(\bbDelta^\opp, \bH)$.
\end{definition}

This is an $\infty$-categorical version of the principality condition for a group action.
It is, in fact, equivalent to the usual principality condition---that the action morphism $P \times G_1 \to P \times_B P$ is an equivalence---in the following sense (in particular, this implies the converse to~\cite[Prop.~3.7]{NSS:oo-bundles}):

\begin{lemma}
\label{st:principality condition}
Let $P \dslash G \to B$ be a $G$-action over $B \in \bH$.
The following are equivalent:
\begin{myenumerate}
\item The $G$-action is principal.

\item The diagram
\begin{equation}
\label{eq:principality at level 1}
\begin{tikzcd}[column sep=1.5cm, row sep=1cm]
	P \times G_1 \ar[r, "d_1 = \pr_P"] \ar[d, "a = d_0"'] & P \ar[d, "p"]
	\\
	P \ar[r, "p"'] & B
\end{tikzcd}
\end{equation}
is a pullback diagram in $\bH$.
\end{myenumerate}
\end{lemma}

\begin{proof}
(1) implies~(2) since the action $P \dslash G \overset{p}{\longrightarrow} B$ is principal precisely if it is equivalent, as an augmented simplicial object in $\bH$, to the \v{C}ech nerve $\cC p = \Ran_\iota \{p\}$.
Thus, the implication follows from Proposition~\ref{st:gpd obs via X_1 X_1 X_0}.

Conversely, (2) also implies~(1):
we know from Theorem~\ref{st:grp actions are gpd obs} that $P \dslash G$ is a groupoid object.
If we additionally have that~\eqref{eq:principality at level 1} is a pullback diagram, then we can again apply Proposition~\ref{st:gpd obs via X_1 X_1 X_0} to obtain the claim.
\end{proof}

We can use Lemma~\ref{st:principality condition} to give a characterisation of principal $\infty$-bundles which can be understood as encoding directly the classical criteria for principal bundles:
a locally trivial map $p \colon P \to B$ and a principal $G$-action over $B$.

\begin{proposition}
\label{st:pfbun characterisation}
Let $P \dslash G \xrightarrow{p} B$ be a $G$-action over an object $B \in \bH$.
The following are equivalent:
\begin{myenumerate}
\item $P \dslash G \overset{p}{\longrightarrow} B$ is a principal $\infty$-bundle (in the sense of Definition~\ref{def:principal oo-bundle}).

\item The morphism $p$ is an effective epimorphism and the action $P \dslash G$ is principal.
\end{myenumerate}
\end{proposition}

\begin{proof}
To see that (1) implies (2), first observe that since $P \dslash G$ is a groupoid object in $\bH$, and since by assumption the canonical morphism $|P \dslash G| \to B$ is an equivalence, it follows from Definition~\ref{def:oo-topos}(4) that the canonical morphism $\alpha \colon P \dslash G \to \cC p$ is an equivalence in $\Fun(\bbDelta^\opp, \bH)$.
In particular, $p$ is an effective epimorphism.
Further, it has been shown in~\cite[Prop.~3.7]{NSS:oo-bundles} that if $P \dslash G \to B$ is a principal bundle, then the action $P \dslash G$ satisfies condition~(2) of Lemma~\ref{st:principality condition}, and so the action is principal.

To see the other direction, consider the commutative diagram
\begin{equation}
\label{eq:pfbun via principality}
\begin{tikzcd}
	{|P \dslash G|} \ar[rr, "|\alpha|"] \ar[dr] & & {|\cC p|} \ar[dl]
	\\
	& B &
\end{tikzcd}
\end{equation}
In this case, both the top and the right-hand morphisms in diagram are equivalences.
It thus follows that also the left-hand morphism is an equivalence, which amounts to the fact that $P \dslash G \to B$ is a principal $G$-bundle in the sense of Definition~\ref{def:principal oo-bundle}.
\end{proof}

\begin{theorem}
\label{st:L pres (princ) actions}
Let $\rmL \colon \bH \to \bH'$ be a functor between $\infty$-topoi which preserves geometric realisations and finite products.
Suppose $G$ is a group object in $\bH$.
\begin{myenumerate}
\item $\rmL$ maps $G$-actions $P \dslash G \xrightarrow{p} B$ over $B \in \bH$ to $\rmL G$-actions $\rmL P \dslash \rmL G \xrightarrow{\rmL p} \rmL B$ over $\rmL B \in \bH'$.

\item If the action $P \dslash G \to B$ is a principal $G$-bundle, then the action $\rmL P \dslash \rmL G \xrightarrow{\rmL p} \rmL B$ is a principal $\rmL G$-bundle.
\end{myenumerate}
\end{theorem}

\begin{proof}
Since $\rmL$ preserves finite products, the first claim follows readily from Definition~\ref{def:oo-group action}.

For the second claim, recall that $P \dslash G \to B$ is a principal $G$-bundle precisely if the map $|P \dslash G| \to B$ is an equivalence.
Applying the functor $\rmL$ to this morphism, we obtain an equivalence \smash{$\rmL |P \dslash G| \eq \rmL B$}.
Since $\rmL$ preserves geometric realisations, and using claim~(1), we obtain further canonical equivalences
\begin{equation}
	|\rmL P \dslash \rmL G| \eq \rmL |P \dslash G| \eq \rmL B\,,
\end{equation}
which establishes the action $\rmL P \dslash \rmL G \xrightarrow{\rmL p} \rmL B$ as a principal $\rmL G$-bundle over $B$.
\end{proof}

\begin{proposition}
\label{st:pb bun construction}
Let $G$ be a group object in $\bH$, and let $P \dslash G \to C$ be a $G$-principal $\infty$-bundle in $\bH$.
For any morphism $f \colon B \to C$ in $\bH$, there is a canonical $G$-action over $B$ on the pullback $Q \coloneqq B \times_C P$ that makes $Q \dslash G \to B$ into a $G$-principal $\infty$-bundle on $B$.
\end{proposition}

\begin{proof}
Let $\sfc \colon \bH \longrightarrow \Fun(\bbDelta^\opp, \bH)$ be the constant-diagram functor.
Consider the pullback diagram
\begin{equation}
\label{eq:pb bundle from pb in sH}
\begin{tikzcd}[row sep=1.25cm, column sep=1.5cm]
	\sfc B \times_{\sfc C} (P \dslash G) \ar[r, "\hat{f}"] \ar[d, "f^*p"']
	& P \dslash G \ar[d, "p"]
	\\
	\sfc B \ar[r, "\sfc f"']
	& \sfc C
\end{tikzcd}
\end{equation}
in $\Fun(\bbDelta^\opp,\bH)$ (or, equivalently, in $\Gpd(\bH)$).
For any $[n] \in \bbDelta$ there exists a canonical equivalence
\begin{equation}
	\big( \sfc B \times_{\sfc C} (P \dslash G) \big)_n
	\simeq B \times_C (P \times G_1^{n-1})
	\simeq (B \times_C P) \times G_1^{n-1}\,.
\end{equation}
We use Lemma~\ref{st:equiv on objs pulls back functoriality} to obtain from these equivalences a canonical pair (up to contractible choices) of an object $(B \times_C P) \dslash G \in \Fun(\bbDelta^\opp,\bH)$, with $((B \times_C P) \dslash G)_n = (B \times_C P) \times G_1^n$ for all $n \in \NN_0$, together with an equivalence
\begin{equation}
\label{eq:action in pb bundle}
	(B \times_C P) \dslash G \eq \sfc B \times_{\sfc C} (P \dslash G)
\end{equation}
of simplicial objects in $\bH$.
By a slight abuse of notation, we also denote the composition
\begin{equation}
	(B \times_C P) \dslash G
	\eq \sfc B \times_{\sfc C} (P \dslash G)
	\longrightarrow \sfc B
\end{equation}
by $f^*p$.
It follows by construction that $(B \times_C P) \dslash G \overset{f^*p}{\longrightarrow} B$ is a $G$-action over $B$.
We are hence left to show that it is a principal $\infty$-bundle.

To that end, we will show that the morphism
\begin{equation}
	|(B \times_C P) \dslash G| \longrightarrow B
\end{equation}
is an equivalence.
Diagram~\eqref{eq:pb bundle from pb in sH} is a diagram of the form $\Delta^1 \times \Delta^1 \longrightarrow \Fun(\bbDelta^\opp, \bH)$.
Composing with the functor $\colim_{\bbDelta^\opp}^\bH = |{-}| \colon \Fun(\bbDelta^\opp, \bH) \longrightarrow\bH$ we obtain a diagram
\begin{equation}
\label{eq:realisation of pb bundle}
\begin{tikzcd}[row sep=1.25cm, column sep=1.5cm]
	{|(B \times_C P) \dslash G|} \ar[r, "|\hat{f}|"] \ar[d, "f^*p"']
	& {|P \dslash G|} \ar[d, "p", "\simeq"']
	\\
	B \ar[r, "f"']
	& C
\end{tikzcd}
\end{equation}
in $\bH$.
The right-hand morphism is an equivalence since $P \dslash G \to C$ is assumed to be a principal $\infty$-bundle.
Using the equivalence~\eqref{eq:action in pb bundle}, diagram~\eqref{eq:realisation of pb bundle} is equivalent to the diagram
\begin{equation}
\label{eq:geo rel of pb bundle}
\begin{tikzcd}[row sep=1.25cm, column sep=1.5cm]
	{|\sfc B \times_{\sfc C} (P \dslash G)|} \ar[r, "|\hat{f}|"] \ar[d, "f^*p"']
	& {|P \dslash G|} \ar[d, "p", "\simeq"']
	\\
	B \ar[r, "f"']
	& C
\end{tikzcd}
\end{equation}
By the universality of colimits in $\bH$, we have a canonical equivalence
\begin{equation}
	|\sfc B \times_{\sfc C} (P \dslash G)|
	\simeq B \times_C |P \dslash G|\,.
\end{equation}
This establishes that the morphism $f^*p$ in diagram~\eqref{eq:geo rel of pb bundle} is the pullback of an equivalence in $\bH$, and hence that $f^*p$ is an equivalence itself.
\end{proof}

One can now show that every $G$-principal $\infty$-bundle arises as a pullback of the bundle $(* \dslash G) \to \rmB G$.
This insight is not new, but has been observed in~\cite[Prop.~3.13, Thm.~3.17]{NSS:oo-bundles} already.
However, in Section~\ref{sec:String Model} it will be important to have a good understanding of the classifying map of a principal $\infty$-bundle, and so we include a brief treatment of these maps.
We start with two short technical lemmas, before constructing for each $G$-principal $\infty$-bundle in $\bH$ its classifying map.

\begin{lemma}
\label{st:G//G as pb Bun}
Let $G$ be a group object in $\bH$, and let $f \colon * \to \rmB G$ be the base point of $\rmB G$.
The pullback of the canonical bundle $(* \dslash G) \to \rmB G$ along $f$ agrees with the bundle $G_1 \dslash G \to *$.
\end{lemma}

\begin{proof}
Consider the commutative square
\begin{equation}
\begin{tikzcd}
	G_1 \dslash G \ar[r] \ar[d] & * \dslash G \ar[d]
	\\
	* \ar[r] & \sfc\rmB G
\end{tikzcd}
\end{equation}
in $\Fun(\bbDelta^\opp, \bH)$.
By the canonical equivalence $G \simeq \Omega \rmB G$ in $\Grp(\bH)$ (see~\eqref{eq:Omega -| B adjoint equiv}), this diagram is level-wise a pullback, i.e.~it is a pullback diagram in $\Fun(\bbDelta^\opp, \bH)$.
That proves the claim by Proposition~\ref{st:pb bun construction}.
\end{proof}

\begin{definition}
\label{def:trivial bundle}
A $G$-principal $\infty$-bundle $P \dslash G \to B$ is \emph{trivial} if it is equivalent in $\Bun_G(B)$ to the \emph{trivial $G$-principal $\infty$-bundle} $B \times (G_1 \dslash G)  \to B$, i.e.~if there is an equivalence of simplicial objects in $\bH_{/B}$ between $P \dslash G$ and $B \times (G \dslash G)$ that commutes with the canonical morphisms to $* \dslash G$.
\end{definition}

\begin{lemma}
\label{st:trivial bundle over pb}
\emph{\cite[Prop.~3.12]{NSS:oo-bundles}}
For every $G$-principal $\infty$-bundle $P \dslash G \to B$ in $\bH$, there exists an effective epimorphism $U \to B$ such that the pullback bundle $U \times_B (P \dslash G)$ is trivial.
\end{lemma}

\begin{proof}
We give an alternative proof to~\cite{NSS:oo-bundles}.
Given a $G$-principal $\infty$-bundle $P \dslash G \to B$ in $\bH$, consider the effective epimorphism $P \to B$ and the pullback bundle $P \times_B (P \dslash G)$.
We have a commutative diagram
\begin{equation}
\begin{tikzcd}[column sep=1cm, row sep=1cm]
	P \times (G_1 \dslash G) \ar[ddr, bend right=25, "\pr"'] \ar[rrd, bend right=-25, "a \times 1"] \ar[dr, dashed, "\psi" description]& &
	\\
	& P \times_B (P \dslash G) \ar[r] \ar[d] & P \dslash G \ar[d]
	\\
	& \sfc P \ar[r] & \sfc B
\end{tikzcd}
\end{equation}
in $\Fun(\bbDelta^\opp, \bH)$, where $a \times 1$ acts on $P$ with the first copy of $G$ and as the identity on the remaining copies of $G$.
The induced morphism $\psi$ is a morphism of $G$-principal $\infty$-bundles (since the triangles in the diagram commute and since $a \times 1$ is a morphism of $G$-actions).
It is thus equivalent to a morphism
\begin{equation}
	\psi' \colon \cC(P \times G \to P) \longrightarrow \cC \big( (P \times_B (P \dslash G)) \longrightarrow P \big)
\end{equation}
of \v{C}ech nerves over $P$.
The level-zero component of $\psi'$ is precisely the equivalence $P \times G_1 \to P \times_B P$ which establishes that $P \dslash G \to B$ is principal (cf.~Proposition~\ref{st:principality condition}).
Since $\psi'$ is the image of $\psi$ under the right Kan extension $\Ran_\iota$ (compare Definition~\ref{def:Cech nerve}), it follows that $\psi'$, and hence $\psi$, is an equivalence.
\end{proof}

\begin{proposition}
\label{st:each G-Bun is pb bun}
For every $G$-principal $\infty$-bundle $P \dslash G \to B$ in $\bH$, the diagram
\begin{equation}
\label{eq:P//G as pb}
\begin{tikzcd}
	P \dslash G \ar[r, "p"] \ar[d] & * \dslash G \ar[d]
	\\
	\sfc B \ar[r, "|p|"'] & \sfc \rmB G
\end{tikzcd}
\end{equation}
is a pullback diagram in $\Fun(\bbDelta^\opp, \bH)$:
there is an equivalence $(P \dslash G \to B) \simeq B \times_{\rmB G} (* \dslash G)$ of $G$-principal $\infty$-bundles over $X$.
In particular, every $G$-principal $\infty$-bundle is a pullback of the bundle $* \dslash G \to \rmB G$.
\end{proposition}

This is a refinement of~\cite[Prop.~3.13]{NSS:oo-bundles} to a statement on the level of simplicial objects, rather than only on their zeroth level.

\begin{proof}
Consider the diagram
\begin{equation}
\label{eq:arb PBun as pb}
\begin{tikzcd}[column sep={1.75cm,between origins}, row sep={1.25cm,between origins}]
	& G_1 \dslash G \ar[dd] \ar[rr]
	& & * \dslash G \ar[dd]
	\\
	 P \times (G_1 \dslash G) \ar[dd] \ar[rr, crossing over] \ar[ur]
	& & P \dslash G \ar[ur]
	&
	\\
	& * \ar[rr]
	& & \sfc \rmB G
	\\
	\sfc P \ar[rr] \ar[ur]
	& & \sfc B \ar[from=uu, crossing over] \ar[ur] &
\end{tikzcd}
\end{equation}
in $\Fun(\bbDelta^\opp, \bH)$.
Here, the front and rear squares are pullbacks (by Lemmas~\ref{st:trivial bundle over pb} and \ref{st:G//G as pb Bun}, respectively), and the diagram is obtained as a morphism of pullback diagrams.
We need to show that the right-hand face is a pullback square in $\Fun(\bbDelta^\opp, \bH)$.

First, we show that the top square of~\eqref{eq:arb PBun as pb} is a pullback.
By Lemma~\ref{st:lims in Gpd and Grp objects} it suffices to check this level-wise:
at simplicial level $n= 0$, it is trivial.
For $n \in \NN$, the square consists of the the image under the functor $(-) \times G_1^{n-1}$ of the diagram
\begin{equation}
\label{eq:modified product PxG}
\begin{tikzcd}
	P \times G_1 \ar[r, "a"] \ar[d, "\pr_1"'] & P \ar[d]
	\\
	G_1 \ar[r] & *
\end{tikzcd}
\end{equation}
This a pullback diagram:
there is a commutative diagram
\begin{equation}
\begin{tikzcd}
	P \times G_1 \ar[ddr, bend right=20, "\pr_1"'] \ar[rrd, bend right=-20, "\pr_0"] \ar[dr, dashed, "g" description]& &
	\\
	& P \times G_1 \ar[r, "a"] \ar[d, "\pr_1"'] & P \ar[d]
	\\
	& G_1 \ar[r] & *
\end{tikzcd}
\end{equation}
in which the dashed morphism is given by the composition
\begin{equation}
	g = (a \times 1_{G_1}) \circ (1_P \times \inv \times 1_{G_1}) \circ (1_P \times \Delta_{G_1})\,,
\end{equation}
where $\Delta_{G_1} \colon G_1 \to G_1^2$ is the diagonal morphism, and where $\inv \colon G_1 \to G_1$ is the choice of an inverse in $G$:
since the group object $G \in \Grp(\bH)$ is in particular a groupoid object, we have a diagram
\begin{equation}
\begin{tikzcd}
	G_1 \simeq G_1 \times * \ar[r]
	& G_1 \times G_1 \simeq G(\Lambda^2_0)
	\ar[from=r, "\simeq"']
	& G_2 \ar[r, "d_0"]
	& G_1\,,
\end{tikzcd}
\end{equation}
where we have used the characterisation of groupoid objects as certain category objects from Proposition~\ref{st:Grpd obs via horns}.
Choosing an inverse for the right-facing morphism defines the morphism $\inv$.

Since $g$ is an equivalence (because $G$ is a group object), diagram~\eqref{eq:modified product PxG} is a pullback in $\bH$, and since the span category $\{0,1\} \leftarrow \{0\} \rightarrow \{0,2\}$ has contractible nerve, the pullback~\eqref{eq:modified product PxG} is preserved by $(-) \times G_1^{n-1}$ (see Lemma~\ref{st:c is ff and Px(-) pres contr lims}).
We thus obtain that the top square in diagram~\eqref{eq:arb PBun as pb} is a pullback.

Next, we prove that the bottom square of~\eqref{eq:arb PBun as pb} is a pullback.
We define $C \coloneqq * \times_{\rmB G} B \in \bH$, and we consider the diagram (omitting constant-diagram functors)
\begin{equation}
\begin{tikzcd}[column sep={3cm,between origins}]
	C \times_B (P \dslash G) \ar[r] \ar[d] & C \ar[r] \ar[d] & * \ar[d]
	\\
	P \dslash G \ar[r] & B \ar[r] & \rmB G
\end{tikzcd}
\end{equation}
Both squares in this diagram are pullbacks in $\Fun(\bbDelta^\opp, \bH)$, so that the pasting law yields a canonical equivalence of simplicial objects
\begin{equation}
	C \underset{B}{\times} (P \dslash G) \simeq * \underset{\rmB G}{\times} (P \dslash G)\,.
\end{equation}
Observe that $C \times_B (P \dslash G) \to C$ is a $G$-principal $\infty$-bundle by Proposition~\ref{st:pb bun construction}, so that
\begin{equation}
	C \simeq \big| C \underset{B}{\times} (P \dslash G) \big|
	\simeq \big| * \underset{\rmB G}{\times} (P \dslash G) \big|\,.
\end{equation}
Now we use that the morphism $P \dslash G \to \rmB G$ factors through $* \dslash G$ (by Definitions~\ref{def:oo-group action} and~\ref{def:principal oo-bundle}) and that $* \times_{\rmB G} (* \dslash G) \simeq G_1 \dslash G$ (by Lemma~\ref{st:G//G as pb Bun}).
Applying the pasting law to the diagram
\begin{equation}
\begin{tikzcd}[column sep={3.75cm,between origins}]
	(P \dslash G) \underset{(* \dslash G)}{\times} (G_1 \dslash G) \ar[r] \ar[d] & G_1 \dslash G \ar[r] \ar[d] & * \ar[d]
	\\
	P \dslash G \ar[r] & * \dslash G \ar[r] & \rmB G
\end{tikzcd}
\end{equation}
in $\Fun(\bbDelta^\opp, \bH)$, in which both squares are pullbacks, we obtain a canonical equivalence
\begin{equation}
	* \underset{\rmB G}{\times} (P \dslash G)
	\simeq (P \dslash G) \underset{(* \dslash G)}{\times} (G_1 \dslash G)\,.
\end{equation}
The right-hand side is precisely the pullback described by the top square in diagram~\eqref{eq:arb PBun as pb}.
Since we already know that the top square of~\eqref{eq:arb PBun as pb} is cartesian, we obtain an equivalence
\begin{equation}
	(P \dslash G) \underset{(* \dslash G)}{\times} (G_1 \dslash G)
	\simeq P \times (G_1 \dslash G)
\end{equation}
in $\Fun(\bbDelta^\opp, \bH)$.
Thus, it follows that
\begin{equation}
	C \simeq \big| * \underset{\rmB G}{\times} (P \dslash G) \big|
	\simeq \big| P \times (G_1 \dslash G) \big|
	\simeq P\,.
\end{equation}
The last equivalence can be seen either by combining Proposition~\ref{st:|G//G| = *} with the fact that $|{-}|$ preserves finite products (because $\bbDelta^\opp$ is sifted~\cite[Lemma~5.5.8.4]{Lurie:HTT}), or simply by recalling that $P \times (G_1 \dslash G) \to P$ is a $G$-principal $\infty$-bundle on $P$.
This shows that the bottom square in~\eqref{eq:arb PBun as pb} is a pullback.

Finally, we prove that the right-hand square in~\eqref{eq:arb PBun as pb} is a pullback as well.
Consider the commutative diagram of solid arrows
\begin{equation}
\begin{tikzcd}
	P \dslash G \ar[ddr, bend right=20] \ar[rrd, bend right=-20] \ar[dr, dashed, "\varphi" description]& &
	\\
	& B \underset{\rmB G}{\times} (* \dslash G) \ar[r,] \ar[d] & * \dslash G \ar[d]
	\\
	& B \ar[r] & \rmB G
\end{tikzcd}
\end{equation}
which induces an essentially unique morphism $\varphi$ of simplicial objects in $\bH$.
By the commutativity of the right-hand triangle in this diagram, $\varphi$ is even a morphism of $G$-actions.
By the commutativity of the left-hand triangle it is even a morphism of $G$-actions over $B$.
Since its source and target are $G$-principal $\infty$-bundles, $\varphi$ is equivalent to a morphism of \v{C}ech nerves
\begin{equation}
	\varphi' \colon \cC(P \to B) \longrightarrow \cC \big( (B \underset{\rmB G}{\times} *) \to B \big)\,.
\end{equation}
That is, $\varphi'$ is the image under $\Ran_\iota$ (compare Definition~\ref{def:Cech nerve}) of the square
\begin{equation}
\begin{tikzcd}
	P \ar[r, "\varphi'_0"] \ar[d] & B \underset{\rmB G}{\times} * \ar[d]
	\\
	B \ar[r, equal] & B
\end{tikzcd}
\end{equation}
This $\varphi'_0$ is an equivalence since the bottom square of~\eqref{eq:arb PBun as pb} is a pullback.
Consequently, the morphism $\varphi$ is an equivalence in $\Fun(\bbDelta^\opp, \bH)$, and thus the right-hand face in~\eqref{eq:arb PBun as pb} is a pullback.
\end{proof}

\begin{corollary}
\label{st:fib pf P//G to X is G}
Let $P \dslash G \to B$ be a $G$-principal $\infty$-bundle in $\bH$.
For any morphism $x \colon * \to B$, we have a pullback diagram
\begin{equation}
\begin{tikzcd}
	G_1 \dslash G \ar[r] \ar[d] & P \dslash G \ar[d]
	\\
	* \ar[r, "x"'] & B
\end{tikzcd}
\end{equation}
in $\Fun(\bbDelta^\opp, \bH)$.
In particular, any fibre of $P \to B$ is canonically equivalent to $G_1$ in $\bH$.
\end{corollary}

\begin{remark}
\label{rmk:Bun_G(X) = H(X,BG)}
For any group object $G$ in $\bH$ and any object $B \in \bH$, there is an equivalence
\begin{equation}
	\Bun_G(B) \simeq \ul{\bH}(B, \rmB G)
\end{equation}
between the $\infty$-category of $G$-principal $\infty$-bundles on $B$ and the mapping space $\ul{\bH}(B, \rmB G)$~\cite[Thm.~3.17]{NSS:oo-bundles}.
This implies that every morphism of principal $G$-bundles on $B$ is an equivalence.
Proposition~\ref{st:each G-Bun is pb bun} feeds into the proof of this equivalence by showing that the functor $\ul{\bH}(B, \rmB G) \to \Bun_G(B)$, sending a morphism $B \to \rmB G$ to the principal $\infty$-bundle $B \times_{\rmB G} (* \dslash G)$, is essentially surjective.
\qen
\end{remark}

In particular, under the equivalence of Remark~\ref{rmk:Bun_G(X) = H(X,BG)}, the morphism $|p| \colon B \to \rmB G$ in diagram~\eqref{eq:P//G as pb} is a classifying morphism for the bundle $P \dslash G \to B$.

\begin{proposition}
\label{st:classifying morphisms under L}
Let $\rmL \colon \bH \to \bH'$ be a functor of $\infty$-topoi which preserves finite products and geometric realisations.
If $P \dslash G \to B$ is a $G$-principal $\infty$-bundle in $\bH$, classified (up to canonical equivalence) by a morphism $|p| \colon B \to \rmB G$, then the $\rmL G$-principal $\infty$-bundle $\rmL P \dslash \rmL G \longrightarrow \rmL B$ (compare Theorem~\ref{st:L pres (princ) actions}) in $\bH'$ is classified by the morphism $|\rmL p| \simeq \rmL |p|$.
\end{proposition}

\begin{proof}
Consider the commutative diagram
\begin{equation}
\begin{tikzcd}[column sep=0.5cm, row sep=0.5cm]
	& \rmL (P \dslash G) \ar[dd] \ar[rr, "\rmL p"]
	& & \rmL(* \dslash G) \ar[dd]
	\\
	\rmL P \dslash \rmL G \ar[dd] \ar[rr, crossing over, "q"' pos=0.25] \ar[ur, "\simeq"]
	& & * \dslash \rmL G \ar[ur, "\simeq"']
	&
	\\
	& \rmL |P \dslash G| \ar[rr, "\rmL |p|" pos=0.25]
	& & \rmL \rmB G
	\\
	{|\rmL P \dslash \rmL G|} \ar[rr, "|q|"'] \ar[ur, "\simeq"]
	& & \rmB \rmL G \ar[from=uu, crossing over] \ar[ur, "\simeq"'] &
\end{tikzcd}
\end{equation}
The morphism $q$ is the canonical morphism induced from the collapse morphism $\rmL P \to *$.
By Proposition~\ref{st:each G-Bun is pb bun}, the front face of this diagram is a pullback in $\bH'$, witnessing $|q|$ as the classifying morphism $\rmL X \to \rmB \rmL G$ of the bundle $\rmL P \dslash \rmL G \longrightarrow \rmL X$.
Since all diagonal morphisms are equivalences, the rear face of the diagram is a pullback as well, showing that $\rmL |p|$ is a classifying morphism for the $\rmL G$-principal $\infty$-bundle $\rmL (P \dslash G) \longrightarrow \rmL B$, which is equivalent to the bundle $\rmL P \dslash \rmL G \longrightarrow \rmL B$.
Finally, since the diagonal morphisms arise from the natural equivalences $\rmL \circ |{-}| \simeq |{-}| \circ \rmL$, it follows that $|q| \simeq |\rmL p|$.
\end{proof}

We now state several alternative characterisations of group extensions in $\infty$-topoi.
These clarify the relation between the original notion of an extension of group objects from Definition~\ref{def:oo-grp extension} and more direct categorifications of several perspectives on group extensions in $\Set$.
The last of these alternative characterisations will be important in Section~\ref{sec:String Model}.

\begin{theorem}
\label{st:oo-grp ext via pfbuns}
Let $\bH$ be an $\infty$-topos, and let \smash{$A \xrightarrow{\iota} G \xrightarrow{p} H$} be a sequence of morphisms in $\Grp(\bH)$.
The following are equivalent:
\begin{myenumerate}
\item \smash{$A \xrightarrow{\iota} G \xrightarrow{p} H$} is an extension of group objects in $\bH$, i.e.~the induced sequence $\rmB A \to \rmB G \to \rmB H$ is a fibre sequence in $\bH$ (see Definition~\ref{def:oo-grp extension}).

\item \smash{$A \xrightarrow{\iota} G \xrightarrow{p} H$} is a fibre sequence in $\Grp(\bH)$ and $G_1 \to H_1$ is an effective epimorphism in $\bH$.

\item \smash{$A_1 \xrightarrow{\iota_1} G_1 \xrightarrow{p_1} H_1$} is a fibre sequence in $\bH$ and $G_1 \to H_1$ is an effective epimorphism in $\bH$.

\item The morphism $p_1 \colon G_1 \to H_1$ together with the action $G_1 \dslash A$ of $A$ on $G_1$ induced by $\iota$ define a principal $A$-bundle over $H_1$.
\end{myenumerate}
\end{theorem}

\begin{proof}
$(1) \Rightarrow (3)$:
This implication was proven in~\cite{NSS:oo-bundles} already.
We import the proof for completeness:
consider the diagram
\begin{equation}
\label{eq:pasting contruction}
\begin{tikzcd}[column sep=1cm, row sep=0.75cm]
	A_1 \ar[r, "\iota_1'"] \ar[d] & G_1 \ar[d, "p_1'"] \ar[r] & * \ar[d] &
	\\
	* \ar[r] & H_1 \ar[r] \ar[d] & \rmB A \ar[d, "\rmB \iota"'] \ar[r] & * \ar[d]
	\\
	& * \ar[r] & \rmB G \ar[r, "\rmB p"'] & \rmB H
\end{tikzcd}
\end{equation}
in $\bH$.
Each square in diagram~\eqref{eq:pasting contruction} is a pullback square (this assumes~(1)).
It thus follows that the sequence
\begin{equation}
	A_1 \xrightarrow{\iota'_1} G_1 \xrightarrow{p'_1} H_1
\end{equation}
is a fibre sequence in $\bH$.
By this construction, the morphisms $\iota'$ and $p'$ coincide with the morphisms $\Omega \circ \rmB (\iota)$ and $\Omega \circ \rmB (p)$, respectively.
The equivalence~\eqref{eq:Omega -| B adjoint equiv} then yields that also \smash{$A_1 \xrightarrow{\iota_1} G_1 \xrightarrow{p_1} H_1$} is a fibre sequence in $\bH$.
Observe that each vertical morphism in diagram~\eqref{eq:pasting contruction} is an effective epimorphism since $* \to \rmB G$ is an effective epimorphism for every group object $G \in \Grp(\bH)$ and since effective epimorphisms are stable under pullback.
In particular, $p_1$ is an effective epimorphism.

$(3) \Leftrightarrow (2)$:
This follows from Lemma~\ref{st:lims in Gpd and Grp objects}.

$(3) \Rightarrow (4)$:
By Proposition~\ref{st:pfbun characterisation} it suffices to show that the action of $A$ on $G_1$ is principal (in the sense of Definition~\ref{def: principal action}).
We will make use of Lemma~\ref{st:principality condition}.
By assumption, the diagram
\begin{equation}
\begin{tikzcd}
	A_1 \ar[r, "\iota_1"] \ar[d]
	& G_1 \ar[d, "p_1"]
	\\
	* \ar[r]
	& H_1
\end{tikzcd}
\end{equation}
is a pullback diagram in $\bH$.
Let $m \colon G_1 \times G_1 \to G_1$ be the multiplication on $G_1$ (it can be identified with the morphism $d_1 \colon G_2 \to G_1$).
Further, let $\inv \colon G_1 \to G_1$ be a choice of inverse for the group object $G$ (compare the proof of Proposition~\ref{st:each G-Bun is pb bun}).
There is a commutative diagram
\begin{equation}
\begin{tikzcd}[column sep=1.25cm]
	G_1 \underset{H_1}{\times} G_1 \ar[r, "\inv \times 1"] \ar[d]
	& G_1 \times G_1 \ar[r, "m"]
	& G_1 \ar[d, "p_1"]
	\\
	* \ar[rr]
	& & H_1
\end{tikzcd}
\end{equation}
Using that $A_1 \simeq * \times_{H_1} G_1$, the universal property of pullbacks thus provides an essentially unique morphism \smash{$\varphi \colon G_1 \times_{H_1} G_1 \to A_1$}.
The morphisms
\begin{equation}
\begin{tikzcd}[column sep=1.5cm]
	G_1 \underset{H_1}{\times} G_1 \ar[r, shift left=0.1cm, "\pr_1 \times \varphi"]
	& G_1 \times A_1 \ar[l, shift left=0.1cm, "\pr_1 \times \mathrm{act}"]
\end{tikzcd}
\end{equation}
are mutually inverse equivalences in $\bH$.
Lemma~\ref{st:principality condition} now implies that $p_1 \colon G_1 \to H_1$, together with the $A$-action on $G_1$ induced by $\iota \colon A \to G$ is a principal $A$-bundle over $H_1$.

$(4) \Rightarrow (1)$:
First, note that the morphism $p \colon G \to H$ induces an action of $G$ on $H_1$ (via Proposition~\ref{st:G//A is action}).
We would like to compute the pullback
\begin{equation}
\label{eq:pb diag for classifying objects}
\begin{tikzcd}
	* \underset{\rmB H}{\times} \rmB G \ar[r] \ar[d]
	& \rmB G \ar[d]
	\\
	* \ar[r]
	& \rmB H
\end{tikzcd}
\end{equation}
in $\bH$.
Since colimits in $\bH$ are universal and $\rmB G = | {*} \dslash G|$, we have equivalences
\begin{equation}
	* \underset{\rmB H}{\times} \rmB G
	\simeq \big| (* \underset{\rmB H}{\times} *) \dslash G \big|
	\simeq |H_1 \dslash G|\,,
\end{equation}
where we have used that $\Omega \rmB H \simeq H$ in $\Grp(\bH)$ by the equivalence~\eqref{eq:Omega -| B adjoint equiv}.
One can also see this equivalence more explicitly by applying the pasting law to the diagram
\begin{equation}
\begin{tikzcd}[column sep={2.5cm,between origins}]
	H_1 \dslash G \ar[r] \ar[d]
	& H_1 \dslash H \ar[r] \ar[d]
	& * \ar[d]
	\\
	* \dslash G \ar[r]
	& * \dslash H \ar[r]
	& \rmB H
\end{tikzcd}
\end{equation}
The left-hand square is a pullback by the definition of the action $H_1 \dslash G$ and Proposition~\ref{st:G//A is action}.
The right-hand square is a pullback by Lemma~\ref{st:G//G as pb Bun}.

The canonical morphism $* \to H_1$ and the morphism $\iota \colon A \to G$ induce a morphism of simplicial objects $\psi \colon * \dslash A \longrightarrow H_1 \dslash G$.
Its colimit is a morphism $|\psi| \colon \rmB A \to |H_1 \dslash G|$.
In particular, we obtain from this an augmented simplicial object $* \dslash A \longrightarrow |H_1 \dslash G|$ in $\bH$.
Let $q \colon * \to |H_1 \dslash G|$ denote the restriction of this augmented simplicial object to $\bbDelta_{+, \leq 0}$.
We claim that $* \dslash A \longrightarrow |H_1 \dslash G|$ is equivalent to the \v{C}ech nerve of $q$.
By Definition~\ref{def:Cech nerve}, Proposition~\ref{st:gpd obs via X_1 X_1 X_0} and the fact that $* \dslash A$ is a groupoid object, it suffices to show that the diagram
\begin{equation}
\begin{tikzcd}
	A_1 \ar[r] \ar[d]
	& * \ar[d]
	\\
	* \ar[r]
	& {|H_1 \dslash G|}
\end{tikzcd}
\end{equation}
is a pullback diagram.
We can see this as follows:
with the $G$-action induced by $p \colon G \to H$, the morphism $H_1 \to |H_1 \dslash G|$ becomes a principal $G$-bundle in $\bH$.
Since the sequence of morphisms $A_1 \to G_1 \to H_1$ is a fibre sequence in $\bH$ by Corollary~\ref{st:fib pf P//G to X is G}, we obtain a double pullback diagram
\begin{equation}
\begin{tikzcd}
	A_1 \ar[r, "\iota_1"] \ar[d]
	& G_1 \ar[r] \ar[d, "p_1"]
	& * \ar[d]
	\\
	* \ar[r]
	& H_1 \ar[r]
	& {|H_1 \dslash G|}
\end{tikzcd}
\end{equation}
We have thus shown that the augmented simplicial object $* \dslash A \longrightarrow |H_1 \dslash G|$ is equivalent to the \v{C}ech nerve of $q \colon * \to |H_1 \dslash G|$.
As the latter describes the loop object $\Omega |H_1 \dslash G|$ as a group object in $\bH$, we infer that
\begin{equation}
	\Omega |H_1 \dslash G| \simeq A
	\quad \in \Grp(\bH)\,,
\end{equation}
i.e.~as group objects in $\bH$.

The claim will now follow from the equivalence~\eqref{eq:Omega -| B adjoint equiv}, provided we can additionally show that $|H_1 \dslash G|$ is a connected object in $\bH$.
One way of proving this is by showing that the morphism $* \to |H_1 \dslash G|$ is an effective epimorphism.
Consider the morphism of simplicial objects
\begin{equation}
	G_1 \dslash G \longrightarrow H_1 \dslash G
\end{equation}
induced by $p$.
In each simplicial level, this is an effective epimorphism by assumption on $p$.
Since effective epimorphisms are stable under colimits (see Lemma~\ref{st:EEpis stable under pb and po}), it follows that the morphism induced on colimits,
\begin{equation}
	* \simeq |G_1 \dslash G| \longrightarrow |H_1 \dslash G|
\end{equation}
is an effective epimorphism as well.
Alternatively, one can see that the left-hand morphism in~\eqref{eq:pb diag for classifying objects} is 0-connected:
the morphism $\rmB p \colon \rmB G \to \rmB H$ is a colimit of effective epimorphisms.
Thus, it is an effective epimorphism itself.
It is also a morphism between connected objects in $\bH$ and therefore necessarily induces an isomorphism on zeroth homotopy groups (i.e.~it is a connected morphism).
The claim then follows since $n$-connected morphisms in $\infty$-topoi are stable under pullback~\cite[Prop.~6.5.1.16(6)]{Lurie:HTT}.
\end{proof}

\begin{corollary}
Suppose \smash{$A \xrightarrow{\iota} G \xrightarrow{p} H$} is an extension of group objects in $\bH$.
Then, there is a canonical equivalence in $\bH$,
\begin{equation}
	|G_1 \dslash A| \simeq H\,.
\end{equation}
\end{corollary}

This is the $\infty$-categorical analogue of the canonical isomorphism $G/A \cong H$ for ordinary (set-theoretic) group extensions $A \to G \to H$.

\begin{corollary}
\label{st:L pres group extensions}
Let $\rmL \colon \bH \to \bH'$ be a functor between $\infty$-topoi which preserves geometric realisations and finite products.
Suppose \smash{$A \xrightarrow{\iota} G \xrightarrow{p} H$} is an extension of group objects in $\bH$.
Then, the sequence \smash{$\rmL A \xrightarrow{\rmL \iota} \rmL G \xrightarrow{\rmL p} \rmL H$} is an extension of group objects in $\bH'$.
\end{corollary}

\begin{proof}
This statement now follows from combining Theorem~\ref{st:L pres (princ) actions} and Theorem~\ref{st:oo-grp ext via pfbuns}.
\end{proof}

\section{Homotopy-theoretic smooth string group models}
\label{sec:HoThy Smooth String}

In this section, we present a definition of string group extensions within the $\infty$-category $\bH_\infty$ of smooth spaces.
It relies on the singular complex functor $\rmS_e \colon \bH_\infty \to \bS$ for smooth spaces from Section~\ref{sec:Sm spaces and oo-topoi} and the theory of group extensions in $\infty$-topoi from Section~\ref{sec:PBuns and Grp Exts in oo-topoi}.
We begin by recalling the definition of a string group extension in the $\infty$-category $\bS$ of spaces.
Then, we use our results thus far to transfer this definition to $\bH_\infty$ along the functor $\rmS_e$, leading to a homotopy-theoretic definition of smooth string group extensions (Definition~\ref{def:String(H) in H_oo}).

After recalling some background on bundle gerbes in Section~\ref{sec:Grb background}, we provide new smooth models for the string group in Section~\ref{sec:String Model}, building on recent constructions of smooth 2-group extensions in~\cite{BMS:Sm2Grp}.
(There already, evidence was given that these smooth 2-group extensions can model the string group; here we provide a full formal framework and proof for that conjecture.)

\subsection{The definition of smooth string groups}
\label{sec:Smooth String groups}

The definition of a string group via the Whitehead tower (see Section~\ref{sec:Introduction}) is purely homotopy-theoretic.
In particular, in a string group extension $A \to \String(H) \to H$ the extending group $A$ is not fixed, but only its underlying homotopy type.
So far, to our knowledge there does not exist a definition of string group extensions in a smooth context that contains this flexibility---the extending group $A$ is usually chosen ad hoc to be some smooth version of $\rmB \sfU(1)$.
Here, we provide a smooth version of the original homotopy-theoretic definition (see Definition~\ref{def:String(H) in H_oo}).
In particular, only the underlying homotopy type of the extending smooth group $A$ is fixed in this definition.

To avoid clashes of notation, we denote a Lie group by a triple $(H_1, \cdot, e_H)$, where $H_1 \in \Mfd$ is the underlying manifold, $(-) \cdot (-)$ denotes the multiplication on $H_1$ and $e_H \colon * \to H_1$ is the neutral element in $H_1$.
Recall that each compact, simple and simply connected Lie group $(H_1, \cdot, e_H)$ is also 2-connected and satisfies $\rmH^3(H_1;\ZN) \cong \ZN$~\cite{Cartan:topologie_des_groupes_de_Lie}.
Any Lie group $(H_1, \cdot, e_H)$ canonically defines a group object $H$ in the $\infty$-topos of spaces $\bS$.
We start by reformulating the definition of a string group extension of topological groups within the $\infty$-category of spaces:

\begin{definition}
\label{def:String(H) in Spaces}
Let $(H_1, \cdot, e_H)$ be a compact, simple and simply connected Lie group, and denote by $H \in \Grp(\bS)$ its associated group object in $\bS$.
A \emph{string group extension of $(H_1, \cdot, e_H)$} is an extension of group objects
\begin{equation}
\begin{tikzcd}
	A \ar[r, "\iota"]
	& \String(H) \ar[r, "p"]
	& H
\end{tikzcd}
\end{equation}
in $\bS$ such that 
\begin{myenumerate}
\item $A_1$ is an Eilenberg-MacLane space $K(\ZN,2)$, and

\item under the isomorphism
\begin{equation}
	\pi_0 \ul{\bS}(H_1, \rmB A)
	\cong \pi_0 \ul{\bS} \big( H_1, K(\ZN,3) \big)
	\cong \rmH^3(H_1; \ZN) \cong \ZN\,,
\end{equation}
the classifying morphism $H_1 \to \rmB A$ of the $A$-principal $\infty$-bundle $\String(H)_1 \dslash A \to H$ (compare Remark~\ref{rmk:Bun_G(X) = H(X,BG)} and Theorem~\ref{st:oo-grp ext via pfbuns}(4)) represents a generator of $\ZN$.
\end{myenumerate}
\end{definition}

Given condition~(1), condition~(2) is equivalent to saying that the map $\String(H)_1 \to H_1$ of spaces induces an isomorphism $\pi_i(\String(H)_1) \to \pi_i(H_1)$ for $i \neq 3$ and that $\pi_3(\String(H)_1) \cong 0$.
This is a consequence of the Hurewicz Theorem, the Universal Coefficient Theorem, and the long exact sequence of homotopy groups associated to a (homotopy) fibre sequence of spaces.
That is, $\String(H)_1 \to H_1$ is a 3-connected covering of $H_1$.

Recall the $\infty$-topos $\bH_\infty = \scP(\Cart)$ from Section~\ref{sec:Psh(Cart)}.
There we also introduced the localisation $L_I \bH_\infty$ of $\bH_\infty$ at the set $I = \{ c \times \RN \to c \, | \, c \in \Cart\}$ and the smooth singular complex functor $\rmS_e \colon \bH_\infty \to \bS$.
Also recall the fully faithful embedding \smash{$\ul{(-)} \colon \Mfd \hookrightarrow \bH_\infty$}, with $\ul{M}(c) = \Mfd(c,M)$; under this embedding, any Lie group $(H_1, \cdot, e_H)$ canonically gives rise to a group object $\ul{H}$ in $\bH_\infty$.
We can now use our results from Section~\ref{sec:PBuns and Grp Exts in oo-topoi} to transfer the definition of a string group extension to the $\infty$-topos $\bH_\infty$:

\begin{definition}
\label{def:String(H) in H_oo}
Let $(H_1, \cdot, e_H)$ be a compact, simple, and simply connected Lie group.
A \emph{smooth string group extension of $(H_1, \cdot, e_H)$} is an extension
\begin{equation}
\begin{tikzcd}
	A \ar[r, "\iota"]
	& \String(\ul{H}) \ar[r, "p"]
	& \ul{H}
\end{tikzcd}
\end{equation}
of group objects in $\bH_\infty$ whose its image under $\rmS_e$ is a string group extension in $\bS$.
\end{definition}

Note that by Theorem~\ref{st:L pres group extensions} the functor $\rmS_e$ maps group extensions in $\bH_\infty$ to group extensions in $\bS$.
Further, even though $\rmS_e$ induces an equivalence between $\bS$ and the localisation $L_I \bH_\infty$ rather than the full $\infty$-category $\bH_\infty$, we do not need to demand that $A$, $\String(\ul{H})$ and $\ul{H}$ are local objects, because $\rmS_e$ sends all $I$-local equivalences in $\bH_\infty$ to equivalences in $\bS$ (Theorem~\ref{st:S_e Thm}(1)).

\begin{remark}
\label{rmk:generalising string group def}
Definition~\ref{def:String(H) in H_oo} is a generalisation as well as a weakening of the following approach to smooth string group extensions (see, for instance, \cite{FRS:U1-Gerbe_connections}):
there, one works in the localisation $L_\tau \bH_\infty$ of $\bH_\infty$ at the \v{C}ech nerves of differentiably good open coverings $\{c_a \to c\}_{a \in \Lambda}$ of cartesian spaces $c \in \Cart$.
Recall that a differentiably good open covering of $c \in \Cart$ is an open covering $\{c_a \hookrightarrow c\}_{a \in \Lambda}$ such that every finite non-empty intersection of the images of the patches $c_a$ is again a cartesian space.
The differentiably good open coverings endow $\Cart$ with a Grothendieck coverage $\tau$~\cite{FSS:Cech_diff_char_classes_via_L_infty,Schreiber:DCCT}.
In~\cite{FRS:U1-Gerbe_connections} string group extension of $H$ are defined via the pullback
\begin{equation}
\label{eq:usual String pb}
\begin{tikzcd}
	\rmB \String(\ul{H}) \ar[r] \ar[d] & * \ar[d]
	\\
	\rmB \ul{H} \ar[r, "\frac{1}{2} p_1"'] & \rmB^3 \ul{\rmU(1)}
\end{tikzcd}
\end{equation}
Here, $\frac{1}{2} p_1$ denotes the fractional first Pontryagin class, which is a generator of $\rmH^4(\rmB \ul{H}; \ZN) \cong \ZN$.
However, this definition of $\String(\ul{H})$ is considerably stricter than the original perception of $\String(H)$ as a 3-connected covering of $H_1$ by another group object (Definition~\ref{def:String(H) in Spaces}).
For instance, the definition of a string group extension based on~\eqref{eq:usual String pb} enforces that the morphism $\String(\ul{H})_1 \to \ul{H}_1$ is a $\rmB \ul{\sfU(1)}$-principal $\infty$-bundle (note that if $\bH$ is an $\infty$-topos and \smash{$A \in \Grp(\bH)$} is a group object whose multiplication lifts to an $\bbE_2$-algebra structure, then $\rmB A$ is canonically the underlying object of a group object in $\bH$~\cite{NSS:oo-bundles}).
However, from the purely homotopy-theoretic point of view, it is not the actual fibre of this map in $\bH$ that should be fixed, but only the \textit{homotopy type of its underlying space in $\bS$} (which must be a $K(\ZN;2)$).
Definition~\ref{def:String(H) in H_oo} emphasises this latter, homotopy-theoretic aspect of string group extensions.

More concretely, for smooth string group extensions \smash{$A \overset{\iota}{\to} \String(\ul{H}) \overset{p}{\to} \ul{H}$} in the sense of Definition~\ref{def:String(H) in H_oo} it is enough if there is an $I$-local equivalence $A_1 \simeq \rmB \ul{\sfU(1)}$ in $\bH_\infty$.
Therefore, this setup is considerably more general than working with the pullback~\eqref{eq:usual String pb}.
In particular, two different smooth string group extensions of a Lie group $H$ need not be equivalent in $\bH_\infty$, but only in $L_I \bH_\infty$.
In Section~\ref{sec:String Model} we present a smooth string group extension which satisfies the criteria from Definition~\ref{def:sm 2Grp}, but not the stricter version~\eqref{eq:usual String pb}:
its fibre $A$ is \textit{not} equivalent to $\rmB \ul{\rmU(1)}$ in $L_\tau \bH_\infty$, but only in $L_I \bH_\infty \simeq \bS$.
\qen
\end{remark}

\begin{remark}
Not allowing for this flexibility would lead to simply missing or being unable to recognise smooth group extensions found in nature whose underlying spaces form a string group extension in the classical sense.
The smooth $\infty$-group $A$ we find in the string group model in Section~\ref{sec:String Model} is much larger than the simple delooping of $\ul{\sfU(1)}$, but this can have advantages:
for example, our string group model should act extremely naturally on the K-theory of the underlying Lie group twisted by its basic gerbe.
\qen
\end{remark}

\begin{remark}
It will be very interesting to see a Lie-algebra version of Definition~\ref{def:String(H) in H_oo}.
The $\infty$-groups $A \in \Grp(\bH_\infty)$ that we allow to appear in string group extensions can have much larger Lie algebras than those which appear in the stricter definition via~\eqref{eq:usual String pb}.
This is true, in particular, for the smooth string group extension we present in Section~\ref{sec:String Model} below.
There might hence be a Lie-algebra version of $I$-local equivalences of group objects in $\bH_\infty$.
\qen
\end{remark}

\subsection{Bundle gerbes and their symmetries}
\label{sec:Grb background}

Before we can present our smooth string group extension, we need to recall some background on bundle gerbes.
We will not give full definitions or details here, but refer the reader to~\cite{Waldorf--More_morphisms, Waldorf--Thesis, Bunk--Thesis, BSS--HGeoQuan} for technical background and~\cite{Bunk-Szabo:Fluxes_brbs_2Hspaces, Bunk:Gerbes_Review} for an introduction to the topic.
Bundle gerbes provide an explicit, geometric model for categorified line bundles.
We point out that there also exists a notion of connection on a bundle gerbe, but here we will only be working with bundle gerbes \emph{without} connection.
(This is the main technical cause for the distinction between our smooth string group model and that in~\cite{FRS:U1-Gerbe_connections}.)

To any manifold $M$, we can assign a symmetric monoidal 2-groupoid $(\Grb(M), \otimes)$ of bundle gerbes on $M$.
Given a bundle gerbe $\CG \in \Grb(M)$, the monoidal groupoid $\Grb(M)(\CG, \CG)$ of automorphisms of $\CG$ is canonically equivalent to the symmetric monoidal groupoid $(\HLB(M),\otimes)$ of hermitean line bundles on $M$ with the usual tensor product (which we also denote by $\otimes$).
Note that $(\HLB(M),\otimes)$ is even a \emph{2-group}; that is, it is a symmetric monoidal groupoid in which every object has an inverse with respect to the monoidal product.
Every smooth map $f \colon N \to M$ of manifolds induces a monoidal 2-functor
\begin{equation}
	f^* \colon \Grb(M) \longrightarrow \Grb(N)\,.
\end{equation} 
Isomorphism classes of gerbes are in canonical bijection with the third integer cohomology of $M$:
there is an isomorphism of abelian groups
\begin{equation}
\label{eq:DD class}
	\pi_0 \big( \Grb(M), \otimes \big) \cong \rmH^3(M;\ZN)\,.
\end{equation}
The class associated to a gerbe $\CG$ under this isomorphism is called the \emph{Dixmier-Douady class of $\CG$}.

We let $\bH_{\leq 1}$ denote the following 2-category:
its objects are functors $\pi \colon \scC \to \Cart$ that are Grothendieck fibrations in groupoids (that is, $\pi$ is a Grothendieck fibration and all its fibres are groupoids).
Its morphisms $(\pi \colon \scC \to \Cart) \longrightarrow (\pi' \colon \scC' \to \Cart)$ are functors $F \colon \scC \to \scC'$ such that $\pi' \circ F = \pi$, and its 2-morphisms are natural transformations $\eta \colon F \to F'$ such that $\pi' \eta$ is the identity natural transformation $1_\Cart \to 1_\Cart$.
Note that the 2-category $\bH_{\leq 1}$ is canonically equivalent to the 2-category of pseudo-functors $\Cart^\opp \to \Gpd$ from $\Cart^\opp$ to the 2-category of groupoids via the Grothendieck construction.
We make the following definitions; for more background, see~\cite{BMS:Sm2Grp,SP:String}.

\begin{definition}
\label{def:sm 2Grp}
\emph{\cite{BMS:Sm2Grp}}
The \emph{2-category of smooth 2-groups} is the 2-category of group objects in the 2-category $\bH_{\leq 1}$.
\end{definition}

\begin{example}
Let $(H_1, \cdot, e_H)$ be a Lie group.
We associate to it the following category, denoted by $\textint \ul{H}_1$:
its objects are pairs $(c,h)$ of a cartesian space $c \in \Cart$ and a smooth map $h \colon c \to H_1$.
A morphism $(c,h) \to (c', h')$ is a smooth map $f \colon c \to c'$ such that $h' \circ f = h$.
The category $\textint \ul{H}_1$ comes with a canonical projection functor $\textint \ul{H}_1 \to \Cart$.
The product on $H_1$-valued maps turns $\textint \ul{H}_1$ into a smooth 2-group in the sense of Definition~\ref{def:sm 2Grp}; we denote this smooth 2-group by $\textint \ul{H}$.
Note that $\textint \ul{H}_1$ is simply the Grothendieck construction of the presheaf of sets $\ul{H}_1$ on $\Cart$.
\qen
\end{example}

\begin{example}
For a manifold $M$, let $\HLB(M)$ denote the groupoid of hermitean line bundles on $M$.
For fixed manifold $M$, we define a category $(\HLB^M)_1$ as follows:
its objects are pairs $(c,L)$ of a cartesian space $c \in \Cart$ and a hermitean line bundle $L \in \HLB(c {\times} M)$.
A morphism $(c,L) \to (c',L')$ is a pair $(f, \psi)$ of a smooth map $f \colon c \to c'$ and an isomorphism $\psi \colon L \to (f {\times} 1_M)^*L'$ of hermitean line bundles over $c$.
This category comes with a projection functor $(\HLB^M)_1 \to \Cart$.
The tensor product of hermitean line bundles turns $(\HLB^M)_1$ into a smooth 2-group, which we denote by $\HLB^M$.
\qen
\end{example}

Let $M$ be a manifold, and let $\CG \in \Grb(M)$ be a gerbe on $M$.
Further, let $(H_1, \cdot, e_H)$ be a connected Lie group acting smoothly on $M$ from the left; we denote the action by $\Phi \colon H_1 \times M \to M$.
Given these data, we define a category $\Sym(\CG)_1$ as follows:
an object in $\Sym(\CG)_1$ is a triple $(c, h, \CA)$, where $c \in \Cart$ is a cartesian space and where $h \colon c \to H_1$ is a smooth map.
These give rise to a smooth map $\Phi_h \colon c \times M \to c \times M$, defined as the composition
\begin{equation}
	\Phi_h \colon c \times M
	\xrightarrow{\Delta \times 1_M} c \times c \times M
	\xrightarrow{1_c \times h \times 1_M} c \times H \times M
	\xrightarrow{1_c \times \Phi} c \times M\,,
\end{equation}
where $\Delta \colon c \to c \times c$ is the diagonal map.
Then, $\CA$ is a 1-isomorphism
\begin{equation}
	\CA \colon \pr_M^*\CG \longrightarrow \Phi_h^*\CG
\end{equation}
of gerbes on the manifold $c \times M$.
A morphism $(c,h,\CA) \to (c',h',\CA')$ is a pair $(f,\psi)$, where $f$ is a smooth map $f \colon c \to c'$ such that $h' \circ f = h$, and where $\psi$ is a 2-isomorphism $\psi \colon \CA \longrightarrow (f \times 1_M)^*\CA'$
(where we have implicitly used that there is a canonical 1-isomorphism $(f \times 1_M)^*\Phi_{h'}^*\CG \cong \Phi_h^*\CG$).
Observe that there is a projection functor $p_1 \colon \Sym(\CG)_1 \to \textint \ul{H}_1$, acting as $(c, h, \CA) \mapsto (c, h)$ and $(f, \psi) \mapsto f$.

\begin{remark}
In this set-up, the following statements hold true:
\begin{myenumerate}
\item There is a canonical inclusion $\iota_1 \colon (\HLB^M)_1 \hookrightarrow\Sym(\CG)_1$ in $\bH_{\leq 1}$.

\item The connectedness of $H_1$ ensures that the functor $p_1$ is surjective on objects.
Further, $p_1$ is an essentially surjective Grothendieck fibration in groupoids; it is even strictly surjective on objects~\cite[Thm.~5.27]{BMS:Sm2Grp}.

\item The equivalence $\Grb(N)(\CG',\CG') \simeq (\HLB(N),\otimes)$ for any gerbe $\CG'$ on any manifold $N$ implies that the diagram
\begin{equation}
\begin{tikzcd}[row sep=1cm]
	(\HLB^M)_1 \ar[r, "\iota_1"] \ar[d] & \Sym(\CG)_1 \ar[d, "p_1"]
	\\
	* \ar[r, "e_{\ul{H}}"'] & \textint \ul{H}_1
\end{tikzcd}
\end{equation}
is a pullback in $\bH_{\leq 1}$, where $e_{\ul{H}}$ is the functor that sends $c \in \Cart$ to the constant map $c \to H_1$ with value the unit element of $H_1$.
Since $p_1$ is a Grothendieck fibration in groupoids, this pullback is even a homotopy pullback~\cite[App.~A.1]{BMS:Sm2Grp}.
\qen
\end{myenumerate}
\end{remark}

\begin{theorem}
{\emph{\cite[Thms.~5.23, 5.27]{BMS:Sm2Grp}}}
Let $\Phi \colon H_1 \times M \to M$ be a smooth action of a connected Lie group $(H_1, \cdot, e_H)$ on a manifold $M$.
Let $\CG \in \Grb(M)$ be a bundle gerbe on $M$.
\begin{myenumerate}
\item $\Sym(\CG)_1$ carries the structure of a smooth 2-group.
We denote this smooth 2-group by $\Sym(\CG)$.

\item The functors $\iota_1$ and $p_1$ canonically lift to morphisms of smooth 2-groups and induce a sequence
\begin{equation}
\label{eq:Sym 2Grp SES}
	\HLB^M \overset{i}{\longrightarrow} \Sym(\CG) \overset{p}{\longrightarrow} \textint \ul{H}
\end{equation}
of smooth 2-groups.
\end{myenumerate}
\end{theorem}

The nerve functor $N \colon \Cat \to \Cat_\infty$ induces a functor $N \colon \bH_{\leq 1} \to \bH_\infty$ (where we have used the canonical equivalence between $\bH_{\leq 1}$ and the 2-category of pseudo-functors $\Cart^\opp \to \Gpd$ from $\Cart^\opp$ to the 2-category of groupoids).
This functor, in particular, preserves final objects and products, so that it maps smooth 2-groups to group objects in $\bH_\infty$.
Our smooth string group model will be obtained by applying this functor to the sequence~\eqref{eq:Sym 2Grp SES}.

\subsection{A smooth string group model}
\label{sec:String Model}

We can now state the main theorem of this section.
It provides a new smooth model for smooth string group extensions which fits Definition~\ref{def:String(H) in H_oo}, but which lies outside the scope of the stricter definition via the pullback~\eqref{eq:usual String pb}.
Note that applying the nerve functor $N$ to $\textint \ul{H}_1 \in \bH_{\leq 1}$ yields the familiar presheaf of spaces $\ul{H}_1 \in \bH_\infty$, defined via $\ul{H}_1(c) = \Mfd(c,H)$ for cartesian spaces $c \in \Cart$.
Further, $N$ maps the smooth 2-group $\textint \ul{H}$ in $\bH_{\leq 1}$ to the group object $\ul{H}$ in $\bH_\infty$.

\begin{theorem}
\label{st:String model from Sym(G)}
Let $(H_1, \cdot, e_H)$ be a compact, simple, simply connected Lie group.
We consider the left-action of $H_1$ on itself via left multiplication.
Let $\CG \in \Grb(H)$ be a gerbe on $H_1$ whose class in $\rmH^3(H;\ZN) \cong \ZN$ is a generator (see~\eqref{eq:DD class}).
Then, the sequence
\begin{equation}
\label{eq:Sym_G String sequence}
\begin{tikzcd}
	N \big( \HLB^H \big) \ar[r, "N \iota"]
	& N \big( \Sym(\CG) \big) \ar[r, "N p"]
	& \ul{H}
\end{tikzcd}
\end{equation}
of group object in $\bH_\infty$ is a smooth string group extension of $(H_1, \cdot, e_H)$.
\end{theorem}

The proof of Theorem~\ref{st:String model from Sym(G)} will occupy the remainder of this section.
By Definition~\ref{def:String(H) in H_oo} we have to show that the sequence~\eqref{eq:Sym_G String sequence} is an extension of group objects in $\bH_\infty$ and that its image under the functor $\rmS_e \colon \bH_\infty \to \bS$ is a string group extension in $\bS$ in the sense of Definition~\ref{def:String(H) in Spaces}.

\begin{proposition}
The sequence~\eqref{eq:Sym_G String sequence} is an extension of group objects in the $\infty$-topos $\bH_\infty$.
\end{proposition}

\begin{proof}
The nerve functor $N \colon \Cat \to \Cat_\infty$ is a right adjoint and hence maps products in $\bH_{\leq 1}$ to products in $\bH_\infty$, and final objects to the final objects.
Consequently, it preserves group objects and group actions.

We will now use the characterisation of group extensions from Theorem~\ref{st:oo-grp ext via pfbuns}(4) to show that the sequence~\eqref{eq:Sym_G String sequence} of group objects in $\bH_\infty$ is an extension of group objects.
That is, we have to show that $N \Sym(\CG)_1$ with the $N \HLB^H$-action induced by the morphism $N \iota$ (cf.~Proposition~\ref{st:G//A is action}) is an $N \HLB^H$-principal $\infty$-bundle over $\ul{H}_1$.
According to the characterisation of principal $\infty$-bundles in Proposition~\ref{st:pfbun characterisation}, it suffices to prove that the morphism $N p_1$ is an effective epimorphism and that the action of $N \HLB^H$ on $N \Sym(\CG)_1$ is principal.

We start by showing that the morphism $N p_1$ is an effective epimorphism:
by~\cite[Sec.~5.1]{BMS:Sm2Grp} the restriction $p_{1|c}$ of $p_1$ to any fibre is essentially surjective, hence $N p_{1|c}$ is surjective on connected components.
Since $\bH_\infty$ is a presheaf $\infty$-topos (in which limits and colimits are computed objectwise), a morphism in $\bH_\infty$ is an effective epimorphism if and only if it is objectwise an effective epimorphism in $\bS$.
The effective epimorphisms in $\bS$, however, are exactly those morphisms which are surjective on connected components~\cite[Cor.~7.2.1.15]{Lurie:HTT}.
Therefore, $N p_1$ is an effective epimorphism in $\bH_\infty$.

The action of $N \HLB^H$ on $N \Sym(\CG)_1$ is principal with respect to $N p_1$ as was shown in~\cite[Thm.~5.27]{BMS:Sm2Grp} (there, the principality condition was shown on the level of the sequence~\eqref{eq:Sym 2Grp SES} of smooth 2-groups---this suffices for the $\infty$-categorical context used here because of Lemma~\ref{st:principality condition} and because the nerve functor is a right adjoint).
Therefore, the sequence~\eqref{eq:Sym_G String sequence} is a group extension in $\bH_\infty$.
\end{proof}

It thus remains to show that the image of the sequence~\eqref{eq:Sym_G String sequence} under $\rmS_e$ is a string group extension in $\bS$.
To that end, we first show the following lemma:

\begin{lemma}
\label{st:I-loc eq NHLB^H to BU(1)}
Let $M$ be a connected manifold with $\rmH^2(M; \ZN) \cong 0$.
\begin{myenumerate}
\item The object $(N \HLB^M)_1 = N (\HLB^M)_1 \in \bH_\infty$ is equivalent to $\rmB( \ul{\sfU(1)}^{\ul{M}})$.

\item If $M$ is additionally simply connected, $(N \HLB^M)_1 \in \bH_\infty$ is $I$-locally equivalent to \smash{$\rmB \ul{\sfU(1)} \in \bH_\infty$}.
\end{myenumerate}
Both equivalences are even established by morphisms of group objects in $\bH_\infty$.
\end{lemma}

Since $\rmS_e$ maps $I$-local equivalences in $\bH_\infty$ to equivalences of spaces, applying Lemma~\ref{st:I-loc eq NHLB^H to BU(1)} to $M = H_1$ establishes axiom~(1) of Definition~\ref{def:String(H) in Spaces} for the image of the sequence~\eqref{eq:Sym_G String sequence} under the functor $\rmS_e$.
Note that for a simplicial object $X \in \Fun(\bbDelta^\opp, \bH)$ in an $\infty$-topos $\bH$ and an object $B \in \bH$, we can form the level-wise exponential $X^B \in \Fun(\bbDelta^\opp, \bH)$.
There are canonical equivalences $(X^B)_n \simeq (X_n)^B$.
Thus, we can simply write $X_n^B$.
If $X$ is a groupoid or group object, then so is $X^B$.

\begin{proof}
We proceed in parallel to the proof of~\cite[Thm.~8.7]{BMS:Sm2Grp}:
since any $c \in \Cart$ is contractible and since $\rmH^2(M; \ZN) \cong 0$, it follows that any hermitean line bundle on $c  \times M$ is trivialisable.
Consequently, the groupoid $\HLB^M(c)$ is equivalent to the groupoid with one object and morphisms given by the group \smash{$\ul{\sfU(1)}^{\ul{M}}(c)$} of smooth maps from $c \times M$ to $\sfU(1)$.
This induces an equivalence \smash{$(N \HLB^M)_1 \simeq \rmB(\ul{\sfU(1)}^{\ul{M}})$} in $\bH_\infty$, which extends to a morphism of group objects in $\bH_\infty$.
This proves~(1).

Next, since $\pi_1(M)$ is trivial, there exists a smooth homotopy equivalence \smash{$\ev_{x,1} \colon \ul{\sfU(1)}{}_1^{\ul{M}} \to \ul{\sfU(1)}{}_1$}, given by restricting a smooth map $c \times M \to \sfU(1)$ to $c \times \{x\}$, where $x \in M$ is any point.
A homotopy inverse to $\ev_{x,1}$ is given by pulling a smooth map $c \to \sfU(1)$ back along the projection $c \times M \to c$~\cite[Lemma~8.9]{BMS:Sm2Grp}.
In particular, $\ev_{x,1}$ is an $I$-local equivalence~\cite[Cor.~3.16]{Bun:Sm_Spaces}.

Observe that $\ev_{x,1}$ induces a morphism of group objects
\begin{equation}
	\ev_x \colon \ul{\sfU(1)}^{\ul{M}} \longrightarrow \ul{\sfU(1)}\,.
\end{equation}
Since $\ev_{x,1}$ is an $I$-local equivalence in $\bH_\infty$ and $I$-local equivalences are closed under finite products (Proposition~\ref{st:Loc pres fin prods}), the morphism $\ev_x$ is a levelwise $I$-local equivalence of simplicial objects in $\bH_\infty$.

Further, the class $W_I$ of $I$-local equivalences in $\bH_\infty$ is strongly saturated~\cite[Lemma~5.5.4.11]{Lurie:HTT}.
In particular, the full subcategory of $\Fun(\Delta^1, \bH_\infty)$ on the $I$-local equivalences is stable under colimits.
Therefore, taking the colimit in $\bH_\infty$ of simplicial objects (i.e.~taking geometric realisations), we obtain an $I$-local equivalence
\begin{equation}
	\rmB \ev_x \colon \rmB \big( \ul{\sfU(1)}^{\ul{M}} \big) \longrightarrow \rmB \ul{\sfU(1)}
\end{equation}
in $\bH_\infty$.
Composing with the morphism constructed in part~(1), we now obtain the desired $I$-local equivalence $N \HLB^M \longrightarrow \rmB \ul{\sfU(1)}$ in $\bH_\infty$.
\end{proof}

We are thus left to show that the sequence of group objects in $\bS$ obtained by applying the functor $\rmS_e$ to the sequence~\eqref{eq:Sym_G String sequence} of group objects in $\bH_\infty$ satisfies axiom~(2) of Definition~\ref{def:String(H) in Spaces}.
That is, we have to show that the principal $\infty$-bundle of spaces
\begin{equation}
	\big( \rmS_e N \Sym(\CG)_1 \big) \dslash \big( \rmS_e N\, \HLB^H \big)
	\longrightarrow \rmS_e \ul{H}_1
\end{equation}
represents a generator of $\rmH^3(H; \ZN) \cong \ZN$.
This is best checked using \v{C}ech cohomology.

Recall the Grothendieck coverage $\tau$ of differentiably good open coverings on $\Cart$ from Remark~\ref{rmk:generalising string group def}.

\begin{lemma}
\label{st:B^nU(1)^N}
Let $M \in \Mfd$ be a simply connected manifold, and let $k \in \NN_0$.
Then,
\begin{myenumerate}
\item In $\bH_\infty$, there is an $I$-local equivalence
\begin{equation}
	\rmB^k (\ul{\sfU(1)}^{\ul{M}}) \simeq \rmB^k \ul{\sfU(1)}\,.
\end{equation}

\item The presheaf $\rmB^k (\ul{\sfU(1)}^{\ul{M}})$ satisfies descent with respect to the Grothendieck coverage $\tau$ of differentiably good open coverings on $\Cart$.
\end{myenumerate}
\end{lemma}

\begin{proof}
Ad~(1):
This is an iteration of the argument in the proof of Lemma~\ref{st:I-loc eq NHLB^H to BU(1)}(2).

Ad~(2):
We prove this claim by induction.
For $k = 0$, we have to check that the functor $\Cart^\opp \to \bS$, $c \mapsto \Mfd(c \times M, \sfU(1))$ satisfies descent with respect to good open coverings of $c$.
However, this follows directly from the fact that, for any manifold $Y$, the functor
\begin{equation}
	\Op(Y)^\opp \to \Set\,, \qquad U \mapsto  \Mfd \big( U, \sfU(1) \big)
\end{equation}
defines a sheaf on $Y$, where $\Op(Y)$ is the category of open subsets of $Y$ and their inclusions.

Suppose that $\rmB^l(\ul{\sfU(1)}^{\ul{M}})$ is a sheaf on $\Cart$ for all $l = 0, \ldots, k$.
Let $c \in \Cart$, and let $\scU = \{c_a \hookrightarrow c\}_{a \in \Lambda}$ be a differentiably good open covering of $c$.
We have to show that the canonical morphism
\begin{equation}
\label{eq:q^* for B^k U(1)^M}
	q^* \colon \rmB^{k+1} \big( \ul{\sfU(1)}^{\ul{M}} \big)(c)
	\longrightarrow \underset{n \in \bbDelta}{\lim}^\bS\, \ul{\bH}_\infty \big( \cC \scU_n, \rmB^{k+1} (\ul{\sfU(1)}^{\ul{M}}) \big)
\end{equation}
is an equivalence of spaces.
Here, $\cC \scU \in \Fun(\bbDelta^\opp, \bH_\infty)$ is the \v{C}ech nerve of the covering $\scU$.

We first show that $q^*$ is essentially surjective; that is, it induces a bijection on isomorphism classes of objects.
Since limits and colimits in $\bH_\infty = \Fun(\Cart^\opp, \bS)$ are computed pointwise, we have isomorphisms
\begin{equation}
	\pi_0 \Big( \rmB^{k+1} \big( \ul{\sfU(1)}^{\ul{M}} \big) \Big) (c)
	\cong \pi_0 \Big( \rmB^{k+1} \big( \ul{\sfU(1)}^{\ul{M}}(c) \big) \Big)
	 = *\,.
\end{equation}
On the other hand, we have that
\begin{equation}
\label{eq:iso classes and coho}
	\pi_0\, \underset{n \in \bbDelta}{\lim}^\bS\, \ul{\bH}_\infty \big( \cC \scU_n, \rmB^{k+1} (\ul{\sfU(1)}^{\ul{M}}) \big)
	\cong \check{\rmH}^{k+1} \big( \scU; \ul{\sfU(1)}^{\ul{M}} \big)\,,
\end{equation}
where on the right-hand side we have the usual \v{C}ech cohomology group with respect to the covering $\scU$ of the sheaf of abelian groups on $c$ given by
\begin{equation}
	\Op(c)^\opp \to \Ab\,, \qquad U \mapsto \Mfd \big( U \times M, \sfU(1) \big)\,.
\end{equation}
By a slight abuse of notation, we also denote this sheaf by \smash{$\ul{\sfU(1)}^{\ul{M}}$}.
We claim that the right-hand side of~\eqref{eq:iso classes and coho} is further isomorphic to $\rmH^{k+2}(c; \ZN) \cong *$.

First, the \v{C}ech cohomology groups \smash{$\check{\rmH}^{k+1} (\scU; \ul{\sfU(1)}^{\ul{M}})$} are isomorphic to the sheaf cohomology groups of \smash{$\ul{\sfU(1)}^{\ul{M}}$}:
since the covering $\scU$ of $c$ is differentiably good, it follows from~\cite[Thm.~1.3.6]{Brylinski:Loops_and_GeoQuan} that there is a canonical isomorphism
\begin{equation}
	\check{\rmH}^k \big( \scU; \ul{\sfU(1)}^{\ul{M}} \big)
	\cong \rmH^k \big( c; \ul{\sfU(1)}^{\ul{M}} \big)\,.
\end{equation}
Next, we observe that since $M$ is simply connected, there is a short exact sequence
\begin{equation}
	\ZN \longrightarrow \ul{\RN}^{\ul{M}} \longrightarrow \ul{\sfU(1)}^{\ul{M}}\,.
\end{equation}
We further observe that the sheaf $\ul{\RN}^{\ul{M}}$ is fine (it admits partitions of unity, for instance those induced from the canonical map $\ul{\RN} \to \ul{\RN}^{\ul{M}}$) when seen as a sheaf on the open subsets of a manifold.
Therefore, for any manifold $Y$, we have a canonical isomorphism
\begin{equation}
	\rmH^k \big( Y; \ul{\sfU(1)}^{\ul{M}} \big) \cong \rmH^{k+1}(Y; \ZN)
\end{equation}
for every $k \geq 1$.
We thus arrive at
\begin{equation}
	\pi_0\, \underset{n \in \bbDelta}{\lim}^\bS\, \ul{\bH}_\infty \big( \cC \scU_n, \rmB^{k+1} (\ul{\sfU(1)}^{\ul{M}}) \big)
	\cong \rmH^{k+2}(c; \ZN)
	\cong *\,,
\end{equation}
for each cartesian space $c \in \Cart$ and each $k > 0$.
This completes the proof that the morphism $q^*$ from~\eqref{eq:q^* for B^k U(1)^M} is bijective on connected components.

It remains to check that the morphism $q^*$ from~\eqref{eq:q^* for B^k U(1)^M} is an isomorphism on all higher homotopy groups in $\bH_\infty = \scP(\Cart)$.
We will achieve this by comparing the automorphisms of the unique object in the source and target space of $q^*$.
On the source side, this automorphism space is given as the pullback of spaces
\begin{equation}
\begin{tikzcd}
	\Omega \rmB^{k+1} \big( \ul{\sfU(1)}^{\ul{M}} \big)(c) \ar[r] \ar[d] & * \ar[d]
	\\
	* \ar[r] & \rmB^{k+1} \big( \ul{\sfU(1)}^{\ul{M}} \big)(c)
\end{tikzcd}
\end{equation}
and there is a canonical equivalence in $\bH_\infty$,
\begin{equation}
	\Omega \rmB^{k+1} \big( \ul{\sfU(1)}^{\ul{M}} \big)(c)
	\simeq \rmB^k \big( \ul{\sfU(1)}^{\ul{M}}(c) \big)\,.
\end{equation}
On the target side of $q^*$, the automorphism space of the (essentially) unique object is the pullback
\begin{equation}
\begin{tikzcd}
	\Omega\, \underset{n \in \bbDelta}{\lim}^\bS\, \ul{\bH}_\infty \big( \cC \scU_n, \rmB^{k+1} (\ul{\sfU(1)}^{\ul{M}}) \big) \ar[r] \ar[d] & * \ar[d]
	\\
	* \ar[r] & \underset{n \in \bbDelta}{\lim}^\bS\, \ul{\bH}_\infty \big( \cC \scU_n, \rmB^{k+1} (\ul{\sfU(1)}^{\ul{M}}) \big)
\end{tikzcd}
\end{equation}
Since limits in $\bH_\infty$ are computed objectwise, there are canonical equivalences
\begin{align}
	\Omega\, \underset{n \in \bbDelta}{\lim}^\bS\, \ul{\bH}_\infty \big( \cC \scU_n, \rmB^{k+1} (\ul{\sfU(1)}^{\ul{M}}) \big)
	&\simeq \underset{n \in \bbDelta}{\lim}^\bS\, \ul{\bH}_\infty \big( \cC \scU_n, \Omega\, \rmB^{k+1} (\ul{\sfU(1)}^{\ul{M}}) \big)
	\\*
	&\simeq \underset{n \in \bbDelta}{\lim}^\bS\, \ul{\bH}_\infty \big( \cC \scU_n, \rmB^k (\ul{\sfU(1)}^{\ul{M}}) \big)\,.
\end{align}
However, by the induction hypothesis, the presheaf $\rmB^k (\ul{\sfU(1)}^{\ul{M}})$ is a sheaf, so that $q^*$ induces an equivalence between the automorphism spaces.
This proves that $q^*$ is indeed an equivalence.
\end{proof}

An application of~\cite[Thm.~1.1]{Bunk:Higher_Sheaves} now implies that, for each simply connected manifold $M$, the presheaf of spaces
\begin{equation}
	\ul{\bH}_\infty \big(-, \rmB^n \big( \ul{\sfU(1)}^{\ul{M}} \big) \big) \colon \Mfd^\opp \longrightarrow \bS
\end{equation}
satisfies descent with respect to open coverings (and even surjective submersions).
Consequently, given any open covering $\scV = \{c_a \hookrightarrow M\}_{a \in \Lambda}$, whose \v{C}ech nerve we denote by $\cC \scV \to \ul{M}$, the canonical morphism
\begin{equation}
	\ul{\bH}_\infty \big( \ul{M}, \rmB^n (\ul{\sfU(1)}^{\ul{M}}) \big)
	\cong \underset{\bbDelta}{\lim}^\bS\, \ul{\bH}_\infty \big( \cC \scV, \rmB^n (\ul{\sfU(1)}^{\ul{M}}) \big)
\end{equation}
is an equivalence of spaces.
Therefore, there is an isomorphism
\begin{equation}
\label{eq:pi_0H_oo and Cech coho}
	\pi_0 \ul{\bH}_\infty \big( \ul{M}, \rmB^n (\ul{\sfU(1)}^{\ul{M}}) \big)
	\cong \check{\rmH}^n \big( M; \ul{\sfU(1)}^{\ul{M}} \big)\,,
\end{equation}
which can be represented explicitly by composing a morphism \smash{$\ul{M} \to \rmB^n (\ul{\sfU(1)}^{\ul{M}})$} with any \v{C}ech nerve $\cC \scV \to \ul{M}$ of an open covering of $M$.
(Alternatively, this can be seen directly in the presentation of $\bH_\infty$ by the projective model structure on simplicial presheaves on $\Cart$.)

Now let us return to the case where $M = H_1$ is the manifold underlying the compact, simple, simply connected Lie group $(H_1, \cdot, e_H)$.
Let \smash{$\ev_{e,1} \colon \ul{\sfU(1)}{}^{\ul{H}_1}_1 \longrightarrow \ul{\sfU(1)}{}_1$} be the morphism induced by pullback along the base-point inclusion $e_H \colon * \hookrightarrow \ul{H}_1$.
It induces a morphism of group objects \smash{$\ev_e \colon \ul{\sfU(1)}^{\ul{H}_1} \longrightarrow \ul{\sfU(1)}$}.
We obtain a commutative diagram
\begin{equation}
\label{eq:Buns in H and Cech coho}
\begin{tikzcd}[column sep=2.5cm, row sep=0.8cm]
	\pi_0 \ul{\bH}_\infty \big( \ul{H}_1, \rmB^n (\ul{\sfU(1)}^{\ul{H}_1}) \big)
	\ar[d, "\cong"]	\ar[r, "(\rmB^n \ev_e)_*"]
	& \pi_0 \ul{\bH}_\infty \big( \ul{H}_1, \rmB^n \ul{\sfU(1)} \big) \ar[d, "\cong"]
	\\
	\check{\rmH}^n \big( H_1; \ul{\sfU(1)}^{\ul{H}_1} \big) \ar[r, "(\ev_e)_*"]
	& \check{\rmH}^n \big( H; \ul{\sfU(1)} \big)
\end{tikzcd}
\end{equation}
It was shown in~\cite[Prop.~8.11]{BMS:Sm2Grp} that the bottom horizontal morphism is an isomorphism for all $n \in \NN$ (with $n > 0$); thus, so is the top horizontal morphism.
Consider the morphisms
\begin{align}
\label{eq:Buns in H to buns in S}
	\pi_0 \ul{\bH}_\infty \big( \ul{H}_1, \rmB^n (\ul{\sfU(1)}^{\ul{H}_1}) \big)
	\longrightarrow \ &\pi_0 \ul{\bS} \big( \rmS_e \ul{H}_1, \rmS_e \rmB^n (\ul{\sfU(1)}^{\ul{H}_1}) \big)
	\\*
	&\cong \pi_0 \ul{\bS} \big( H_1, \rmB^n \rmS_e (\ul{\sfU(1)}^{\ul{H}_1}) \big)
	\\
	&\cong \pi_0 \ul{\bS} \big( H_1, \rmB^n \rmS_e \ul{\sfU(1)} \big)
	\\*
	&\cong \pi_0 \ul{\bS} \big( H_1, \rmB^n \sfU(1) \big)\,.
\end{align}
The first morphism is applying the functor $\rmS_e$.
For the second morphism we have used~\cite[Thm.~5.1]{Bun:Sm_Spaces}: for every manifold $Y \in \Mfd$, there is a canonical equivalence $\rmS_e \ul{Y} \simeq Y$ in $\bS$.
Further, here we have used that $\rmS_e$ commutes with $\rmB$ (Proposition~\ref{st:preservation of Grp and B}).
For the third morphism, we have used that the inclusion $\ul{\sfU(1)} \hookrightarrow \ul{\sfU(1)}^{\ul{H}_1}$ is an $I$-local equivalence in $\bH_\infty$:
since $H_1$ is connected and simply connected, this morphism is a smooth homotopy equivalence by~\cite[Lemma~8.9]{BMS:Sm2Grp}, and by~\cite[Cor.~3.16]{Bun:Sm_Spaces} any smooth homotopy equivalence is an $I$-local equivalence.
The last morphism again uses~\cite[Thm.~5.1]{Bun:Sm_Spaces}.
Since $\rmS_e$ preserves finite products, the equivalence $\rmS_e \ul{\sfU(1)}{}_1 \simeq \sfU(1)_1$ in $\bS$ is compatible with the group structure%
\footnote{This can also be seen directly: for any manifold $M$, the comparison map $\rmS_e M \to \Sing(M)$ sends a smooth map $\Delta_e^k \to M$ to the restriction $|\Delta^k| \to M$---see~\cite[Secs.~4, 5]{Bun:Sm_Spaces} for details.}
on $\sfU(1)$.

We can describe the map~\eqref{eq:Buns in H to buns in S} more explicitly as follows:
we have already seen above that any element in \smash{$\pi_0 \ul{\bH}_\infty (\ul{H}_1, \rmB^n(\ul{\sfU(1)}^{\ul{H}_1}))$} can be described as a smooth $\sfU(1)^{H_1}$-valued \v{C}ech cocycle with respect to a (differentiably good) open cover $\scV$ of $H_1$.
Under the map~\eqref{eq:Buns in H to buns in S}, these data are sent first to the same \v{C}ech cocycle, but seen as a map of spaces, and then this resulting \v{C}ech cocycle is composed with the evaluation $\sfU(1)^{H_1}_1 \to \sfU(1)_1$ at the unit element in $H$.
Therefore, using the canonical isomorphism $\pi_0 \ul{\bS} ( H_1, \rmB^n \sfU(1)) \cong \check{\rmH}^n(H_1; \sfU(1))$ and combining this with the maps~\eqref{eq:Buns in H and Cech coho} and~\eqref{eq:Buns in H to buns in S} we obtain a commutative diagram of abelian groups
\begin{equation}
\label{eq:S_e, U(1)^H and Cech coho}
\begin{tikzcd}[column sep=1.25cm, row sep=0.8cm]
	\pi_0 \ul{\bH}_\infty( \ul{H}_1, \rmB^n (\ul{\sfU(1)}^{\ul{H}_1}) \big)
	\ar[d] \ar[r]
	& \pi_0 \ul{\bS} \big( H_1, \rmB^n \rmS_e (\ul{\sfU(1)}^{\ul{H}_1}) \big)
	\ar[d, "(\rmB^n \rmS_e(\ev_e))_*"]
	\\
	\check{\rmH}^n \big( H_1; \ul{\sfU(1)}^{\ul{H}_1} \big) \ar[r, "(\ev_e)_*"]
	& \check{\rmH}^n \big( H_1; \sfU(1) \big)
\end{tikzcd}
\end{equation}
In this diagram, the left-hand vertical morphism is invertible as argued before~\eqref{eq:pi_0H_oo and Cech coho}.
The bottom morphism is an isomorphism by~\cite[Prop.~8.11]{BMS:Sm2Grp} and the fact that \v{C}ech cohomology and abelian sheaf cohomology are isomorphic on manifolds.
The right-hand vertical morphism is invertible as a consequence of the isomorphisms in~\eqref{eq:Buns in H and Cech coho} and the fact that the map \smash{$\ev_{e,1} \colon \ul{\sfU(1)}{}^{\ul{H}_1}_1 \to \ul{\sfU(1)}{}_1$} is an $I$-local equivalence.

Combining diagram~\eqref{eq:S_e, U(1)^H and Cech coho} with Proposition~\ref{st:classifying morphisms under L} and Lemma~\ref{st:I-loc eq NHLB^H to BU(1)}, we obtain that the class in $\rmH^3(H_1; \ZN) \cong \check{\rmH}^2(H_1, \sfU(1))$ defined by the $N\HLB^{H_1}$-principal $\infty$-bundle
\begin{equation}
\label{eq:Sym(G) as smooth U(1)^H bundle}
	\big(N \Sym(\CG)_1 \big) \dslash N \HLB^{H_1}
	\longrightarrow \ul{H}_1
\end{equation}
in $\bH_\infty$ agrees with the class defined by the principal $\infty$-bundle
\begin{equation}
	\big( \rmS_e N \Sym(\CG)_1 \big) \dslash \big( \rmS_e N \HLB^{H_1} \big)
	\longrightarrow \rmS_e \ul{H}_1 \simeq H_1
\end{equation}
in $\bS$.
Here we have used that there is an equivalence \smash{$N \HLB^{H_1} \simeq \rmB \ul{\sfU(1)}^{\ul{H}_1}$} in $\Grp(\bH_\infty)$, so that
\begin{equation}
	\pi_0 \ul{\bH}_\infty \big( \ul{H}_1, \rmB\, N\HLB^{H_1} \big)
	\simeq \pi_0 \ul{\bH}_\infty \big( \ul{H}_1, \rmB^2 \ul{\sfU(1)}^{\ul{H}_1} \big) \,.
\end{equation}
(Again, one can alternatively see the coincidence of the cohomology classes more explicitly on the level of \v{C}ech cocycles in the presentation of $\bH_\infty$ by the simplicial model category $\scH_\infty^p$:
a smooth bundle represented by a smooth $\sfU(1)^{H_1}$-valued cocycle on $H$ gets sent to the topological bundle represented by the same \v{C}ech cocycle interpreted as a collection of continuous maps.)
It thus remains to compute the cohomology class associated to these bundles.
In~\cite[Sec.~8]{BMS:Sm2Grp} it has been shown that the class in $\rmH^3(H_1; \ZN)$ of the bundle~\eqref{eq:Sym(G) as smooth U(1)^H bundle} agrees with the class in $\rmH^3(H_1; \ZN)$ that classifies the gerbe $\CG$ under the isomorphism~\eqref{eq:DD class}.
Since we started our construction from a so-called \emph{basic} gerbe, i.e.~one whose Dixmier-Douady class is a generator of $\rmH^3(H_1; \ZN)$, this completes the proof of Theorem~\ref{st:String model from Sym(G)}.

\begin{remark}
We conclude with the following remarks:
\begin{myenumerate}
\item In~\cite{BMS:Sm2Grp}, we suggested the smooth 2-group extension~\eqref{eq:Sym_G String sequence} as a model for the string group extension of $(H_1, \cdot, e_H)$.
However, the necessary formalism to make this precise was not available then---its development was the main goal of the present paper.

\item Moreover, in~\cite[Sec.~5.5]{BMS:Sm2Grp} we also presented a second smooth 2-group extension
\begin{equation}
\label{eq:Des_L extension}
\begin{tikzcd}
	\HLB^{H_1} \ar[r, "i"]
	& \Des_\sfL \ar[r, "p"]
	& \textint \ul{H}
\end{tikzcd}
\end{equation}
of $(H_1, \cdot, e_H)$;
its construction uses a connection on $\CG$ as auxiliary data and relies heavily on a notion of parallel transport on a gerbe $\CG$ with connection, as developed in~\cite{BMS:Sm2Grp}.
The extension~\eqref{eq:Des_L extension} is then obtained via a homotopy-coherent version of an associated bundle construction.
By~\cite[Thm.~5.33]{BMS:Sm2Grp}, there is an equivalence (in $\bH_{\leq 1}$) between the smooth 2-group extension in~\eqref{eq:Des_L extension} and~\eqref{eq:Sym 2Grp SES}, so that we automatically obtain an equivalence between the group objects in $\bH_\infty$ they induce under the nerve functor.
Hence, given the input of a basic gerbe $\CG$ on $H_1$, by Theorem~\ref{st:String model from Sym(G)} the extension~\eqref{eq:Des_L extension} also gives rise to a second, equivalent smooth string group extension
\begin{equation}
\begin{tikzcd}
	N \HLB^{H_1} \ar[r] & N \Des_\sfL \ar[r] & \ul{H}
\end{tikzcd}
\end{equation}
of $(H_1, \cdot, e_H)$, for any compact, simple and simply connected Lie group $(H_1, \cdot, e_H)$.
\qen
\end{myenumerate}
\end{remark}

\begin{appendix}

\section{Actions and category objects}
\label{app:Actions and CatObs}

In this appendix, we prove Theorem~\ref{st:grp actions are gpd obs}; that is, we show that group actions $P \dslash G$ in $\infty$-topoi (as in Definition~\ref{def:oo-group action}) are automatically groupoid objects.

\begin{definition}
\label{def:category object}
A \emph{category object} in an $\infty$-category $\scC$ is a simplicial object $X \in \Fun(\bbDelta^\opp, \scC)$ such that for every $n \in \NN_0$ the pullback $X_1 \times_{X_0} \cdots \times_{X_0} X_1$ exists in $\scC$ and the morphism
\begin{equation}
	X_n \longrightarrow X_1 \underset{X_0}{\times} \cdots \underset{X_0}{\times} X_1\,,
\end{equation}
induced by the spine decomposition $[n] \cong [1] \sqcup_{[0]} \cdots \sqcup_{[0]} [1]$ of finite ordered sets, is an equivalence.
\end{definition}

Suppose $\scC$ has a final object.
In analogy with Definition~\ref{def:Grp(H)}, a \emph{monoid object} in $\scC$ is a category object $M \in \Fun(\bbDelta^\opp, \scC)$ such that $M_0 \simeq *$ is a final object in $\scC$.
As for group objects, it follows that there are canonical natural equivalences $M_n \simeq M_1^{n-1}$.
Therefore, by Lemma~\ref{st:equiv on objs pulls back functoriality} we may assume, without loss of generality, that we have $M_n = M_1^{n-1}$ for any $n \in \NN_0$.
We set $M_1 \dslash M \coloneqq \Dec^0 M \in \Fun(\bbDelta^\opp, \scC)$.
Monoid objects can act on objects in their ambient $\infty$-category.
A monoid action is defined in the same way as a group action (Definition~\ref{def:oo-group action}), but for the reader's convenience, we spell out the definition (compare also~\cite[Def.~4.2.2.2, Rmk.~4.2.2.3]{Lurie:HA}):

\begin{definition}
\label{def:monoid action}
Let $\scC$ be an $\infty$-category with a final object, and let $M$ be a monoid object in $\scC$.
Let $P \in \scC$ be an object in $\scC$.
An \emph{action of $M$ on $P$} is a simplicial object $P \dslash M \in \Fun(\bbDelta^\opp, \scC)$ such that
\begin{myenumerate}
\item for each $n \in \NN_0$, we have $(P \dslash M)_n = P \times M_1^n$,

\item the morphism $d_1 \colon P \times M_1 \to P$ agrees with the projection onto $P$, the morphism $s_0 \colon P \to P \times M_1$ agrees with the morphism $1_P \times (* \to M_1)$, and

\item the collapse morphism $P \to *$ induces a morphism $P \dslash M \to M$ in $\Fun(\bbDelta^\opp, \scC)$.
\end{myenumerate}
\end{definition}

The trivial action of $M$ on the final object $*$ satisfies $* \dslash M = M$ as simplicial objects in $\scC$.
We set $a \coloneqq d_0 \colon P \times M_1 \to P$.
The pasting law for pullbacks implies that there are canonical equivalences of morphisms between $d_0 \colon P \times M_1^n \to P \times M_1^{n-1}$ and \smash{$a \times 1_{M_1^{n-1}} \colon P \times M_1^n \to P \times M_1^{n-1}$}, and similarly between $d_n \colon P \times M_1^n \to P \times M_1^{n-1}$ and the projection onto the first $n$ factors.

\begin{proposition}
\label{st:monoid actions are Cat obs}
Let $\scC$ be an $\infty$-category with pullbacks and a final object, let $M \in \Fun(\bbDelta^\opp, \scC)$ be a monoid object in $\scC$, and let $P \dslash M \in \Fun(\bbDelta^\opp, \scC)$ be an action of $M$ on an object $P \in \scC$.
Then, $P \dslash M$ is a category object in $\scC$.
\end{proposition}

\begin{proof}
Consider the diagram
\begin{equation}
\begin{tikzcd}[column sep={1.5cm,between origins}, row sep={1.5cm,between origins}]
	(P \dslash M)_1 \ar[dr, "d_0"'] & & (P \dslash M)_1 \ar[dl, "d_1"]
	\\
	& (P \dslash M)_0 &
\end{tikzcd}
	\quad = \quad
\begin{tikzcd}[column sep={1.5cm,between origins}, row sep={1.5cm,between origins}]
	P \times M_1 \ar[dr, "a"'] & & P \times M_1 \ar[dl, "\pr_0"]
	\\
	& P &
\end{tikzcd}
\end{equation}
We use the following notational convention:
let $I$ be a set, and consider a product $\prod_{i \in I} C_i$ of objects in $\scC$.
For a subset $J \subset I$, we let $\pr_J \colon \prod_{i \in I} C_i \to \prod_{j \in J} C_j$ denote the canonical projection.
If $J = \{i_0, \ldots, i_n\}$ is finite, we also write $\pr_{i_0 \ldots i_n}$ instead of $\pr_{\{ i_0, \ldots, i_n \}}$.

We can augment the above diagram to a diagram
\begin{equation}
\begin{tikzcd}[column sep=1.25cm, row sep=1cm]
	P \times M_1 \times M_1 \ar[d, "\pr_{01}"'] \ar[r, "a \times 1_{M_1}"]
	& P \times M_1 \ar[d, "\pr_0"] \ar[r, "\pr_1"]
	& M_1 \ar[d]
	\\
	P \times M_1 \ar[r, "a"']
	& P \ar[r]
	& *
\end{tikzcd}
\end{equation}
Here, the right and the outer rectangle are pullback diagrams, and hence the left square is a pullback diagram as well by the pasting law.
It follows that the canonical morphism
\begin{equation}
	(P \dslash M)_2 \longrightarrow (P \dslash M)_1 \underset{ (P \dslash M)_0}{\times}  (P \dslash M)_1
\end{equation}
is an equivalence in $\scC$.

We now proceed by induction:
suppose that the canonical morphism
\begin{equation}
	(P \dslash M)_k \longrightarrow (P \dslash M)_1 \underset{ (P \dslash M)_0}{\times} \cdots \underset{ (P \dslash M)_0}{\times}  (P \dslash M)_1
\end{equation}
is an equivalence, for each $2 \leq k \leq n$.
By this assumption, it now suffices to show that the morphism
\begin{equation}
\label{eq:Segal cond for monoid actions}
	(P \dslash M)_{n+1} \longrightarrow (P \dslash M)_n \underset{ (P \dslash M)_0}{\times}  (P \dslash M)_1
\end{equation}
induced by the partition $[n+1] = [n] \sqcup_{[0]} [1]$ is an equivalence.
We again have an augmented diagram
\begin{equation}
\begin{tikzcd}[column sep=1.5cm, row sep=1cm]
	P \times M_1^{n+1} \ar[d, "\pr_{0 \ldots n}"'] \ar[r, "a^{(n)} \times 1_{M_1}"]
	& P \times M_1 \ar[d, "\pr_0"] \ar[r, "\pr_1"]
	& M_1 \ar[d]
	\\
	P \times M_1^n \ar[r, "a^{(n)}"']
	& P \ar[r]
	& *
\end{tikzcd}
\end{equation}
where the morphism $a^{(n)}$ is, up to canonical equivalence, the morphism
\begin{equation}
	a \circ (a \times 1_{M_1}) \circ \cdots \circ (a \times 1_{M_1^n}) \colon P \times M_1^n \to P\,.
\end{equation}
Again, the right-hand and outer squares in this diagram is a pullback squares.
It follows by the pasting law that the left-hand square is a pullback as well.
Since $(P \dslash M)_{n+1} = P \times M_1^{n+1}$, and the morphisms $P \times M_1^{n+1} \to P \times M_1^n$ and $P \times M_1^{n+1} \to P \times M_1$ in~\eqref{eq:Segal cond for monoid actions} are canonically equivalent to the morphisms induced from the partition $[n+1] = [n] \sqcup_{[0]} [1]$.
This completes the proof.
\end{proof}

We recall a criterion from (the proof of)~\cite[Prop.~1.1.8]{Lurie:Goodwillie} for when a category object is a groupoid object.
Given a simplicial object $X \in \Fun(\bbDelta^\opp, \scC)$ in an $\infty$-category $\scC$ and a simplicial set $K \in \sSet$, we define an object $X(K) \in \scC$ as the limit (if it exists) of the diagram
\begin{equation}
	N(\bbDelta_{/K})^\opp \longrightarrow N\bbDelta^\opp \overset{X}{\longrightarrow} \scC\,.
\end{equation}
The following can be found in the proof of~\cite[Prop.~1.1.8]{Lurie:Goodwillie}:

\begin{proposition}
\label{st:Grpd obs via horns}
Let $\scC$ be an $\infty$-category with finite limits.
A category object $X$ in $\scC$ is a groupoid object in $\scC$ if and only if the inclusion $\Lambda^2_0 \hookrightarrow \Delta^2$ induces an equivalence $X_2 \eq X(\Lambda^2_0)$.
\end{proposition}

Let $\scI$ be the span category, depicted as $\{0,1\} \leftarrow \{0\} \rightarrow \{0,2\}$.
Consider the functor \smash{$D \colon \scI \to \bbDelta_{/\Lambda^2_0}$}, which sends the object $\{0\} \in \scI$ to the tip inclusion $\Delta^{\{0\}} \hookrightarrow \Lambda^2_0$ and the object $\{0,i\}$ to the edge inclusion $\Delta^{\{0,i\}} \hookrightarrow \Lambda^2_0$, for $i = 0,2$.

\begin{lemma}
\label{st:gpd obs via pullback}
Let $D \colon \scI \to \bbDelta_{/\Lambda^2_0}$ be defined as above, and let $\scC$ be an $\infty$-category with finite limits.
The following statements hold true:
\begin{myenumerate}
\item The functor \smash{$D \colon \scI \to \bbDelta_{/\Lambda^2_0}$} is cofinal.

\item For any $X \in \Fun(\bbDelta^\opp, \scC)$, the diagram
\begin{equation}
\begin{tikzcd}
	X(\Lambda^2_0) \ar[r, "\iota_{0,1}^*"] \ar[d, "\iota_{0,2}^*"']
	& X_1 \ar[d, "d_1"]
	\\
	X_1 \ar[r, "d_1"']
	& X_0
\end{tikzcd}
\end{equation}
is a pullback diagram in $\scC$, where $\iota_{0,i}^*$ denotes the morphism $X(\Delta^{\{0,i\}} \hookrightarrow \Lambda^2_0)$.

\item A category object $X \in \Fun(\bbDelta^\opp,\scC)$ in $\scC$ is a groupoid object precisely if the diagram
\begin{equation}
\begin{tikzcd}
	X_2 \ar[r, "d_2"] \ar[d, "d_1"']
	& X_1 \ar[d, "d_1"]
	\\
	X_1 \ar[r, "d_1"']
	& X_0
\end{tikzcd}
\end{equation}
is a pullback diagram in $\scC$.
\end{myenumerate}
\end{lemma}

\begin{proof}
For claim~(1), note that an object of \smash{$\bbDelta_{/\Lambda^2_0}$} is a pair $([n], \varphi)$ of an object $[n] \in \bbDelta$ and a morphism of simplicial sets $\varphi \colon \Delta^n \to \Lambda^2_0$.
We show that, for each object $([n], \varphi)$ of \smash{$\bbDelta_{/\Lambda^2_0}$}, the slice category $([n], \varphi)_{/D}$ is contractible.

Since the horn $\Lambda^2_0$ fits into a pushout diagram
\begin{equation}
\begin{tikzcd}
	\Delta^{\{0\}} \ar[r, hookrightarrow, "d^1"] \ar[d, hookrightarrow, "d^1"']
	& \Delta^{\{0,1\}} \ar[d, hookrightarrow]
	\\
	\Delta^{\{0,2\}} \ar[r, hookrightarrow]
	& \Lambda^2_0
\end{tikzcd}
\end{equation}
in $\sSet$, the morphism $\varphi \colon \Delta^n \to \Lambda^2_0$ is either the constant map at the apex of the horn, i.e.~$\varphi$ factors as $\varphi \colon \Delta^n \to \Delta^{\{0\}} \hookrightarrow \Lambda^2_0$, or it factors through a unique map $\Delta^n \to \Delta^{\{0,i\}}$, for $i = 0$ or $i = 2$, but not through the apex $\Delta^{\{0\}} \hookrightarrow \Lambda^2_0$.
(One can see this either by writing $\Delta^n = N[n]$ and $\Lambda^2_0 = N \scI$ and using that the nerve functor is fully faithful, or by using that $\sSet(\Delta^n, -) = (-)_n$ preserves colimits.)

In the first case, the slice category $([n], \varphi)_{/D}$ is the category describing spans; in other words, it is isomorphic to $\scI$, and we have $|N\scI| \cong |\Delta^1 \sqcup_{\Delta^0} \Delta^1| \cong |\Delta^1| \sqcup_{|\Delta^0|} |\Delta^1| \simeq *$.
In the other cases, the slice category $([n], \varphi)_{/D}$ is the final category, and hence contractible as well.

Claim~(2) now follows directly from the definition of $X(K)$, for $K \in \sSet$, together with part~(1) (after taking opposites), and claim~(3) then follows by combining claim~(2) with Proposition~\ref{st:Grpd obs via horns}.
\end{proof}

\begin{lemma}
\label{st:(co)lims of const diags}
Let $K$ be a simplicial set, let $\scC$ be an $\infty$-category, and let $C \in \scC$ be an object.
Let $\sfc \colon \scC \to \Fun(K,\scC)$ denote the constant-diagram functor.
\begin{myenumerate}
\item If $K$ is contractible, i.e.~$K \simeq *$ in $\sSet$ with the Kan-Quillen model structure, and $\colim^\scC_K(\sfc C)$ exists in $\scC$, then the canonical morphism $\colim^\scC_K(\sfc C) \to C$ in $\scC$ is an equivalence.

\item Dually, if $K$ is contractible and $\lim^\scC_K(\sfc C)$ exists in $\scC$, then the canonical morphism $C \to \lim^\scC_K(\sfc C) $ in $\scC$ is an equivalence.
\end{myenumerate}
\end{lemma}

\begin{proof}
By the definition of $\sfc$, there is a commutative diagram
\begin{equation}
\begin{tikzcd}
	K \ar[d, "coll"'] \ar[r, "\sfc C"] & \scC
	\\
	* \ar[ur, "C"'] &
\end{tikzcd}
\end{equation}
in $\bS$.
By~\cite[Cor.~4.1.2.6, Thm.~4.1.3.1]{Lurie:HTT}, the morphism $coll$ is cofinal if and only if the simplicial set $K \times * \cong K$ is contractible, i.e.~precisely if $coll \colon K \to *$ is an equivalence in $\sSet$ (in the Kan-Quillen model structure).
The first claim then follows from the fact that cofinal morphisms preserve colimits~\cite[Prop.~4.1.1.8]{Lurie:HTT}.
The second statement follows by duality.
\end{proof}

\begin{example}
\label{eg:contractible Cats for const diags}
We need the following two specific cases in which Lemma~\ref{st:(co)lims of const diags} applies:
\begin{myenumerate}
\item The nerve $N \scI \in \sSet$ is contractible, as already seen in the proof of Lemma~\ref{st:gpd obs via pullback}.

\item The inclusion $\{[0]\} \hookrightarrow \bbDelta$ is the inclusion of a final object.
Thus, the nerve $N \bbDelta \in \sSet$ is contractible in $\sSet$.
\qen
\end{myenumerate}
\end{example}

\begin{lemma}
\label{st:c is ff and Px(-) pres contr lims}
Let $K \in \sSet$ be contractible (in the Kan-Quillen model structure) and let $\scC$ be an $\infty$-category admitting limits of shape $K$.
Let $P \in \scC$ be any object.
\begin{myenumerate}
\item The constant diagram functor $\sfc \colon \scC \to \Fun(K,\scC)$ is fully faithful.

\item If $\scC$ admits finite products, then the functor $P \times (-) \colon \scC \to \scC$ preserves limits of shape $K$.
\end{myenumerate}
\end{lemma}

\begin{proof}
Let $C,D \in \scC$ be any objects.
To see~(1), we use the adjunction $\sfc \dashv \lim^\scC_K$ and Lemma~\ref{st:(co)lims of const diags}, which yield canonical equivalences
\begin{equation}
	\ul{\scC^K}(\sfc C, \sfc D)
	\simeq \ul{\scC}(C, \lim^\scC_K\, \sfc D)
	\simeq \ul{\scC}(C,D)\,.
\end{equation}

For claim~(2), let $C, P \in \scC$ be objects, and let $F \colon K \to \scC$ be a diagram.
We now have canonical equivalences
\begin{align}
	\ul{\scC} \big( C, \lim^\scC_K( \sfc P \times F) \big)
	&\simeq \ul{\scC^K} (\sfc C, \sfc P \times F)
	\\*
	&\simeq \ul{\scC^K} (\sfc C, \sfc P) \times \ul{\scC^K} (\sfc C, F)
	\\
	&\simeq \ul{\scC} (C, P) \times \ul{\scC} (C, \lim^\scC_K F)
	\\*
	&\simeq \ul{\scC} (C, P \times \lim^\scC_K F)\,.
\end{align}
In the third equivalence we have used part~(1), i.e.~that $\sfc$ is fully faithful here.
Then, the statement follows from the Yoneda Lemma.
\end{proof}

We can now prove Theorem~\ref{st:grp actions are gpd obs}:

\begin{proof}[Proof of Theorem~\ref{st:grp actions are gpd obs}]
Since every group object in $\scC$ is in particular a monoid object in $\scC$, it follows from Proposition~\ref{st:monoid actions are Cat obs} that $P \dslash G$ is a category object in $\scC$.
We now use Lemma~\ref{st:gpd obs via pullback} to show that it is even a groupoid object.
By that lemma, it suffices to check that the diagram
\begin{equation}
\label{eq:Lambda^2_0 in P//G}
\begin{tikzcd}[column sep=1.25cm, row sep=1cm]
	(P \dslash G)_2 \ar[r, "d_2"] \ar[d, "d_1"']
	& (P \dslash G)_1 \ar[d, "d_1"]
	\\
	(P \dslash G)_1 \ar[r, "d_1"']
	& (P \dslash G)_0
\end{tikzcd}
	 \quad = \quad
\begin{tikzcd}[column sep=1.5cm, row sep=1cm]
	P \times G_1^2 \ar[r, "d_2 = \pr_{01}"] \ar[d, "d_1"']
	& P \times G_1 \ar[d, "d_1 = \pr_0"]
	\\
	P \times G_1 \ar[r, "d_1 = \pr_0"']
	& P
\end{tikzcd}
\end{equation}
is a pullback diagram in $\scC$, where we have used axioms~(1) and~(2) of Definition~\ref{def:monoid action} and their consequences pointed out after Definition~\ref{def:oo-group action}.

Our goal now is to split off the factor $P$ in diagram~\eqref{eq:Lambda^2_0 in P//G}.
To that end, consider the diagram
\begin{equation}
\label{eq:diag d_1 = 1_P x d_1^G}
\begin{tikzcd}[row sep={1.25cm,between origins}, column sep={1.75cm,between origins}]
	& P \ar[rr, equal]
	& & P
	\\
	P \times G_1^2 \ar[ur, "\pr_0 \simeq d_1 d_2"] \ar[dr, "\pr_{12}"'] \ar[rr, "d_1" description]
	& & P \times G_1 \ar[ur, "\pr_0 = d_1"'] \ar[dr, "\pr_1"]
	&
	\\
	& G_1^2 \ar[rr, "d_1^G"']
	& & G_1
\end{tikzcd}
\end{equation}
The bottom rectangle in diagram~\eqref{eq:diag d_1 = 1_P x d_1^G} commutes by axiom~(3) of Definition~\ref{def:oo-group action}.
The top rectangle commutes because $P \dslash G$ is a simplicial object in $\scC$, so we have a canonical equivalence $d_1 d_2 \simeq d_1 d_1$.
This establishes the morphism $d_1$ as a product of morphisms in $\scC$:
it is induced by the morphisms $1_P \colon P \to P$ and $d_1^G \colon G_1^2 \to G_1$.
That is, there is a canonical equivalence
\begin{equation}
\label{eq:d_1 = 1_P x d_1^G}
	d_1 \simeq 1_P \times d_1^G\,,
\end{equation}
of morphisms $(P \dslash G)_2 \to (P \dslash G)_1$ in $\bH$.
We thus have an equivalence of diagrams
\begin{equation}
\begin{tikzcd}[column sep=1.5cm, row sep=1cm]
	P \times G_1^2 \ar[r, "d_2 = \pr_{01}"] \ar[d, "d_1 \simeq 1_P \times d_1^G"']
	& P \times G_1 \ar[d, "d_1 = \pr_0"]
	\\
	P \times G_1 \ar[r, "d_1 = \pr_0"']
	& P
\end{tikzcd}
	\quad \simeq \quad
	P \times \left(
\begin{tikzcd}[column sep=1.25cm, row sep=1cm]
	G_1^2 \ar[r, "d_2^G"] \ar[d, "d_1^G"']
	& G_1 \ar[d]
	\\
	G_1 \ar[r]
	& *
\end{tikzcd}
	\right)
\end{equation}
By the definition of $G$ as a groupoid object, the diagram 
\begin{equation}
\begin{tikzcd}[column sep=1.25cm, row sep=0.75cm]
	G_1^2 \ar[r, "d_2^G"] \ar[d, "d_1^G"']
	& G_1 \ar[d]
	\\
	G_1 \ar[r]
	& *
\end{tikzcd}
\end{equation}
is a pullback diagram in $\scC$.
It now follows from Lemma~\ref{st:c is ff and Px(-) pres contr lims} that the functor $P \times (-)$ sends this pullback diagram to a pullback diagram.
Consequently, the square~\eqref{eq:Lambda^2_0 in P//G} is a pullback diagram in $\bH$.
\end{proof}

\end{appendix}

%\vspace{1cm}
\begin{small}
\bibliographystyle{alphaurl}
\addcontentsline{toc}{section}{References}
\bibliography{SmString_Bib}

\vspace{0.5cm}

\noindent
%(Severin Bunk)
\noindent
Mathematical Institute, University of Oxford.
\\
severin.bunk@maths.ox.ac.uk

\end{small}

\end{document}